\documentclass[a4paper]{article}

\usepackage[all]{xy}\usepackage[latin1]{inputenc}       
\usepackage[dvips]{graphics,graphicx}
\usepackage{amsfonts,amssymb,amsmath,xcolor,mathrsfs, amstext}
\usepackage{amsbsy, amsopn, amscd, amsxtra, amsthm, authblk, enumerate,dsfont}
\usepackage{upref}
\usepackage{geometry}
\usepackage[displaymath]{lineno}
\usepackage{float}
\usepackage{bbm}
\usepackage{upgreek}

\allowdisplaybreaks

\usepackage[colorlinks,
            linkcolor=blue,
            anchorcolor=green,
            citecolor=blue
            ]{hyperref}

\numberwithin{equation}{section}

\newtheorem{theorem}{Theorem}[section]
\newtheorem{lemma}[theorem]{Lemma}
\newtheorem{proposition}[theorem]{Proposition}
\newtheorem*{proposition*}{Proposition}
\newtheorem{corollary}[theorem]{Corollary}
\newtheorem{conjecture}[theorem]{Conjecture}
\newtheorem*{corollary*}{Corollary}
\newtheorem{definition}[theorem]{Definition}
\newtheorem*{definitions*}{Definitions}

\newtheorem*{example*}{\bf Example}
\theoremstyle{remark}
\newtheorem{remark}[theorem]{\bf Remark}

\numberwithin{equation}{section}

\newcommand{\ol}{\overline}
\newcommand{\wt}{\widetilde}

\renewcommand{\1}{\mathds{1}}
\newcommand\dd{\mathop{}\!\mathrm{d}} % differential operator
\newcommand{\D}{\mathbb{D}} % unit disk
\newcommand{\PP}{\mathbb{P}} % probability
\newcommand{\E}{\mathbb{E}} % expectation
\newcommand{\SLE}{\mathrm{SLE}} % SLE
\newcommand{\CLE}{\mathrm{CLE}} % CLE
\newcommand{\BCLE}{\mathrm{BCLE}} % BCLE
\newcommand{\clockwise}{\circlearrowright} % clockwise
\newcommand{\counterclockwise}{\circlearrowleft} % counterclockwise
\newcommand{\CR}{\mathrm{CR}} % conformal radius
\newcommand{\QD}{\mathrm{QD}}% quantum disk
\newcommand{\GQD}{\mathrm{GQD}} %generalized QD
\newcommand{\QT}{\mathrm{QT}} % quantum triangle
\newcommand{\QA}{\mathrm{QA}} % quantum annulus

\newcommand{\bbC}{\mathbb{C}}
\newcommand{\bbD}{\mathbb{D}}
\newcommand{\bbE}{\mathbb{E}}
\newcommand{\bbH}{\mathbb{H}}

\newcommand{\bbR}{\mathbb{R}}

\newcommand{\bbZ}{\mathbb{Z}}

\newcommand{\cC}{\mathcal{C}}
\newcommand{\cT}{\mathcal{T}}
\newcommand{\cM}{\mathcal{M}}
\newcommand{\cD}{\mathcal{D}}

\newcommand{\cL}{\mathcal{L}}

\newcommand{\Leb}{\mathrm{Leb}}
\newcommand{\Md}{{\mathcal{M}}^\mathrm{disk}}
\newcommand{\Mfd}{{\mathcal{M}}^\mathrm{f.d.}}
\newcommand{\LF}{\mathrm{LF}}
\def\e{\varepsilon}

\newcommand{\eqb}{\begin{equation}}
\newcommand{\eqe}{\end{equation}}

\title{The bulk one-arm exponent for the CLE$_{\kappa'}$ percolations}
\author{}

\begin{document}
\author{Haoyu Liu \thanks{School of Mathematical Sciences, Peking University \href{lhy0629@pku.edu.cn}{lhy0629@pku.edu.cn}} \qquad Xin Sun \thanks{Beijing International Center for Mathematical Research, Peking University \href{xinsun@bicmr.pku.edu.cn}{xinsun@bicmr.pku.edu.cn}} \qquad Pu Yu \thanks{Courant Institute of Mathematical Sciences, New York University  \href{py628@nyu.edu}{py628@nyu.edu}} \qquad Zijie Zhuang \thanks{Wharton Statistics and Data Science Department, University of Pennsylvania \href{zijie123@wharton.upenn.edu}{zijie123@wharton.upenn.edu}}}

\maketitle

\begin{abstract}
The conformal loop ensemble (CLE) is a conformally invariant random collection of loops. In the non-simple regime $\kappa'\in (4,8)$, it describes the scaling limit of the critical Fortuin--Kasteleyn (FK) percolations. CLE percolations were introduced by Miller--Sheffield--Werner (2017). The CLE$_{\kappa'}$ percolations describe the scaling limit of a natural variant of the FK percolation called the fuzzy Potts model, which has an additional percolation parameter $r$. Based on CLE percolations and assuming the convergence of the FK percolation to CLE, K{\"o}hler-Schindler and Lehmkuehler (2022) derived all the arm exponents for the fuzzy Potts model except the bulk one-arm exponent. In this paper, we exactly solve this exponent, which prescribes the dimension of the clusters in CLE$_{\kappa'}$ percolations. As a special case, the bichromatic one-arm exponent for the critical 3-state Potts model should be $4/135$. To the best of our knowledge, this natural exponent was not predicted in physics. Our derivation relies on the iterative construction of CLE percolations from the boundary conformal loop ensemble (BCLE), and the coupling between Liouville quantum gravity and SLE curves. The source of the exact solvability comes from the structure constants of boundary Liouville conformal field theory.  A key technical step is to prove a conformal welding result for the target-invariant radial SLE curves. As intermediate steps in our derivation, we obtain several exact results for BCLE in both the simple and non-simple regimes, which extend results of Ang--Sun--Yu--Zhuang (2024) on the touching probability of non-simple CLE. This also provides an alternative derivation of the relation between the BCLE parameter $\rho$ and the additional percolation parameter $r$ in CLE percolations, which was originally due to Miller--Sheffield--Werner (2021, 2022). \\

\noindent\textbf{MSC classes:} 60D05, 60G60, 60J67
\end{abstract}

\setcounter{tocdepth}{1}
\tableofcontents

\section{Introduction}
Scaling limits of percolation-type models are a rich source of interesting random fractals. A basic quantity of interest for any such model is the fractal dimension of percolation clusters. For two-dimensional Bernoulli percolation, the cluster dimension is  {$91/48$}~\cite{smirnov2001critical}. More generally, the scaling limit of the critical Fortuin--Kasteleyn (FK) percolation with cluster weight $q \in (0,4]$ is supposed to be described by conformal loop ensemble (CLE) with a parameter $\kappa' \in [4,8)$ such that $\kappa'=4\pi/\arccos(-\sqrt{q}/2)$~\cite{SheffieldCLE}. The Bernoulli percolation case corresponds to $q=1$. The continuum counterpart of the critical FK$_q$ percolation cluster is the CLE$_{\kappa'}$ gasket ({the collection of points not surrounded by any loop of the CLE}), which has dimension $1+2/\kappa'+3\kappa'/32$~\cite{SSW09,MSW14}. The FK$_q$ percolation model  and the $q$-state Potts model are related by the Edwards--Sokal coupling~\cite{Edwards-Sokal}. In this coupling, it is natural to introduce an additional parameter $r$, which gives rise to the so-called fuzzy Potts model. By making sense of the continuum analog of the Edwards--Sokal coupling,  Miller, Sheffield, and Werner~\cite{MSW2017} introduced the CLE percolations, with a new parameter $\rho$ in addition to $\kappa'$. The paper~\cite{MSW2017} together with Liouville quantum gravity techniques developed later in~\cite{MSW21-nonsimple,MSW22-simple} determines the parameter relation between the fuzzy Potts model and the CLE percolation. 

The cluster dimension $d_C$ of a percolation model can be encoded in the bulk one-arm exponent $\alpha_1$. For the planar case, the relation is $d_C=2-\alpha_1$. Assuming the convergence of the critical FK$_q$ percolation to CLE$_{\kappa'}$, K{\"o}hler-Schindler and Lehmkuehler~\cite{KSL22} derived all the bulk and boundary arm exponents for the fuzzy Potts model except the bulk one-arm exponent. In this paper, we derive the exact value of this exponent. More precisely, we define the bulk one-arm exponent purely in terms of the CLE$_{\kappa'}$ percolation, and derive an exact formula for it in terms of the parameters $\kappa'$ and $\rho$; see Theorem~\ref{thm:one-arm}. Similar to the backbone exponent for percolation derived in~\cite{NQSZ23}, the fuzzy Potts one-arm exponent is expressed as the root of an elementary equation. 
Under the same convergence assumption as in~\cite{KSL22}, it can be shown that the bulk one-arm exponent for the corresponding fuzzy Potts model has the same value; see Theorem~\ref{thm:fuzzy-potts-one-arm}. As a special case, the bichromatic one-arm exponent for the critical 3-state Potts model should be $4/135$; see Corollary~\ref{cor:3-Potts-bichromatic}. 

The starting point of our derivation is the iterative construction of CLE percolations from the boundary conformal loop ensemble (BCLE) illustrated in~\cite{MSW2017,KSL22}. Using this construction, the bulk one-arm exponent can be encoded by the conformal radius distribution of BCLE loops, which we determine in Section~\ref{subsec:intro-BCLE}. Our method is based on the coupling between Liouville quantum gravity and SLE curves, which originates from~\cite{She16a,DMS21}. The core of our proof is a conformal welding result for BCLE loops, based on which the moment of its conformal radius may be extracted from the structure constants of boundary Liouville conformal field theory. See Section~\ref{subsec:lqg} for an overview of the proof.

\subsection{One-arm exponent for the CLE percolations and the fuzzy Potts model}\label{subsec:intro-one-arm}

We first recall the CLE percolations studied in~\cite{MSW2017}. For $\kappa' \in (4,8)$, sample a nested $\CLE_{\kappa'}$ on $\mathbb{D}$. For each loop in $\CLE_{\kappa'}$, its nesting level is defined as the number of distinct loops surrounding it plus 1. For instance, the outermost loops have nesting level 1. The loops can be separated into even and odd ones depending on their nesting levels.
We consider nested $\CLE_{\kappa'}$ as the scaling limit of FK percolations with free boundary condition, and so the CLE clusters correspond to the gasket squeezed inside an odd loop and outside of all the even loops that it surrounds. Fix $r \in (0,1)$ and independently color each CLE cluster in red with probability $r$, and in blue otherwise. For $\e>0$, let $\mathcal{A}_\epsilon$ be the event that there exists a sequence of neighboring blue clusters that connect $\e \mathbb{D}$ to $\partial \mathbb{D}$. The (blue) bulk one-arm exponent for the CLE percolation $\alpha_1(r)$ is defined by
\begin{equation}
    \label{eq:def-one-arm-cle}
    \mathbb{P}[\mathcal{A}_\e] = \e^{\alpha_1(r) + o(1)} \quad \mbox{as } \e \rightarrow 0.
\end{equation}
The following theorem shows the existence and provides the explicit value of $\alpha_1(r)$.
\begin{theorem}\label{thm:one-arm}
    Let $\kappa=16/\kappa' \in (2,4)$. The bulk one-arm exponent $\alpha_1(r)$ exists and is given by the unique positive solution in $(0,1-\frac{2}{\kappa'}-\frac{3\kappa'}{32})$ to the equation 
    \eqb\label{eq:one-arm}
    \frac{\sin(\frac{\pi(\kappa + 2\rho+8)}{4 \kappa} \sqrt{(4-\kappa)^2+8 \kappa x})}{\sin(\frac{\pi(\kappa - 2\rho - 8)}{4 \kappa} \sqrt{(4-\kappa)^2+8 \kappa x})} = \frac{\sin(\frac{\pi}{4}(\kappa+2\rho))}{\sin(\frac{\pi}{4}(\kappa-2\rho))},
    \eqe
    where  {$\rho$ is the unique number in $(-2,\kappa-4)$ such that $\tan(\frac{\pi(\rho+2)}{2}) = \frac{\sin(\pi \kappa/2)}{1 + \cos(\pi \kappa/2) - 1/(1-r)}$.} %$\rho = \frac{2}{\pi} \arctan \Big(\frac{\sin(\pi \kappa/2)}{1 + \cos(\pi \kappa/2) - 1/(1-r)} \Big) - 2 \in (-2, \kappa-4)$.
\end{theorem}

The fuzzy Potts model is the discrete analog of CLE percolations~\cite{MSW2017,KSL22}. It is obtained by first sampling a critical Fortuin--Kasteleyn (FK) percolation with cluster weight $q$, and then coloring the vertices of each open cluster independently in red or blue with probability $r$ and $1-r$, respectively. Write $\Lambda_n=[-n,n]^2 \cap \mathbb{Z}^2$ for the box of size $n$. Let $A_B(m,n)$ (resp.\ $A_R(m,n)$) be the event that there is a blue (resp.\ red) path in $\Lambda_n \setminus \Lambda_m$ from $\Lambda_m$ to $\Lambda_n$. The blue (bulk) one-arm exponent $\alpha_B(r)$ for the fuzzy Potts model is defined by
\[ \mathbb{P}[A_B(m,n)]=(m/n)^{\alpha_B(r)+o(1)} \quad \text{as } n/m \to \infty. \]
One can similarly define the red one-arm exponent $\alpha_R(r)$ and multiple-arm exponents $\alpha_\tau(r)$ for any color sequence $\tau \in \cup_{k \in \mathbb{Z}_+} \{R,B\}^k$.  {It is conjectured that  the critical FK percolation model converges to CLE in the scaling limit, which is referred to as the \emph{FK conformal invariance conjecture}; see Conjecture~\ref{conj:fk-to-cle} for details.} The existence of arm exponents for the fuzzy Potts model and the values of all \emph{polychromatic} ones are derived in~\cite{KSL22}  {under the FK conformal invariance conjecture}. It turns out that $\alpha_\tau(r)$ depends on $\tau$ only through the number of its adjacent pairs of different colors, as the existence of several arms of the same color only costs an additional probability of constant order.

The following theorem provides the missing one-arm exponent (and thus the monochromatic ones).

\begin{theorem}\label{thm:fuzzy-potts-one-arm}
    Let $q \in [1,4)$, and suppose that the FK conformal invariance conjecture (Conjecture~\ref{conj:fk-to-cle}) holds for the critical FK percolation with cluster weight $q$. Write $\kappa=4 \arccos(-\sqrt{q}/2)/\pi \in [8/3,4)$. For the fuzzy Potts model with red probability $r$, its blue (resp.\ red) bulk one-arm exponent $\alpha_B(r)$ (resp.\ $\alpha_R(r)$) is equal to $\alpha_1(r)$ (resp.\ $\alpha_1(1-r)$) in Theorem~\ref{thm:one-arm}.
\end{theorem}

\begin{remark}
    The conformal invariance conjecture for the FK-Ising model (i.e., $q=2$ and $\kappa'=16/3$) is known to hold~\cite{Smi10,KS19,KS16}, hence the previous theorem can be stated unconditionally.
\end{remark}

Theorem~\ref{thm:fuzzy-potts-one-arm} extends the classical results on the Hausdorff dimension of the CLE carpet/gasket. The CLE gasket describes the scaling limit of FK$_q$ percolation clusters, while the CLE carpet corresponds to the single-color cluster in the Potts model.
In Theorem~\ref{thm:fuzzy-potts-one-arm}, the case $1-r \to 0$ (i.e., $\rho \to -2$) is related to the critical FK percolation, where $\lim_{1-r \to 0} \alpha_1(r) = 1-2/\kappa'-3\kappa'/32$ gives the one-arm exponent for critical FK$_q$ percolation. Indeed, this matches with $\CLE_{\kappa'}$ gasket dimension derived in~\cite{SSW09,MSW14}. The case $1-r=t/q$ for $t \in \{1,2,\ldots,q-1\}$ is related to the critical Potts model. In particular, for $q=2$ (i.e., $\kappa=3$) or $q=3$ (i.e., $\kappa=10/3$), and $r=(q-1)/q$ (i.e., $\rho=-\frac{\kappa}{2}$), the one-arm exponents for the critical Ising model (i.e., 2-state Potts model) and the critical 3-state Potts model should be $5/96$ and $7/80$, which agrees with $\CLE_{\kappa}$ carpet dimension~\cite{SSW09,NW11}.

Theorem~\ref{thm:fuzzy-potts-one-arm} has an interesting corollary on the 3-state Potts model. 
For $q=3$ and $r=1/3$ (i.e., $\kappa=10/3$ and $\rho=-1$), the one-arm exponent becomes the bichromatic one-arm exponent for the critical 3-state Potts model. Here we view a state as a color, and by ``bichromatic'' we mean that this path is allowed to consist of vertices of two colors out of three, see Figure~\ref{fig:3-Potts} for an illustration and Section~\ref{subsec:bichromatic} for more details.
To the best of our knowledge, the following result was not predicted in physics.
\begin{corollary}\label{cor:3-Potts-bichromatic}
    Assuming that the critical $\mathrm{FK}_3$ percolation clusters converge in distribution to $\CLE_{24/5}$ gasket, then the bichromatic one-arm exponent for the critical 3-state Potts model is $4/135$.
\end{corollary}

\begin{figure}[h]
    \centering
    \begin{tabular}{ccc}
         \includegraphics[scale=0.25]{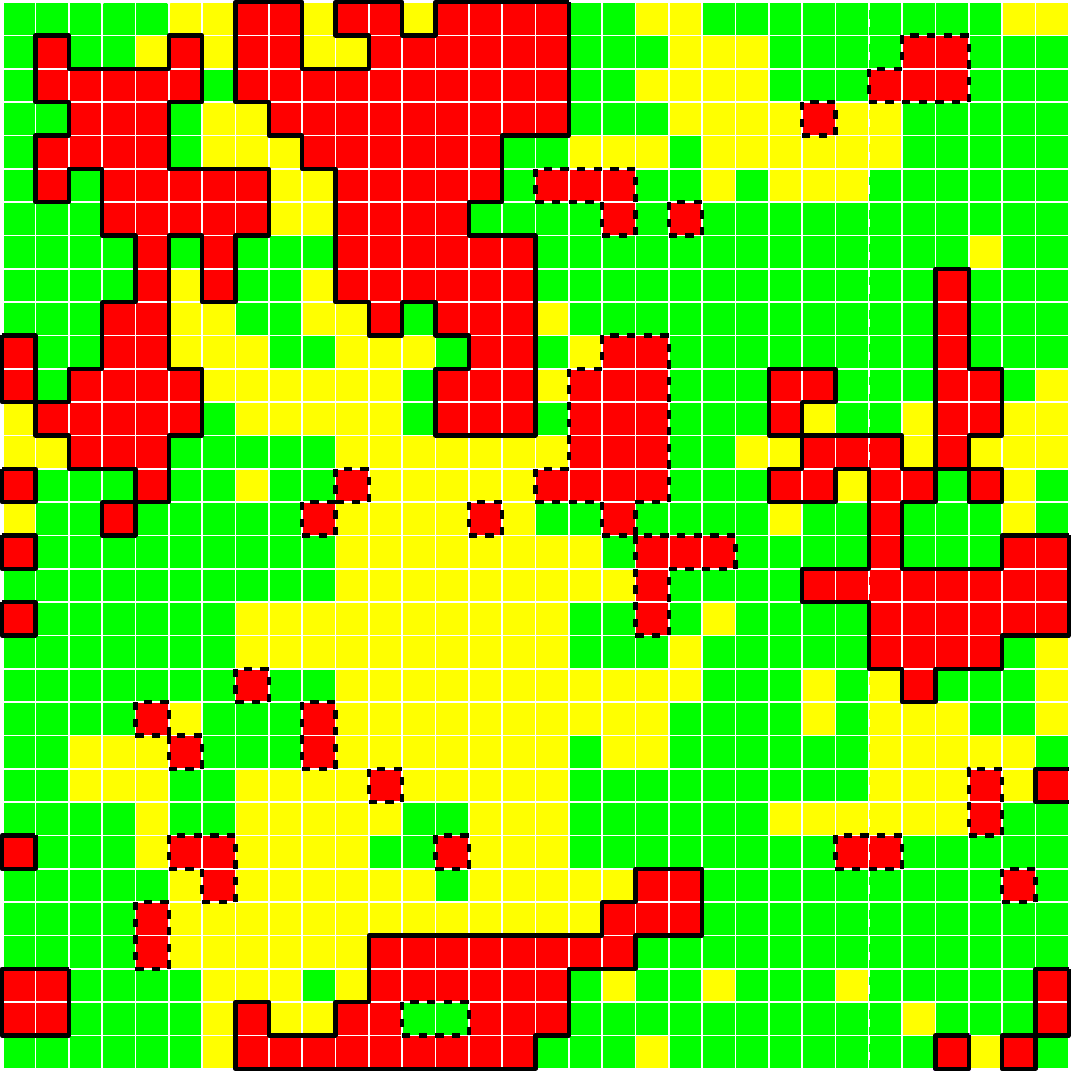} \ \ \ & \includegraphics[scale=0.4]{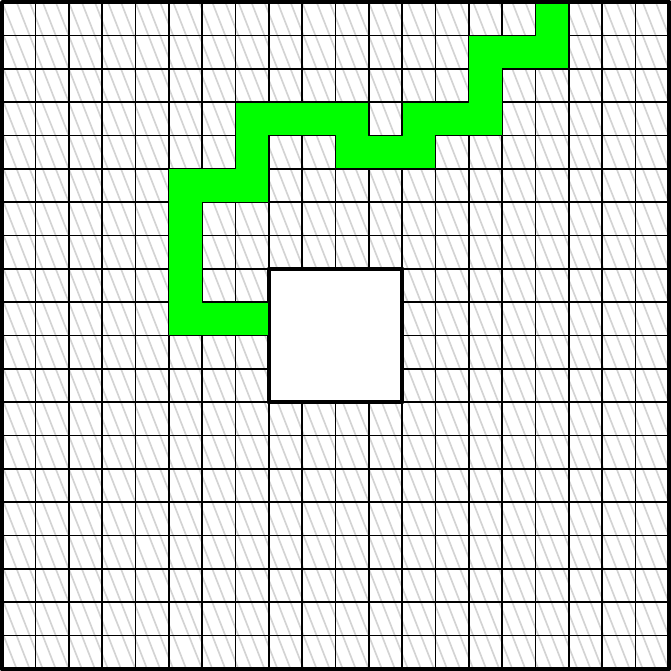} \ \ \ & \includegraphics[scale=0.4]{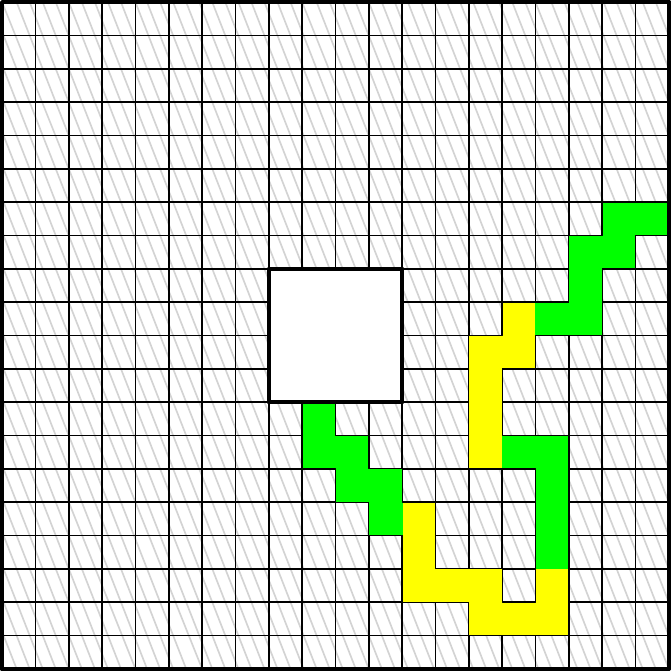}
    \end{tabular}
    \caption{\textbf{Left:} A critical 3-state Potts configuration in the box $\Lambda_n$ drawn on the dual lattice, where vertices assigned with spin $1,2,3$ are colored in red, yellow and green, respectively. If we treat yellow and green vertices as the same color (say, blue), it becomes the fuzzy Potts configuration with $q=3$ and $r=1/3$. Shown in solid (resp.\ dashed) black are the boundary-touching (resp.\ non-boundary touching) red/blue interfaces, and the former should converge to $\BCLE_{10/3}(-1)$. \textbf{Middle:} The green one-arm event inside $\Lambda_{m,n}$ requests the existence of a green path from $\partial \Lambda_m$ to $\partial \Lambda_n$. The one-arm exponent is equal to $7/80$. \textbf{Right:} The bichromatic one-arm event inside $\Lambda_{m,n}$ requests the existence of a path consisting of green or yellow vertices. Corollary~\ref{cor:3-Potts-bichromatic} shows that the bichromatic one-arm exponent is $4/135$.}
    \label{fig:3-Potts}
\end{figure}

\subsection{Boundary CLE: winding orientation and conformal radius distribution}
\label{subsec:intro-BCLE}

Our starting point for proving Theorem~\ref{thm:one-arm} is the iterative construction of CLE percolation interfaces via BCLE in~\cite{MSW2017}, which we will recall in Section~\ref{subsec:CLE-prelim}. For this construction the one-arm exponent can be extracted from the law of certain conformal radii associated with BCLE. Here we obtain the exact formulae for the moments of these conformal radii, which are reminiscent of the formula in~\cite{SSW09} for the ordinary CLE loops. We find these formulae of independent interest. 

Let $\D$ be the unit disk. For $\kappa \in (2,4)$ and $\rho \in (-2, \kappa - 4)$, let $\BCLE_{\kappa}^\clockwise(\rho)$ be the boundary conformal loop ensemble  {in $\bbD$} defined using a chordal $\SLE_\kappa(\rho;\kappa-6-\rho)$ following the notation from~\cite{MSW2017}. By convention, we use the superscript $\clockwise$ to indicate that each loop is oriented clockwise. The outer boundary of each region not surrounded by a clockwise loop can be seen as a counterclockwise (false) loop. The following theorem gives the probability that a given point is surrounded by a clockwise or a counterclockwise loop.

\begin{theorem}\label{thm:touch-BCLE-simple}
    For $\kappa \in (2,4)$ and $\rho \in (-2, \kappa - 4)$, let $\{ 0 \in \BCLE_{\kappa}^\clockwise(\rho) \}$ (resp.\ $\{ 0 \not \in \BCLE_{\kappa}^\clockwise(\rho) \}$) denote the event that the origin is surrounded by a clockwise true (resp.\ counterclockwise false) loop in $\BCLE_{\kappa}^\clockwise(\rho)$. Then we have
    \begin{align}
        &\PP[0 \in \BCLE_{\kappa}^\clockwise(\rho)] = \frac{\sin(\frac{2\pi}{\kappa}(\kappa-\rho-4)) \sin(\frac{\pi(4-\kappa)}{4\kappa}(\kappa-2\rho-4))}{\sin(\frac{\pi(4-\kappa)}{\kappa}) \sin(\frac{\pi}{4}(\kappa - 2\rho-4))}, \label{eq:touch-BCLE-simple-A} \\
        &\PP[0 \notin \BCLE_{\kappa}^\clockwise(\rho)] = \frac{\sin(\frac{2\pi}{\kappa}(\rho+2)) \sin(\frac{\pi(4-\kappa)}{4\kappa}(2\rho+8-\kappa))}{\sin(\frac{\pi(4-\kappa)}{\kappa}) \sin(\frac{\pi}{4}(\kappa - 2\rho-4))}.\label{eq:touch-BCLE-simple-B}
    \end{align}
\end{theorem}

Although not obvious, one can check that the expressions on the right hand side of~\eqref{eq:touch-BCLE-simple-A} and~\eqref{eq:touch-BCLE-simple-B} indeed sum up to 1. To prove Theorem~\ref{thm:touch-BCLE-simple}, we actually prove the following stronger statement (Theorem~\ref{thm:CR-BCLE-simple}). For a simply connected domain $D \subset \mathbb{C}$ and $z \in D$, let $f:\D \rightarrow D$ be a conformal map with $f(0)=z$. The \emph{conformal radius} of $D$ seen from $z$ is defined as ${\rm CR}(z,D):=|f'(0)|$. Let $\mathcal{L}$ be the loop in $\BCLE_{\kappa}(\rho)$ surrounding the origin which can be either clockwise or counterclockwise, and let 
$D_{\mathcal{L}}$ be the connected component of $\mathbb{D} \setminus \mathcal{L}$ that contains the origin.
The following theorem gives the moment of ${\rm CR}(0,D_{\mathcal{L}})$ restricted to the event that 0 is surrounded by a clockwise loop or counterclockwise loop. Theorem~\ref{thm:touch-BCLE-simple} can be obtained from it by setting $\lambda = 0$.

\begin{theorem}\label{thm:CR-BCLE-simple}
   Fix $\kappa \in (2,4)$ and $\rho \in (-2, \kappa - 4)$. Let $\lambda>\frac{\kappa}{8}-1$ and $\theta = \frac{\pi}{4}\sqrt{(4-\kappa)^2-8\kappa\lambda}$. Then
    \begin{align}
       &\E[\CR(0, D_{\mathcal{L}})^{\lambda} \1_{0 \in \BCLE_{\kappa}^\clockwise(\rho)}] = \frac{\sin(\frac{\pi(4-\kappa)}{4}) \sin(\frac{2\pi}{\kappa}(\kappa-\rho-4))}{\sin(\frac{\pi(4-\kappa)}{\kappa}) \sin(\frac{\pi}{4}(\kappa - 2\rho-4))} \cdot \frac{\sin(  \frac{\kappa - 2\rho-4}{\kappa}\theta )}{\sin(\theta)}, \label{eq:CR-BCLE-simple-A}\\
      & \E[\CR(0, D_{\mathcal{L}})^{\lambda} \1_{0 \notin \BCLE_{\kappa}^\clockwise(\rho)}] = \frac{\sin(\frac{\pi(4-\kappa)}{4}) \sin(\frac{2\pi}{\kappa}(\rho+2))}{\sin(\frac{\pi(4-\kappa)}{\kappa}) \sin(\frac{\pi}{4}(\kappa - 2\rho-4))} \cdot \frac{\sin(  \frac{2 \rho + 8 - \kappa}{\kappa}\theta)}{\sin(\theta)}.\label{eq:CR-BCLE-simple-B}
   \end{align}
  {Here for $a<0$, we choose the branch cut $\sqrt{a} = i\sqrt{-a}$ and use the convention $\sin(\sqrt{a}) = \frac{e^{-\sqrt{-a}}-e^{\sqrt{-a}}}{2i}$.}  Moreover, if $\lambda\leq\frac{\kappa}{8}-1$, then the left hand sides of~\eqref{eq:CR-BCLE-simple-A} and~\eqref{eq:CR-BCLE-simple-B} are infinite.
\end{theorem}

For $\kappa' \in (4,8)$ and $\rho' \in (\frac{\kappa'}{2}-4, \frac{\kappa'}{2}-2)$, let $\BCLE_{\kappa'}^\clockwise(\rho')$ be the boundary conformal loop ensemble defined using a chordal $\SLE_{\kappa'}(\rho'; \kappa' - 6 -\rho')$. Let $\mathcal{L}'$ be the loop in $\BCLE_{\kappa'}(\rho')$ surrounding the origin and $D_{\mathcal{L}'}$ be the connected component of $\mathbb{D} \setminus \mathcal{L}'$ that contains the origin. Theorems~\ref{thm:touch-BCLE-nonsimple} and~\ref{thm:CR-BCLE-nonsimple} provide similar results in this regime.

\begin{theorem}\label{thm:touch-BCLE-nonsimple}
    For $\kappa' \in (4,8)$ and $\rho' \in (\frac{\kappa'}{2}-4, \frac{\kappa'}{2}-2)$, let $\{ 0 \in \BCLE_{\kappa'}^\clockwise(\rho') \}$ (resp.\ $\{ 0 \not \in \BCLE_{\kappa'}^\clockwise(\rho') \}$) denote the event that the origin is surrounded by a clockwise true (resp.\ counterclockwise false) loop in $\BCLE_{\kappa'}^\clockwise(\rho')$. Then, we have
    \begin{align}
        &\PP[0 \in \BCLE_{\kappa'}^\clockwise(\rho')] = \frac{\sin(\frac{2\pi}{\kappa'}(\kappa' - \rho' - 4)) \sin(\frac{\pi(\kappa'-4)}{4 \kappa'}(\kappa' - 2\rho' - 4))}{\sin(\frac{\pi(\kappa'-4)}{\kappa'}) \sin(\frac{\pi}{4}(\kappa'-2\rho' - 4))}, \label{eq:touch-BCLE-nonsimple-A} \\
        &\PP[0 \notin \BCLE_{\kappa'}^\clockwise(\rho')] = \frac{\sin(\frac{2\pi}{\kappa'}(\rho' + 2)) \sin(\frac{\pi(\kappa'-4)}{4 \kappa'}(2\rho' + 8 - \kappa'))}{\sin(\frac{\pi(\kappa'-4)}{\kappa'}) \sin(\frac{\pi}{4}(\kappa'-2\rho' - 4))}. \label{eq:touch-BCLE-nonsimple-B}
    \end{align}
\end{theorem}

\begin{theorem}\label{thm:CR-BCLE-nonsimple}
   Fix $\kappa' \in (4,8)$ and $\rho' \in (\frac{\kappa'}{2}-4, \frac{\kappa'}{2}-2)$. Let $\lambda'>\frac{\kappa'}{8}-1$ and $\theta' = \frac{\pi}{4}\sqrt{(4-\kappa')^2-8\kappa'\lambda'}$. Then
    \begin{align}
       &\E[\CR(0, D_{\mathcal{L}'})^{\lambda'} \1_{0 \in \BCLE_{\kappa'}^\clockwise(\rho')}] = \frac{\sin(\frac{\pi(\kappa'-4)}{4}) \sin(\frac{2\pi}{\kappa'}(\kappa' - \rho' - 4))}{\sin(\frac{\pi(\kappa'-4)}{\kappa'}) \sin(\frac{\pi}{4}(\kappa'-2\rho' - 4))} \cdot \frac{\sin( \frac{\kappa'-2\rho'-4}{\kappa'}\theta')}{\sin(\theta')}, \label{eq:CR-BCLE-nonsimple-A}\\
       &\E[\CR(0, D_{\mathcal{L}'})^{\lambda'} \1_{0 \notin \BCLE_{\kappa'}^\clockwise(\rho')}] = \frac{\sin(\frac{\pi(\kappa'-4)}{4}) \sin(\frac{2\pi}{\kappa'}(\rho' + 2))}{\sin(\frac{\pi(\kappa'-4)}{\kappa'}) \sin(\frac{\pi}{4}(\kappa'-2\rho' - 4))} \cdot \frac{\sin(  \frac{2 \rho'+8-\kappa'}{\kappa'}\theta')}{\sin(\theta')}, \label{eq:CR-BCLE-nonsimple-B}
    \end{align}
   {where we set $\sqrt{a}=i\sqrt{-a}$ for $a<0$ as in Theorem~\ref{thm:CR-BCLE-simple}.}       Moreover, if $\lambda'\leq\frac{\kappa'}{8}-1$, then the left hand sides of~\eqref{eq:CR-BCLE-nonsimple-A} and~\eqref{eq:CR-BCLE-nonsimple-B} are infinite.
\end{theorem}

Although~\eqref{eq:touch-BCLE-nonsimple-A} ---~\eqref{eq:CR-BCLE-nonsimple-B} look identical to~\eqref{eq:touch-BCLE-simple-A} ---~\eqref{eq:CR-BCLE-simple-B} after replacing $(\kappa,\rho)$ with $(\kappa',\rho')$, we insist on this separate formulation to emphasize that the ranges of admissible values are different.

\begin{remark}
    In each of Theorems~\ref{thm:touch-BCLE-simple}--\ref{thm:CR-BCLE-nonsimple}, the two equations are actually equivalent. Indeed, we can define $\BCLE_{\kappa}^\counterclockwise(\rho)$ in the same way as $\BCLE_{\kappa}^\clockwise(\rho)$ with clockwise orientation replaced by counterclockwise orientation. Then the collection of counterclockwise false loops of $\BCLE_{\kappa}^\clockwise(\rho)$  is $\BCLE_{\kappa}^\counterclockwise(\kappa-6-\rho)$. This gives the equivalence between~\eqref{eq:CR-BCLE-simple-A} and~\eqref{eq:CR-BCLE-simple-B}. The same applies to~\eqref{eq:CR-BCLE-nonsimple-A} and~\eqref{eq:CR-BCLE-nonsimple-B}.
    The case $\kappa'\in (4,8)$ and $\rho'=0$ where we have the ordinary $\CLE_{\kappa'}$ was previously obtained in~\cite{ASYZ24}, since $\BCLE_{\kappa'}(0)$ is exactly the collection of boundary-touching loops of $\CLE_{\kappa'}$.
\end{remark}

In Section~\ref{subsec:one-arm}, we explain how to derive Theorem~\ref{thm:one-arm} from Theorems~\ref{thm:touch-BCLE-simple}--\ref{thm:CR-BCLE-nonsimple}. 
A main outcome of~\cite{MSW21-nonsimple, MSW22-simple} is the explicit relation between the $\rho$-parameter in the BCLE and the $r$-parameter in the CLE percolation; see Equations (7.6) and (7.7) of~\cite{MSW2017}. At the end of Section~\ref{subsec:one-arm}, we explain how  Theorems~\ref{thm:touch-BCLE-simple}--\ref{thm:CR-BCLE-nonsimple} can be used to give a new derivation of this relation. 

Although the proof of Theorem~\ref{thm:fuzzy-potts-one-arm} does not require understanding the critical case $\kappa=4$, the previous theorems also extend to this regime. Not surprisingly, the formulae from the $\kappa \uparrow 4$ limit of Theorems~\ref{thm:touch-BCLE-simple} and~\ref{thm:CR-BCLE-simple} match with those from the $\kappa' \downarrow 4$ limit of Theorems~\ref{thm:touch-BCLE-nonsimple} and~\ref{thm:CR-BCLE-nonsimple}. We state here the results precisely, with the same notations as before. These results were first established in~\cite{ASW19}.

\begin{theorem}\label{thm:touch-BCLE-4}
    For $\rho \in (-2,0)$, let $\{ 0 \in \BCLE_4^\clockwise(\rho) \}$ (resp.\ $\{ 0 \not \in \BCLE_4^\clockwise(\rho) \}$) denote the event that the origin is surrounded by a clockwise true (resp.\ counterclockwise false) loop in $\BCLE_4^\clockwise(\rho)$. Then,
    \begin{equation}\label{eq:touch-BCLE-4}
        \PP[0 \in \BCLE_4^\clockwise(\rho)] = -\frac{\rho}{2}  \quad \textrm{and}\quad 
        \PP[0 \notin \BCLE_4^\clockwise(\rho)] = \frac{\rho+2}{2} \,. 
    \end{equation}
\end{theorem}

\begin{theorem}\label{thm:CR-BCLE-4}
    For $\rho \in (-2,0)$ and $\lambda>-\frac{1}{2}$, we have
    \begin{equation}\label{eq:CR-BCLE-4}
        \E[\CR(0,D_{\mathcal{L}})^\lambda \1_{0 \in \BCLE_4^\clockwise(\rho)}]=\frac{\sinh(\sqrt{2\lambda} (-\frac{\rho}{2} \pi))}{\sinh(\sqrt{2\lambda} \pi)}  \quad \textrm{and}\quad  \E[\CR(0,D_{\mathcal{L}})^\lambda \1_{0 \notin \BCLE_4^\clockwise(\rho)}]=\frac{\sinh(\sqrt{2\lambda} \frac{\rho+2}{2} \pi)}{\sinh(\sqrt{2\lambda} \pi)} \,,
    \end{equation}
  {where we set $\sqrt{a}=i\sqrt{-a}$ for $a<0$ as in Theorem~\ref{thm:CR-BCLE-simple}.}      Moreover, if $\lambda \le -\frac{1}{2}$, then both the left hand sides of~\eqref{eq:CR-BCLE-4} are infinite.
\end{theorem}
 
\subsection{Proof of Theorems~\ref{thm:touch-BCLE-simple}--\ref{thm:CR-BCLE-nonsimple} based on Liouville quantum gravity}\label{subsec:lqg}

Liouville quantum gravity (LQG) is a theory of random surfaces arising from string theory~\cite{Pol81}, where the geometry on the surface is governed by variants of Gaussian free field (GFF). The interplay between LQG and random geometry starts with the study of random planar maps~\cite{LG13,BM17,HS19,GM21}. A  {pioneering} work of Sheffield~\cite{She16a} points out that SLE curves can be viewed as the interface between two conformally-welded LQG surfaces. This results in a powerful and rich coupling theory of SLE and LQG~\cite{DMS21}, known as the mating-of-trees theory. See~\cite{ghs-mating-survey} for its various applications. In~\cite{MSW2017, MSW21-nonsimple, MSW22-simple}, Miller, Sheffield, and Werner use this coupling to study CLE. They developed enough tools for the derivation in~\cite{KSL22} of the boundary arm exponents and all of the bulk arms exponents except the one-arm case.

Our proof of Theorems~\ref{thm:touch-BCLE-simple}--\ref{thm:CR-BCLE-nonsimple} is another application of the SLE/LQG coupling. The key difference of our method from the one in~\cite{MSW2017, MSW21-nonsimple, MSW22-simple} lies in the synergy with Liouville conformal field theory (LCFT). This is a 2D quantum field theory rigorously developed in~\cite{DKRV16} and subsequent works. In the framework of Belavin, Polyakov, and Zamolodchikov's conformal field theory~\cite{BPZ84},  LCFT enjoys rich and deep exact solvability, which has been established recently~\cite{krv-dozz, GKRV-sphere, RZ22}; see the review~\cite{GKR-review}. It turns out that the variants of GFF governing many natural LQG surfaces can be described in terms of LCFT~\cite{AHS17,Cer21,AHS21,ASY22}. Armed with the exact solvability of LCFT, the conformal welding of LQG surfaces can then be used to derive exact formulae for conformal radii related to SLE and CLE. This method  {was} pioneered by~\cite{AHS21},  {and we apply this method} to prove Theorems~\ref{thm:touch-BCLE-simple}--\ref{thm:CR-BCLE-nonsimple}.

Our proof starts with the construction of BCLE using target invariant SLE processes. For the simple regime $\kappa \in (2,4)$, this construction allows us to characterize the BCLE loop surrounding the origin using $\SLE_\kappa$-type processes. These curves arise naturally as conformal welding interfaces of quantum disks~\cite{AHS23} and quantum triangles~\cite{ASY22}, which makes it possible to describe the BCLE loop also as a conformal welding interface of two $\gamma$-LQG surfaces, where $\gamma=\sqrt{\kappa}$.  For the non-simple regime $\kappa' \in (4,8)$, we follow the same line except that we work with quantum surfaces with non-simple boundary, also known as generalized quantum surfaces~\cite{DMS21,MSW21-nonsimple,AHSY23}. These welding results allow us to express~\eqref{eq:touch-BCLE-simple-A}--\eqref{eq:CR-BCLE-nonsimple-B} in terms of boundary lengths of LQG surfaces, whose laws are described via boundary Liouville two-point and three-point functions which were derived in~\cite{RZ22}. 

There are several difficulties in implementing the strategy above that we have to overcome. The main technical step is to identify proper conformal welding results (Theorems~\ref{thm:weld-BCLE} and~\ref{thm:weld-BCLE-non-simple}) and prove them. Our treatment in the non-simple regime is an extension from~\cite{ASYZ24} except that the relevant quantum surfaces are more involved. The simple regime requires new ideas. First, we need to identify a conformal welding picture where the orientation of the loop naturally arise. Moreover, to prove the conformal welding result, we start with an auxiliary welding statement regarding weight $\gamma^2-2$ quantum disks, then reform and discard certain quantum surfaces. Another difficulty comes from deducing exact formulae~\eqref{eq:touch-BCLE-simple-A}--\eqref{eq:CR-BCLE-nonsimple-B} from the conformal welding, as the structure constants of LCFT are a priori very complicated. We overcome this difficulty using the so-called thick/thin duality from~\cite{AHS21} and shift equations for boundary Liouville reflection coefficients~\cite{RZ22}.

The rest of the paper is organized as follows. In Section~\ref{sec:BCLE}, we first review the concept of CLE percolations and BCLE, and then present the proof of Theorem~\ref{thm:one-arm} using Theorems~\ref{thm:CR-BCLE-simple} and~\ref{thm:CR-BCLE-nonsimple}. We also recap the fuzzy Potts model and show how to pass from the continuum to the discrete (Theorem~\ref{thm:fuzzy-potts-one-arm}). In Section~\ref{sec:quantum-surface}, we recall the necessary background on LQG surfaces and introduce the pinched thin quantum annulus, which plays a key role in the subsequent conformal welding results. The proofs of conformal weldings for the simple and non-simple regimes are then established in Sections~\ref{sec:welding-simple} and~\ref{sec:welding-non-simple}, respectively. Based on them, we derive Theorems~\ref{thm:CR-BCLE-simple} and~\ref{thm:CR-BCLE-nonsimple} in Section~\ref{sec:proof}.
Finally, we supplement the proof of Theorems~\ref{thm:touch-BCLE-4} and~\ref{thm:CR-BCLE-4} in Section~\ref{sec:BCLE-4}.

\subsection{Outlook and perspectives}\label{subsec:outlook}

We conclude the introduction with some discussion on related works.

\begin{itemize}
    \item The expression~\eqref{eq:one-arm} for the bulk one-arm exponent is  similar to the backbone exponent for percolation derived in~\cite{NQSZ23}, which is the root of an elementary equation. In~\cite{SXZ24} with Xu, the second named and the fourth named authors derived an exact expansion formula for the annulus crossing probability for percolation. Its leading asymptotic is given by the backbone exponent, while the growth rate of the remaining terms  in the expansion are captured by the other roots of the elementary equation. 
    We expect that a similar phenomenon occurs for the  annulus crossing probability of the fuzzy Potts model with proper boundary conditions, and plan to derive it using the strategy from~\cite{SXZ24} based on Liouville quantum gravity on the annulus. See the introduction of~\cite{SXZ24} for a possible CFT interpretation of such results.

    \item The two-dimensional critical 3-state Potts model is of substantial interest in the physics community.
    We obtained in Corollary~\ref{cor:3-Potts-bichromatic} that the bichromatic one-arm exponent is $4/135$.  In the future, we plan to use CLE percolation to investigate other aspects of the 3-state Potts model that are of interest in physics, such as the spin cluster~\cite{DPSV13}, the spin interfaces~\cite{DJS10,FPS20}, the connection to CFT minimal models~\cite[Section 7.4.4]{CFT-book}, and the conformal boundary conditions~\cite{AOS98}. In particular, in a forthcoming work by the first named and the fourth named authors with Gefei Cai and Baojun Wu, they will derive the three-point connectivity constant for the spin cluster, whose exact value is not known before this work.    
 
    \item In~\cite{MSW21-nonsimple,MSW22-simple}, the authors established the explicit relation between the $\rho$-parameter in the BCLE and the $r$-parameter in the CLE percolation (see~\eqref{eq:relation-nonsimple} and~\eqref{eq:relation-simple}) using the coupling of CLE and LQG, where the sine functions naturally arises from the ratio between the intensity of upward and downward jumps of certain L\'evy processes. It is interesting to see if  {this type of technique} can lead to a new derivation of the bulk one-arm exponent without using the integrability of Liouville CFT.
\end{itemize}

\medskip
\noindent\textbf{Acknowledgements.} 
We are grateful for enlightening discussions with Pierre Nolin and Wei Qian during the early stage of this project. We also thank Wendelin Werner for helpful communication.  {We thank an  anonymous referee for helpful comments on an earlier version of this article.} H.L. and X.S.\ were supported by National Key R\&D Program of China (No.\ 2023YFA1010700).
P.Y.\ was partially supported by NSF grant DMS-1712862. Z.Z.\ is partially supported by NSF grant DMS-1953848.

\section{CLE percolations, BCLE, and the fuzzy Potts model}\label{sec:BCLE}

This section is dedicated to the proof of the one-arm exponent for $\CLE_{\kappa'}$ percolations (Theorem~\ref{thm:one-arm}) and for the fuzzy Potts model (Theorem~\ref{thm:fuzzy-potts-one-arm}). Starting from the continuum side, we will first recall some background of CLE and BCLE in Section~\ref{subsec:CLE-prelim}, and briefly restate CLE percolations results in terms of BCLE.  {See also~\cite[Section 2.5 and Section 2.6]{KSL22} for a detailed introduction on BCLE and CLE percolations.} In Section~\ref{subsec:one-arm}, we define the one-arm exponent for $\CLE_{\kappa'}$ percolations, and derive its value for all $\kappa' \in (4,8)$ via moments of conformal radius of BCLE derived in Theorem~\ref{thm:CR-BCLE-simple} and~\ref{thm:CR-BCLE-nonsimple}. Finally, in Section~\ref{subsec:discrete}, we will introduce the discrete counterpart of $\CLE_{\kappa'}$ percolations --- the fuzzy Potts model. Assuming that the critical FK percolation with cluster weight $q \in [1,4)$ converges to nested $\CLE_{\kappa'}$ in the scaling limit for $\kappa'=4\pi/\arccos(-\sqrt{q}/2)$, based on~\cite{KSL22} we show that the one-arm exponent for the fuzzy Potts model with cluster weight $q$ is the same as that of $\CLE_{\kappa'}$ percolations.

\subsection{Description of CLE percolations via BCLE}\label{subsec:CLE-prelim}

The motivation of this paper is the seminal work on CLE percolations~\cite{MSW2017}, and we will briefly explain their main results in this subsection. We start with the chordal \emph{Schramm-Loewner evolution} (SLE) process on the upper half plane $\mathbb{H}$. Let $(B_t)_{t \ge 0}$ be the standard Brownian motion. Each non-self-crossing curve $\eta$ in $\overline{\mathbb{H}}$ from 0 to $\infty$ is associated with mapping-out function $(g_t)_{t \ge 0}$, that is for each $t \ge 0$, $g_t(z)$  {is} the unique conformal map from the unbounded component of $\mathbb{H} \setminus \eta[0,t]$ to $\mathbb{H}$ such that $\lim_{|z|\to\infty}|z(g_t(z)-z)|=2t$. For $\kappa>0$, the chordal $\SLE_\kappa$ is a probability measure on these curves $\eta$ such that we have the Loewner equation with driving function $W_t=\sqrt{\kappa} B_t$:
\begin{equation}\label{eq:loewner}
    g_t(z)=z+\int_0^t \frac{2}{g_s(z)-W_s} \dd s, \quad \mbox{for }z \in \mathbb{H}.
\end{equation}
For a simply connected domain $D$ with two boundary points $a,b$, the chordal $\SLE_\kappa$ in $D$ from $a$ to $b$ is defined using conformal transformations.

For $\kappa>0$, let $v_1,\ldots,v_m \in \mathbb{R}$ be $m$ boundary marked points and $\rho_1,\ldots,\rho_m>-2$ be their weights. The $\SLE_\kappa(\rho_1,\ldots,\rho_m)$ process on $\mathbb{H}$ from 0 to $\infty$ with force points $v_1,\ldots,v_m$ is the probability measure on curves $\eta$ satisfying the same equation~\eqref{eq:loewner}, except that the driving function is determined by
\begin{align*}
    W_t &=\sqrt{\kappa} B_t+\sum_{i=1}^m \int_0^t \frac{\rho_i}{W_s-g_s(v_i)} \dd s \,, \\
    g_t(v_i) &=v_i+\int_0^t \frac{2}{g_s(v_i)-W_s} \dd s, \quad i=1,\ldots,m \,.
\end{align*}
For chordal $\SLE_\kappa(\underline\rho)$ processes with only two force points which are located at $0^\pm$,  we write   $\SLE_\kappa(\rho_1;\rho_2)$ without reference to force points. These processes are conformally invariant, which makes it possible to define them in any simply connected domains by conformal transformations.

We will also need the radial version of certain SLE processes. We begin with the radial $\SLE_\kappa$ processes. For a curve $\eta$ from 1 targeted at 0 in $\bbD$, we write $D_t$ for the connected component of  {$\overline\bbD\backslash \eta([0,t])$} containing 0, and $K_t = \overline\bbD\backslash D_t$. %\haoyu{**should be ``we write $D_t$ for the connected component of $\overline\bbD\backslash \eta([0,t])$ containing 0, and $K_t = \overline\bbD\backslash D_t$'' ? **}
Let
$$\Psi(u,z)=\frac{u+z}{u-z},\, \Phi(u,z) = z\Psi(u,z),\, \text{and}\, \hat{\Phi}(u,z) = \frac{\Phi(u,z)+\Phi(1/\bar{u},z)}{2}. $$
For $\kappa>0$, the radial $\SLE_\kappa$ curve $\eta$ in $\bbD$ from 1 to 0 is defined by 
\eqb\label{eq-rad-sle}
\dd g_t(z) = \Phi(U_t, g_t(z)) \dd t\quad \text{ for } z\in\mathbb{D}\backslash K_t,  
\eqe
where $U_t = \exp(i\sqrt\kappa B_t)$ and $g_t$ is the unique conformal transformation $\mathbb{D}\backslash K_t\to \mathbb{D}$ fixing 0 with $g_t'(0)>0$ and $\log g_t'(0)=t$. 

For $\rho_1,\ldots,\rho_n\in\bbR$, %\haoyu{**mention $\rho>-2$ ?**}, $x_1,\ldots, x_n\in\partial{\mathbb{D}}$, 
the radial $\SLE_\kappa(\underline{\rho})$ in $\mathbb{D}$  {with force points at $x_i$ and} targeted at 0 is the curve $\eta$ in $\mathbb{D}$ characterized by a random family of conformal maps $(g_t)_{t \ge 0}$ solving
\begin{equation}\label{eqn:def-radial-sle}
    \begin{split}
        &\dd U_t = -\frac{\kappa}{2}U_t \dd t+i\sqrt{\kappa}U_t \dd B_t+\sum_{j=1}^n\frac{\rho_j}{2}\hat{\Phi}(g_t(x_j),U_t) \dd t \,, \\
        &\dd g_t(z) = \Phi(U_t, g_t(z)) \dd t\ \text{for}\ z\in\ol{\mathbb{D}} \,.
    \end{split}
\end{equation}
Again $g_t$ is the unique conformal transformation $\mathbb{D}\backslash K_t\to \mathbb{D}$ fixing 0 with $g_t'(0)>0$ and $\log g_t'(0)=t$.
For $\underline\rho = (\rho_1,\ldots,\rho_n)$ with $\rho_1+\cdots+\rho_n=\kappa-6$, the $\SLE_\kappa(\underline\rho)$ processes satisfy the \emph{target invariance property}, in the sense that two radial/chordal $\SLE_\kappa(\underline\rho)$ curves with the same starting point and force points can be coupled such that they agree with each other until their targets are separated. In this paper, we will frequently use the notion of target-invariant chordal/radial $\SLE_\kappa(\kappa-6)$ and $\SLE_\kappa(\rho;\kappa-6-\rho)$.

For $\kappa \in (8/3,8)$, the non-nested ${\rm CLE}_\kappa$  {in a simply connected domain $D$}  is a random collection $\Gamma$ of non-crossing loops  {in $D$} introduced in~\cite{SheffieldCLE,Sheffield-Werner-CLE}, where each loop is an $\SLE_\kappa$-type curve and no loop surrounds another loop. When $\kappa \in (8/3,4]$, each loop in $\Gamma$ is almost surely simple and does not intersect either the boundary of the domain or other loops. When $\kappa \in (4,8)$, loops in  $\Gamma$ are nonsimple; they may touch each other without crossing.  {The \emph{carpet} ($\kappa\in(8/3,4]$) or \emph{gasket} ($\kappa\in(4,8)$) of $\Gamma$ is the set of points in $D$ not surrounded by any loop of $\Gamma$.} For $\beta \in [-1,1]$, the labeled $\CLE_\kappa^{\beta}$ is the oriented version of {non-nested} $\CLE_\kappa$, where each loop is independently oriented counterclockwise (resp.\ clockwise) with probability $(1+\beta)/2$ (resp.\ $(1-\beta)/2$). The nested version of $\CLE_\kappa$ can be constructed using iteration procedures inside the simply connected domains enclosed by CLE$_\kappa$ loops, and the union of the carpet/gasket in each of the non-nested CLE$_\kappa$ in the iteration forms the carpet/gasket of the nested $\CLE_\kappa$.

The \emph{boundary conformal loop ensemble} (BCLE), which is introduced in~\cite{MSW2017}, is a random collection of boundary-touching loops in a simply connected domain $D$ whose law is conformally invariant. For $\kappa\in(2,4]$ and $\rho \in (-2,\kappa-4)$, using branching $\SLE_\kappa(\rho;\kappa-6-\rho)$ processes, one can construct a branching tree $\cT$  which starts from a boundary point $x \in \partial D$ and targets at all other boundary points.
The branches in $\cT$ are naturally oriented from root $x$ towards other boundary points, and the boundaries of the connected components of $D\backslash\cT$ are either clockwise or counterclockwise loops. Then $\BCLE^\clockwise_\kappa(\rho)$ is defined to be the collection of those clockwise loops, which are also referred to as the \emph{true} loops. The boundaries of {the components of $D\backslash\cT$} that are not surrounded by a true loop then form a collection of counterclockwise loops, and are called the \emph{false} loops of $\BCLE^\clockwise_\kappa(\rho)$. It is clear that the collection of true loops and false loops determines each other. Moreover, as explained in~\cite{SheffieldCLE,MSW2017}, the law of $\BCLE^\clockwise_\kappa(\rho)$ does not depend on the choice of root $x$ and is invariant under any conformal automorphism of $D$.
We can also define a collection $\BCLE^\counterclockwise_\kappa(\rho)$ of counterclockwise loops by reversing the orientation of each loop in $\BCLE^\clockwise_\kappa(\rho)$, and its false loops are now clockwise loops. One can show that (the true loops of) $\BCLE^\counterclockwise_\kappa(\rho)$ can be realized as the false loops of $\BCLE^\clockwise_\kappa(\kappa-6-\rho)$. For $\rho = -2$, $\BCLE_\kappa^\clockwise(\rho)$ is defined to be the single loop tracing $\partial D$ clockwise, and there are no false loops. For $\rho = \kappa-4$, $\BCLE_\kappa^\counterclockwise(\rho)$ is  {defined to} be the counterclockwise loop $\partial D$, respectively.

Similarly, for $\kappa' \in (4,8)$ and $\rho' \in (\kappa'/2-4,\kappa'/2-2)$, we can define $\BCLE^\clockwise_{\kappa'}(\rho')$ and $\BCLE^\counterclockwise_{\kappa'}(\rho')$ using branching $\SLE_{\kappa'}(\rho';\kappa'-6-\rho')$ processes, which extend to $\rho'=\kappa'/2-4$ and $\rho'=\kappa'/2-2$ as well. Conformal invariance and the above $\rho'\leftrightarrow\kappa'-6-\rho'$ symmetry also hold for $\BCLE_{\kappa'}(\rho')$.

For the rest of the paper, unless explicitly stated, we always assume $\kappa \in (2,4)$ and write $\kappa'=16/\kappa \in (4,8)$. 
The CLE percolations  {give a} duality between $\CLE_\kappa$ and $\CLE_{\kappa'}$. Roughly speaking, it has two counterparts: on the one hand, if we independently color each $\CLE_{\kappa'}$  {cluster} red or blue using a $\frac{1+\beta}{2}$ v.s.\ $\frac{1-\beta}{2}$ biased coin where $\beta\in[-1,1]$, then the outer boundaries of the clusters of red (or blue) $\CLE_{\kappa'}$ clusters appear to be some $\BCLE_\kappa$ loops; on the other hand, $\CLE_{\kappa'}$ loops can be interpreted as critical percolation interfaces within $\CLE_{\kappa}$ carpets.
To be precise, we have the following duality results from~\cite[Theorems 7.2 and 7.7]{MSW2017}.  {See Figure~\ref{fig:cle_percolation} for an illustration.}

\begin{theorem}\label{thm:CLE-percolation}
     For each $\kappa \in (2,4)$ and $\beta \in [-1,1]$, there exists $\rho=\rho(\beta,\kappa) \in [-2,\kappa-4]$ such that the following holds.  {Let $D\subsetneq\bbC$ be a simply connected domain.} Let 
\begin{equation}\label{eq:parameter-nonsimple}
    \rho_R'=-\frac{\kappa'}{2}-\frac{\kappa'}{4} \rho, \quad \rho_B'=\kappa'-4+\frac{\kappa'}{4} \rho \,,
\end{equation}
then we can construct the labeled $\CLE_{\kappa'}^{\beta}$ from the iteration of $\BCLE^\clockwise_\kappa(\rho)$, $\BCLE^\counterclockwise_{\kappa'}(\rho_R')$ and $\BCLE^\clockwise_{\kappa'}(\rho_B')$. In particular, we first let $\Gamma^\clockwise = \Gamma^\counterclockwise = \emptyset$ and iterate as follows: 
\begin{enumerate}
    \item Sample $\Lambda \sim \BCLE^\clockwise_\kappa(\rho)$ in $D$.
    \item In the domains enclosed by clockwise true loops (resp.\ counterclockwise false loops) of $\Lambda$, we independently sample $\BCLE^\counterclockwise_{\kappa'}(\rho_R')$ (resp.\ $\BCLE^\clockwise_{\kappa'}(\rho_B')$). Then add the counterclockwise true loops of $\BCLE^\counterclockwise_{\kappa'}(\rho_R')$ to $\Gamma^\counterclockwise$ and the clockwise true loops of $\BCLE^\clockwise_{\kappa'}(\rho_B')$ to $\Gamma^\clockwise$.
    \item Iterate the previous two steps independently in every simply connected domain not enclosed by loops in $\Gamma^\clockwise \cup \Gamma^\counterclockwise$. (These domains correspond to the  {interiors} of false loops of $\BCLE_{\kappa'}^{\counterclockwise}(\rho_R')$ or $\BCLE_{\kappa'}^{\clockwise}(\rho_B')$ in the previous step.)
\end{enumerate}
 Finally, let $\Gamma = \Gamma^\clockwise \cup \Gamma^\counterclockwise$, then it has the same law as a labeled $\CLE_{\kappa'}^{\beta}$ on $D$. Moreover, the relation between $\beta$, $\kappa=16/\kappa'$ and $\rho$ is given by
\begin{equation}\label{eq:relation-nonsimple}
    \frac{1-\beta}{2}=\frac{\sin(\pi \rho/2)}{\sin(\pi \rho/2)-\sin(\pi (\kappa-\rho)/2)} \,.
\end{equation}
\end{theorem}

\begin{figure}
    \centering
    \includegraphics[scale=0.4]{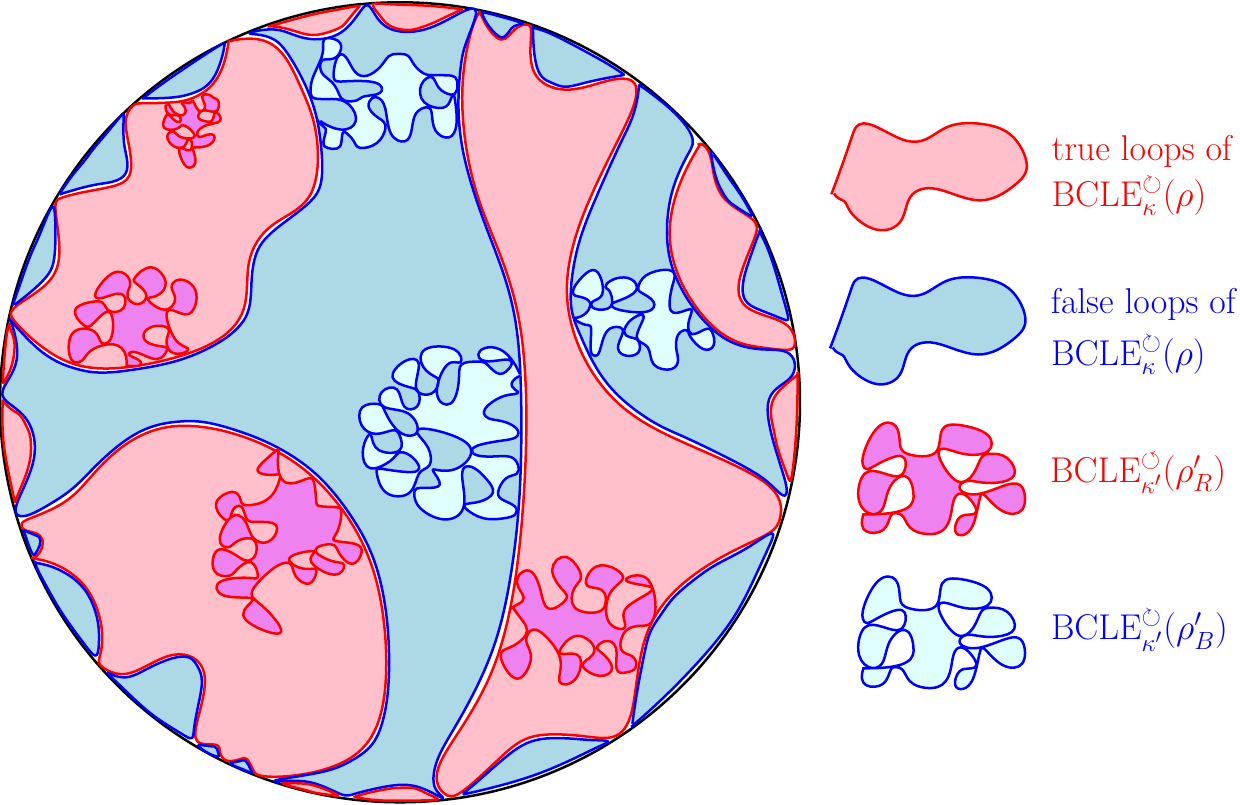}
    \caption{ {An illustration of Theorem~\ref{thm:CLE-percolation}. First, sample $\Lambda \sim \BCLE_\kappa^\clockwise(\rho)$. Then, in each domain enclosed by a clockwise true loop (resp.\ counterclockwise false loop), sample $\BCLE_{\kappa'}^\counterclockwise(\rho_R')$ (resp.\ $\BCLE_{\kappa'}^\clockwise(\rho_B')$), and add the resulting loops to $\Gamma^\counterclockwise$ (resp.\ $\Gamma^\clockwise$); these loops will form part of $\CLE_{\kappa'}^\beta$. Iterating this two-step procedure indefinitely in the remaining simply connected domains not enclosed by loops in $\Gamma^\counterclockwise\cup \Gamma^\clockwise$ gives rise to the full collection of non-nested, labeled $\CLE_{\kappa'}^\beta$ loops.}}
    \label{fig:cle_percolation}
\end{figure}
     
The relation~\eqref{eq:relation-nonsimple} is derived in~\cite{MSW21-nonsimple}.  Similarly, by~\cite[Theorems 7.4 and 7.7]{MSW2017}, for each $\kappa' \in (4,6)$ and $\beta' \in [-1,1]$, there exists $\rho'=\rho'(\beta',\kappa') \in [\kappa'-6,0]$ such that the following holds. Let
\begin{equation}\label{eq:parameter-simple}
    \rho_R=-\frac{\kappa}{2}-\frac{\kappa}{4} \rho', \quad \rho_B=\kappa-4+\frac{\kappa}{4} \rho' \,.
\end{equation}
Then we can construct the labeled $\CLE_\kappa^{\beta'}$ from the iteration of $\BCLE^\clockwise_{\kappa'}(\rho')$, $\BCLE^\counterclockwise_\kappa(\rho_R)$ and $\BCLE^\clockwise_\kappa(\rho_B)$ using the same procedure as Theorem~\ref{thm:CLE-percolation} with   $\kappa$, $\kappa'$, $\rho$, $\rho_R'$, and $\rho_B'$ replaced by $\kappa'$, $\kappa$, $\rho'$, $\rho_R$ and $\rho_B$ correspondingly.
As proved in~\cite{MSW22-simple}, the relation between $\beta'$, $\kappa'=16/\kappa$ and $\rho'$ is given by
\begin{equation}\label{eq:relation-simple}
    \frac{1-\beta'}{2}=\frac{\sin(\pi \rho'/2)}{\sin(\pi \rho'/2)-\sin(\pi (\kappa'-\rho')/2)} \,.
\end{equation}

\subsection{Derivation of the one-arm exponent given Theorems~\ref{thm:CR-BCLE-simple} and~\ref{thm:CR-BCLE-nonsimple}}
\label{subsec:one-arm}

We focus on the $\CLE_{\kappa'}$ percolation, a continuum percolation on the $\CLE_{\kappa'}$ gasket. For each $\kappa' \in (4,8)$, the \emph{nested} $\CLE_{\kappa'}$ is conjectured to be the scaling limit of critical FK percolation with cluster weight $q=4\cos(\pi \frac{\kappa'-4}{\kappa'})^2 \in (0,4)$ (see Conjecture~\ref{conj:fk-to-cle}), where loops alternatively correspond to outer and inner boundaries of FK clusters. For a nested collection of loops in $\bbD$, we define the nesting level of each loop to be the number of distinct loops surrounding it plus 1, and let $\partial \mathbb{D}$ have nesting level 0 by convention. For FK percolation with free (resp.\ wired) boundary conditions, loops of odd and even (resp.\ even and odd) nesting levels correspond respectively to the outer and inner boundaries of FK clusters. Likewise, we can introduce free and wired boundary conditions for the continuum $\CLE_{\kappa'}$ percolation, while we mainly work on the free boundary condition. 

Consider a nested $\CLE_{\kappa'}$ $\Gamma$ in the unit disk $\bbD$, and let $\CLE_{\kappa'}$ clusters be the gasket squeezed inside an odd loop and outside of all the even loops that it surrounds. Fix $\beta\in[-1,1]$ and let $r=\frac{1+\beta}{2}$. We orient each odd loop in $\Gamma$ independently counterclockwise (resp.\ clockwise) with probability $(1+\beta)/2$ (resp.\ $(1-\beta)/2$), and orient each even loop in the opposite direction of its parent (i.e., the loop surrounding it with maximal nesting level).  We then color each $\CLE_{\kappa'}$ cluster in red (resp.\ blue) if its outer boundary is counterclockwise (resp.\ clockwise). This gives the $\CLE_{\kappa'}$ percolation with red probability $r$, i.e.,  each $\CLE_{\kappa'}$ cluster is independently colored in red with probability $r$ and in blue with probability $1-r$. Similar to discrete percolations, we are interested in the clusters of $\CLE_{\kappa'}$ clusters that are obtained by agglomerating all $\CLE_{\kappa'}$ clusters of the same color that touch each other. Then one can infer from Theorem~\ref{thm:CLE-percolation} that the boundary-touching interfaces between red and blue clusters of $\CLE_{\kappa'}$ clusters will form $\BCLE_{\kappa}(\rho)$ for $\kappa=16/\kappa'$ and $\rho$ determined by~\eqref{eq:relation-nonsimple}.

It is also natural to consider arm events in this continuum percolation. For any $\e>0$, let $\mathcal{A}_\e$ be the event that there exists a finite sequence of blue $\CLE_{\kappa'}$ clusters $C_1,\ldots,C_n$, such that $\e \mathbb{D} \cap \ol C_1 \neq \emptyset$, $\ol C_n \cap \partial \mathbb{D} \neq \emptyset$, and $\ol C_i \cap \ol C_{i+1} \neq \emptyset$ for all $1 \le i \le n-1$.   {Lemma~\ref{lem:continuum-exponent} below implies} that the probability of $\mathcal{A}_\e$ has power law decay in $\e$, and the exponent $\alpha_1(r)$ appearing in~\eqref{eq:def-one-arm-cle} is called the (blue) \emph{bulk one-arm exponent} of $\CLE_{\kappa'}$ percolation. Here, `bulk' refers to that one require a blue path from the neighborhood of a bulk point to the boundary. One can also define \emph{boundary} one-arm exponent by replacing $\mathbb{D}$ with half-disk $\mathbb{D} \cap \mathbb{H}$. This exponent was derived in~\cite[Theorem 1.2]{KSL22}  {under the FK conformal invariance conjecture}.

The main goal of this subsection is to derive the bulk one-arm exponent for the $\CLE_{\kappa'}$ percolation using Theorems~\ref{thm:CR-BCLE-simple} and~\ref{thm:CR-BCLE-nonsimple}. To this end, we may assume that $\bbC\backslash\bbD$ is colored  blue, and let $\mathcal{L}^o$ be the outermost interface between red and blue clusters which surrounds the origin and is red along its inner side. Denote by $D_{\mathcal{L}^o}$ the connected component of $\mathbb{D}\setminus\mathcal{L}^o$ that contains the origin. The next lemma shows that the one-arm exponent $\alpha_1(r)$ is the first pole of $\mathbb{E}[{\rm CR}(0,D_{\mathcal{L}^o})^{-x}]$.

\begin{lemma}\label{lem:continuum-exponent}
    Let $\xi=\inf\{x>0 : \mathbb{E}[{\rm CR}(0, D_{\mathcal{L}^o})^{-x}]=\infty\}$, then
    \[ \PP(\mathcal{A}_\e) = \e^{\xi+o(1)} \text{ as } \e \to 0.\]
\end{lemma}

\begin{proof}
    Note that $x \mapsto \mathbb{E}[{\rm CR}(0, D_{\mathcal{L}^o})^x]$ is the Laplace transform of the non-negative random variable $-\log {\rm CR}(0, D_{\mathcal{L}^o})$, which is clearly monotone in $x$ since ${\rm CR}(0, D_{\mathcal{L}^o}) \in [0,1]$. By Tauberian theorem (see, e.g.,~\cite[Theorem 3]{Nak07}), we find that
    \begin{equation}\label{eq:laplace}
        \lim_{x \to \infty} x^{-1} \log \PP(-\log {\rm CR}(0, D_{\mathcal{L}^o})>x)=-\xi.
    \end{equation}

    Let $R_{\mathcal{L}^o}$ be the Euclidean distance from 0 to $\mathcal{L}^o$, then the one-arm event $\mathcal{A}_\e$ is exactly $\{R_{\mathcal{L}^o}<\e\}$. By Schwarz lemma and Koebe $1/4$ theorem, we have $\frac{1}{4} {\rm CR}(0, D_{\mathcal{L}^o}) \le R_{\mathcal{L}^o} \le {\rm CR}(0, D_{\mathcal{L}^o})$, which gives $\{ {\rm CR}(0, D_{\mathcal{L}^o})<\e \} \subseteq \mathcal{A}_\e \subseteq \{ {\rm CR}(0, D_{\mathcal{L}^o})<4\e \}$. The conclusion is immediate from \eqref{eq:laplace}.
\end{proof}

\begin{proof}[Proof of Theorem~\ref{thm:one-arm}]
Recall the nested $\CLE_{\kappa'}$ $\Gamma$ in $\bbD$ and its percolation.
Let $\rho$ be the solution  {in $(-2,\kappa-4)$} to equation~\eqref{eq:relation-nonsimple} with $r=(1+\beta)/2$, i.e.,
\begin{equation}\label{eq:rho_R}
    {\tan( \frac{\pi}{2}(\rho+2)) =   \frac{\sin(\pi \kappa/2)}{1 + \cos(\pi \kappa/2) - 1/(1-r)} \,.}
\end{equation}
Let $\rho_R'$ and $\rho_B'$ be defined as in~\eqref{eq:parameter-nonsimple}. Then following Theorem~\ref{thm:CLE-percolation}, we can explore $\mathcal{L}^o$ as follows:

\noindent {\bf Step 1.} Sample $\Xi_0 \sim {\rm BCLE}^\clockwise_\kappa(\rho)$ in the unit disk $D_0:=\mathbb{D}$. There are (almost surely) two cases:
\begin{itemize}
    \item If 0 is enclosed by a clockwise true loop $\eta_0$ of $\Xi_0$, then let $\mathcal{L}^o=\eta_0$,  {$D_1 = \emptyset$} and \textbf{stop}.
    \item If 0 is enclosed by a counterclockwise false loop $\eta_0^*$ of  {$\Xi_0$}, then independently sample $\Xi_0' \sim {\rm BCLE}^\clockwise_{\kappa'}(\rho_B')$ in the domain enclosed by $\eta_0^*$. Go to \textbf{Step 2}.
\end{itemize}
\noindent {\bf Step 2.} Define the domain $D_1$ as follows:
\begin{itemize}
    \item If 0 is enclosed by a clockwise true loop $\eta_0'$ of $\Xi_0'$, let $\Gamma_0'$ be the non-nested ${\rm CLE}_{\kappa'}$ in the connected component of $\mathbb{D}\setminus  {\eta_0'}$ containing 0. There is a unique loop $\tilde{\eta}_0$ in  {$\Gamma'_0$} surrounding 0, and let $D_1$ be the connected component of $\mathbb{D}\setminus\tilde{\eta}_0$ containing 0.
    \item If 0 is enclosed by a counterclockwise false loop $\eta_0'^*$ of  {$\Xi'_0$}, let $D_1$ be the connected component of $\mathbb{D}\!\setminus\!\eta_0'^*$ containing 0.
\end{itemize}
\noindent {\bf Step 3.} Go back to {\bf Step 1} and continue the exploration in the domain $D_1$. Increase the indices of the corresponding random objects by 1 each time.

This procedure ends with probability $\mathbb{P}[0 \in \BCLE_{\kappa}^\clockwise(\rho)]>0$ each time going to \textbf{Step 1}, hence the exploration terminates a.s.\ and outputs $\mathcal{L}^o$. Since conformal radii are multiplicative and each step of exploration is independent of each other, we have
\begin{equation}\label{eq:CR-D1-1.1}
\begin{split}
    \bbE[\CR(0,D_1)^\lambda\mathds{1}_{D_1\neq\emptyset}] &= \E[\CR(0, D_{\mathcal{L}})^{\lambda} \1_{0 \notin \BCLE_{\kappa}^\clockwise(\rho)}]\times\bigg( C_{\CLE}(\lambda)\E[\CR(0, D_{\mathcal{L}'})^{\lambda} \1_{0 \in \BCLE_{\kappa'}^\clockwise(\rho_B')}] \\&+\E[\CR(0, D_{\mathcal{L}'})^{\lambda} \1_{0 \notin \BCLE_{\kappa'}^\clockwise(\rho_B')}] \bigg).
    \end{split}
\end{equation}
Here the domains $D_{\cL}$ and $D_{\cL'}$ are defined as in Theorems~\ref{thm:CR-BCLE-simple} and~\ref{thm:CR-BCLE-nonsimple}, and $C_{\CLE}(\lambda) = \bbE[\CR(0,D_{\Gamma'})^\lambda]$ with $D_{\Gamma'}$ being the connected component of a unit disk minus a non-nested $\CLE_{\kappa'}$ containing 0. Furthermore,  {the exploration rules and independence between different exploration steps  yield that}
\begin{equation}\label{eq:CR-D1-1.0}
    { \mathbb{E}[{\rm CR}(0, D_{\mathcal{L}^o})^\lambda] = \sum_{k=0}^\infty  \bbE[\CR(0,D_1)^\lambda\mathds{1}_{D_1\neq\emptyset}]^k \times 
   \E[\CR(0, D_{\mathcal{L}})^{\lambda} \1_{0  \in \BCLE_{\kappa}^\clockwise(\rho)}].}
\end{equation}
In particular, if $\E[\CR(0, D_{\mathcal{L}})^{\lambda} \1_{0  \in \BCLE_{\kappa}^\clockwise(\rho)}]<\infty$ and $\bbE[\CR(0,D_1)^\lambda\mathds{1}_{D_1\neq\emptyset}]<1$, then it is clear from~\eqref{eq:CR-D1-1.0} that  {$ \mathbb{E}[{\rm CR}(0, D_{\mathcal{L}^o})^\lambda] = \E[\CR(0, D_{\mathcal{L}})^{\lambda} \1_{0  \in \BCLE_{\kappa}^\clockwise(\rho)}]\times (1-\bbE[\CR(0,D_1)^\lambda\mathds{1}_{D_1\neq\emptyset}])^{-1} <\infty$}. Furthermore, we can infer from Theorem~\ref{thm:CR-BCLE-simple} and~\eqref{eq:CR-D1-1.1} that if  $\E[\CR(0, D_{\mathcal{L}})^{\lambda} \1_{0  \in \BCLE_{\kappa}^\clockwise(\rho)}]=\infty$ then $\bbE[\CR(0,D_1)^\lambda\mathds{1}_{D_1\neq\emptyset}]=\infty$.  {Furthermore, if we let $\lambda_0<0$ be the solution to $\bbE[\CR(0,D_1)^\lambda\mathds{1}_{D_1\neq\emptyset}]=1$, then   $ \mathbb{E}[{\rm CR}(0, D_{\mathcal{L}^o})^\lambda]$ is finite for $\lambda>\lambda_0$ and tends to $\infty$ as $\lambda \downarrow\lambda_0$.} Thus combined with Lemma~\ref{lem:continuum-exponent} it is not hard to see that the one arm exponent $\alpha_1(r)$ is equal to  {$-\lambda_0$.} %$\inf\{x>0:\bbE[\CR(0,D_1)^{-x}\mathds{1}_{D_1\neq\emptyset}]=1\}$.

Now let $\lambda>\frac{3\kappa'}{32}+\frac{2}{\kappa'}-1$, $x = \frac{\pi}{4}\sqrt{(4-\kappa)^2-8\kappa\lambda}$, $y = \frac{2\rho}{\kappa}x$, and $z =\frac{4}{\kappa}x$.
By Theorems~\ref{thm:CR-BCLE-simple} and~\ref{thm:CR-BCLE-nonsimple},
\begin{align*}
    \mathbb{E}[{\rm CR}(0,D_{\cL})^{\lambda} \1_{0 \in {\rm BCLE}^\clockwise_\kappa(\rho)}] &= \frac{\sin(\frac{\pi(4-\kappa)}{4}) \sin(\frac{2\pi}{\kappa}(\kappa-\rho-4))}{\sin(\frac{\pi(4-\kappa)}{\kappa}) \sin(\frac{\pi}{4}(\kappa - 2\rho-4))} \cdot \frac{\sin(x-y-z)}{\sin(x)} \,, \\
    \mathbb{E}[{\rm CR}(0,D_{\cL})^{\lambda} \1_{0 \notin {\rm BCLE}^\clockwise_\kappa(\rho)}] &= \frac{\sin(\frac{\pi(4-\kappa)}{4}) \sin(\frac{2\pi}{\kappa}(\rho+2))}{\sin(\frac{\pi(4-\kappa)}{\kappa}) \sin(\frac{\pi}{4}(\kappa - 2\rho-4))} \cdot \frac{\sin(y+2z-x)}{\sin(x)} \,, \nonumber \\
    \mathbb{E}[{\rm CR}(0,D_{\cL'})^{\lambda} \1_{0 \in {\rm BCLE}^\clockwise_{\kappa'}(\rho_B')}] &= \frac{\sin(\frac{\pi(4-\kappa)}{\kappa}) \sin(-\frac{\pi}{2} \rho)}{\sin(\frac{\pi(4-\kappa)}{4}) \sin(\frac{2\pi}{\kappa}(\rho + 2))} \cdot \frac{\sin(x-y-z)}{\sin(z)} \,, \nonumber \\
    \mathbb{E}[{\rm CR}(0,D_{\cL'})^{\lambda} \1_{0 \notin {\rm BCLE}^\clockwise_{\kappa'}(\rho_B')}] &= \frac{\sin(\frac{\pi(4-\kappa)}{\kappa}) \sin(\frac{\pi}{4}(\kappa-2\rho - 4))}{\sin(\frac{\pi(4-\kappa)}{4}) \sin(\frac{2\pi}{\kappa}(\rho + 2))} \cdot \frac{\sin(y+z)}{\sin(z)}. \nonumber
\end{align*}
Further, from~\cite[Theorem 1]{SSW09} we know that $C_{\CLE}(\lambda) =  \frac{\cos(\frac{\pi(4-\kappa)}{4})}{\cos(x)}$. Thus by~\eqref{eq:CR-D1-1.1}, $1-\bbE[\CR(0,D_1)^{\lambda}\mathds{1}_{D_1\neq\emptyset}]$ is equal to
\begin{align*}
    &\quad 1-\frac{\sin(y+2z-x)}{\sin(x)} \Big(\frac{\sin(-\frac{\pi}{2} \rho) \cos(\frac{\pi(4-\kappa)}{4})}{\sin(\frac{\pi}{4}(\kappa - 2\rho-4))} \cdot \frac{\sin(x-y-z)}{\sin(z) \cos(x)} +\frac{\sin(y+z)}{\sin(z)} \Big) \\
    &= \frac{\sin(x-y-z)\cdot \Big(\sin(\frac{\pi}{4}(\kappa - 2\rho-4)) \sin(x+y+2z) - \sin(\frac{\pi}{4}(-\kappa-2\rho+4)) \sin(y+2z-x)  \Big)}{2\sin(x)\cos(x)\sin(z)\sin(\frac{\pi}{4}(\kappa - 2\rho-4))}.
\end{align*}
The above equation follows from elementary identities for trigonometric functions and we omit the detailed steps of calculation.
It is clear that~$1-\bbE[\CR(0,D_1)^{\lambda}\mathds{1}_{D_1\neq\emptyset}]$ is increasing in terms of $\lambda$, takes a positive value when $\lambda=0$ and goes to $-\infty$ when  {$\lambda$} approaches $\frac{3\kappa'}{32}+\frac{2}{\kappa'}-1$ since $C_{\CLE}(\lambda)$ blows up. Moreover $\sin(x-y-z)$ is nonzero in this range. Therefore it follows that there is a unique $\lambda \in (\frac{3\kappa'}{32}+\frac{2}{\kappa'}-1,0)$ such that $\lambda=-\alpha_1(r)$ and~\eqref{eq:one-arm} holds. This concludes the proof.
\end{proof}

We conclude this subsection with an alternative proof to the relations~\eqref{eq:relation-nonsimple} and~\eqref{eq:relation-simple} using Theorems~\ref{thm:touch-BCLE-simple} and~\ref{thm:touch-BCLE-nonsimple}. These relations are established in~\cite{MSW21-nonsimple, MSW22-simple} using more involved techniques.
\begin{proof}[Proof of~\eqref{eq:relation-nonsimple} and~\eqref{eq:relation-simple}]
    Recall the construction of labeled $\CLE_{\kappa'}^\beta$ $\Gamma$ illustrated in Theorem~\ref{thm:CLE-percolation}. We say a loop in $\Gamma$ is $n$-th generation if it is added to $\Gamma$ in the $n$-th iteration.  By conformal invariance of BCLE, the probability where 0 is surrounded by a first generation counterclockwise (resp.\ clockwise) loop is $\mathbb{P}[0 \in {\rm BCLE}^\clockwise_\kappa(\rho)] \cdot \mathbb{P}[0 \in {\rm BCLE}^\counterclockwise_{\kappa'}(\rho_R')]$ (resp.\ $\mathbb{P}[0 \not \in {\rm BCLE}^\clockwise_\kappa(\rho)] \cdot \mathbb{P}[0 \in {\rm BCLE}^\clockwise_{\kappa'}(\rho_B')]$). If we write $p_1$ for the probability where $0$ is not surrounded by a first generation loop, then it is clear that the probability where $0$ is surrounded by an $n$-th generation counterclockwise (resp.\  {clockwise}) loop is $p_1^{n-1}\mathbb{P}[0 \in {\rm BCLE}^\clockwise_\kappa(\rho)] \cdot \mathbb{P}[0 \in {\rm BCLE}^\counterclockwise_{\kappa'}(\rho_R')]$ (resp.\ $p_1^{n-1}\mathbb{P}[0 \not \in {\rm BCLE}^\clockwise_\kappa(\rho)] \cdot \mathbb{P}[0 \in {\rm BCLE}^\clockwise_{\kappa'}(\rho_B')]$). Thus
    $$
    \frac{1+\beta}{1-\beta} = \frac{\mathbb{P}[0 \in {\rm BCLE}^\clockwise_\kappa(\rho)] \cdot \mathbb{P}[0 \in {\rm BCLE}^\counterclockwise_{\kappa'}(\rho_R') ]}{\mathbb{P}[0 \not \in {\rm BCLE}^\clockwise_\kappa(\rho)] \cdot \mathbb{P}[0 \in {\rm BCLE}^\clockwise_{\kappa'}(\rho_B') ]} \,.
    $$
    Here, $\{ 0 \in \BCLE^\counterclockwise_{\kappa'}(\rho_R') \}$ denotes the event that the origin is enclosed by a \emph{counterclockwise} true loop of a sample from $\BCLE^\counterclockwise_{\kappa'}(\rho_R')$. Then $\mathbb{P}[0 \in \BCLE^\counterclockwise_{\kappa'}(\rho_R') ] = \mathbb{P}[0 \in \BCLE^\clockwise_{\kappa'}(\rho_R') ]$, since the true loops of $\BCLE^\counterclockwise_{\kappa'}(\rho_R')$  {are} obtained by reversing the orientation of  the true loops of $\BCLE^\clockwise_{\kappa'}(\rho_R')$. Therefore, by Theorems~\ref{thm:touch-BCLE-simple} and \ref{thm:touch-BCLE-nonsimple},   $\frac{1+\beta}{1-\beta} = -\sin(\pi(\kappa-\rho)/2)/\sin(\pi \rho/2)$. This confirms~\eqref{eq:relation-nonsimple}.

    Similarly, from the construction of labeled $\CLE_\kappa^{\beta'}$, we obtain that
    $$
    \frac{1+\beta'}{1-\beta'} = \frac{\mathbb{P}[0 \in {\rm BCLE}^\clockwise_{\kappa'}(\rho')] \cdot \mathbb{P}[0 \in {\rm BCLE}^\counterclockwise_{\kappa}(\rho_R) ]}{\mathbb{P}[0 \not \in {\rm BCLE}^\clockwise_{\kappa'}(\rho')] \cdot \mathbb{P}[0 \in {\rm BCLE}^\clockwise_{\kappa}(\rho_B) ]} \,.
    $$
    Using $\mathbb{P}[0 \in {\rm BCLE}^\counterclockwise_{\kappa}(\rho_R) ] = \mathbb{P}[0 \in {\rm BCLE}^\clockwise_{\kappa}(\rho_R) ]$ and Theorems~\ref{thm:touch-BCLE-simple} and \ref{thm:touch-BCLE-nonsimple}, we derive that $\frac{1+\beta'}{1-\beta'} = -\sin(\pi(\kappa'-\rho')/2)/\sin(\pi \rho'/2)$. This confirms~\eqref{eq:relation-simple}.
\end{proof}

\subsection{Fuzzy Potts one-arm exponent}
\label{subsec:discrete}

In this section, we introduce the fuzzy Potts model based on the critical FK percolation. Assuming that the critical FK percolation converges to CLE in scaling limits, we show that the bulk one-arm exponent for the fuzzy Potts model is the same as that of $\CLE_{\kappa'}$ percolations.

We first recall the definition and some basic properties of FK percolation. For two vertices $u,v \in \mathbb{Z}^2$, we say $u,v$ are neighbors and write $u \sim v$ if they have Euclidean distance $1$. Let $G=(V,E)$ be a finite graph with vertex set $V \subset \mathbb{Z}^2$, and edge set $E=\{ \{u,v\}: u,v \in V \mbox{ with } u \sim v\}$. We define the (inner) vertex boundary of $V$ as $\partial V:=\{ v \in V : v \sim u \mbox{ for some } u \in \mathbb{Z}^2 \setminus V\}$. An edge configuration on $G$ is an element $\omega \in \{0,1\}^E$, where an edge $e \in E$ is said to be \emph{open} if $\omega_e=1$, and \emph{closed} otherwise. With a slight abuse of notation, we can view $\omega$ as a subgraph of $G$ with vertex set $V$ and edge set $o(\omega):=\{e \in E: \omega_e=1\}$. Partitions $\xi$ of $\partial V$ are called \emph{boundary conditions} on $G$. The FK percolation on $G$ with edge weight $p$, cluster weight $q$ and boundary condition $\xi$ is a probability measure on $\{0,1\}^E$ given by
\[ \phi_{G,p,q}^\xi(\omega):=\frac{1}{Z_{G,p,q}^\xi} p^{|o(\omega)|} (1-p)^{|E \setminus o(\omega)|} q^{k(\omega^\xi)} \,, \]
where $\omega^\xi$ is the graph obtained from $\omega$ by identifying vertices belonging to the same partition element of $\xi$ as wired together and let $k(\omega^\xi)$ be the number of connected components of the corresponding graph. The normalizing constant $Z_{G,p,q}^\xi$ is called \emph{partition function}, in the sense that $\phi_{G,p,q}^\xi$ is a probability measure.
The \emph{free} boundary conditions (denoted by 0) refer to the partition $\xi_0$ that each vertex in $\partial V$ forms a singleton, and the \emph{wired} boundary conditions (denoted by $1$) refer to the partition $\xi_1$ where the whole set $\partial V$ is a partition element.

Let $\Lambda_n=[-n,n]^2 \cap \mathbb{Z}^2$. The \emph{infinite-volume} FK percolation with boundary  {conditions} $\xi \in \{0,1\}$ is the measure $\phi^\xi_{\mathbb{Z}^2,p,q}$ defined as the weak limit of the measures $\phi^\xi_{\Lambda_n,p,q}$ along the sequence $(\Lambda_n)_{n \ge 1}$. The critical FK percolation refers to the case $p=p_c(q):=\sqrt{q}/(1+\sqrt{q})$ ~\cite{BDC12}.
It is shown in~\cite{DCST17} that $\phi^1_{\mathbb{Z}^2,p_c,q}=\phi^0_{\mathbb{Z}^2,p_c,q}$ for $q \in [1,4]$, and we write $\phi_{\mathbb{Z}^2,q}$ for the critical FK percolation with parameter $q$.

We are now in place to introduce the fuzzy Potts model, which was extensively studied in \cite{MV95,Hag99,KW07,KSL22}. A vertex configuration on the graph $G=(V,E)$ is an element $\sigma \in \{R,B\}^V$, where a vertex $v \in V$ is \emph{red} if $\sigma_v=R$, and \emph{blue} otherwise. For $q \in [1,4]$ and $r \in (0,1)$, the fuzzy Potts model on $G$ with cluster weight $q$, coloring parameter $r$ and boundary condition $\xi$ is a probability measure on $\{R,B\}^V$ constructed as follows:
\begin{enumerate}[(i)]
    \item Sample $\omega \in \{0,1\}^E$ from the  {critical} FK percolation $\phi_{G,p_c,q}^\xi$.
    \item Color each connected component $C$ of the graph independently in red with probability $r$ and in blue with probability $1-r$. By coloring a connected component $C$ in red (resp.\ blue), we refer to assigning $\sigma_v=R$ (resp.\ $\sigma_v=B$) for all vertices $v \in C$.
    \item In this way, we get a joint distribution of both edge and vertex configuration $(\omega,\sigma) \in \{0,1\}^E \times \{R,B\}^V$. Its second marginal (i.e., on $\{R,B\}^V$) is called the fuzzy Potts measure, for which we denote by $\mu_{G,q,r}^\xi$.
\end{enumerate}
In other words, the fuzzy Potts model is obtained by `forgetting about the edges' from the colored critical FK percolation. We can also define the infinite-volume fuzzy Potts measure $\mu_{\mathbb{Z}^2,q,r}$ by replacing $\phi^\xi_{G,p_c,q}$ with $\phi_{\mathbb{Z}^2,q}$ in the former construction.
In this paper, we mainly focus on the fuzzy Potts model with free boundary condition, i.e., $\xi=0$.

The readers may notice that the above construction is a generalization of the Edwards--Sokal coupling~\cite{Edwards-Sokal} between  {$q$-state} Potts model and FK percolation with cluster weight $q$ to $q\notin\bbZ$. When $q$ is an integer and $r=k/q$ for some $k \in \{1,\ldots,q-1\}$, the fuzzy Potts model can be obtained from the  {$q$-state} Potts model by coloring the vertices with spin in $\{1,\ldots,k\}$ in red and in blue otherwise. However, the fuzzy Potts model itself admits a continuous parameter $r \in (0,1)$.

For the rest of this section, we fix $q \in [1,4)$ and $r \in (0,1)$, and write $\mu$ for the fuzzy Potts measure $\mu_{\mathbb{Z}^2,q,r}$.
Let
\[ \kappa'=4\pi/\arccos(-\sqrt{q}/2) \in (4,6] \, \mbox{ and } \, \kappa=16/\kappa' \in [8/3,4) \,. \]
To properly describe the scaling limit of the fuzzy Potts model, we first recall the distance between two collections of loops.
For a non self-crossing loop $\eta\subset\bbC$, we can parameterize it by $(\eta(t))_{t \in \partial \mathbb{D}}$, and define its diameter by $\mathrm{diam}(\eta)=\sup_{t_1,t_2 \in \partial \mathbb{D}} |\eta(t_1)-\eta(t_2)|$. Let $\mathfrak{C}$ be the set of non self-crossing loops in $\mathbb{C}$ modulo time-parametrization, i.e., two loops $\eta_1$ and $\eta_2$ are equivalent if there exists a homeomorphism $\phi$ from $\partial \mathbb{D}$ to itself such that $\eta_1=\eta_2 \circ \phi$. For two elements $\eta_1$ and $\eta_2$ in $\mathfrak{C}$, we define
\[ d_{\mathfrak{C}}(\eta_1,\eta_2)=\inf_{\phi} \sup_{t \in \partial \mathbb{D}} |\eta_1(t)-\eta_2 \circ \phi (t)| \,, \]
where the infimum is taken over all homeomorphisms $\phi$ from $\partial \mathbb{D}$ to itself.

Let $\mathfrak{L}$ be the set of countable subsets $\Gamma$ of $\mathfrak{C}$ satisfying local finiteness property, i.e., for any $\e>0$, the number of loops in $\Gamma$ with diameter larger than $\e$ is finite. We define a metric on $\mathfrak{L}$ as follows: for two collections $\Gamma_1$ and $\Gamma_2$, we say $d_{\mathfrak{L}}(\Gamma_1,\Gamma_2) \le \e$ if and only if for any $\eta_1 \in \Gamma_1$ with $\mathrm{diam}(\eta_1)>\e$, there exist $\eta_2 \in \Gamma_2$ such that $d_{\mathfrak{C}}(\eta_1,\eta_2) \le \e$, and vice versa. Then $(\mathfrak{L},d_{\mathfrak{L}})$ is a Polish metric space.

For any simply connected domain $D$, let $D_n=(V_n,E_n)$ be discrete domains in $\frac{1}{n} \mathbb{Z}^2$ converging to $D$ as $n \to \infty$. Consider a critical FK percolation with cluster weight $q$ on $D_n$, let $\Gamma_n$ be the set of all inner and outer boundaries of its open clusters. It is widely believed that the limit of these collections of discrete loops can be characterized by a nested $\CLE_{\kappa'}$. As mentioned in the introduction, this so-called \emph{conformal invariance conjecture} of FK percolations is only known to hold for $q=2$ (i.e., $\kappa'=16/3$), due to~\cite{Smi10,KS16,KS19}.

\begin{figure}
    \centering
    \includegraphics[scale=0.45]{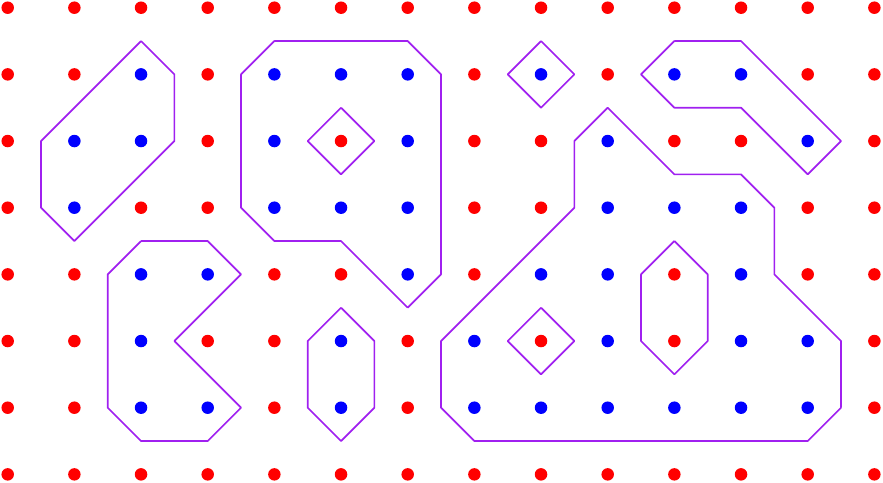}
    \caption{ {An illustration of the simple loops in $\Sigma^+_n$, which are the interfaces between red star-neighbor clusters and blue nearest-neighbor clusters.}}
    \label{fig:interface}
\end{figure}

\begin{conjecture}\label{conj:fk-to-cle}
    Let $\Gamma$ be a nested $\CLE_{\kappa'}$ in $D$, then $\Gamma_n$ converges in distribution to $\Gamma$ with respect to the metric $d_{\mathfrak{L}}$.
\end{conjecture}

For a site percolation model on a subgraph of $\frac{1}{n}\bbZ^2$, we  first clarify the neighboring relations. For $u,v \in \frac{1}{n}\mathbb{Z}^2$ and $u \neq v$, we say $u,v$ are \emph{nearest-neighbors} if $|u-v|_1=\frac{1}{n}$, and we say $u,v$ are \emph{star-neighbors} if $|u-v|_\infty=\frac{1}{n}$. Consider a fuzzy Potts model with parameters $q,r$ on $D_n$, let $\Sigma^+_n$ (resp.\ $\Sigma^-_n$) be the set of all interfaces between red star-neighbor (resp.\ nearest-neighbor) clusters and blue nearest-neighbor (resp.\ star-neighbor) clusters.  {Here we set the interfaces to be simple loops formed by straight line segments connecting the midpoints of edges of $\frac{1}{n}\bbZ^2$ (in nearest-neighbor sense) connecting vertices of different colors; see also Figure~\ref{fig:interface} for an illustration.} Consider the $\CLE_{\kappa'}$ percolation in $D$ constructed at the beginning of Section~\ref{subsec:one-arm}, and let $\Sigma$ be the collection of interfaces between red and blue clusters. See also~\cite[Section 2.6]{KSL22} for a detailed description.

\begin{theorem}[Theorem 4.2 of~\cite{KSL22}]
\label{thm:loop-conv}
Assuming Conjecture~\ref{conj:fk-to-cle}, we have both $\Sigma^+_n \to \Sigma$ and $\Sigma^-_n \to \Sigma$ in probability with respect to the metric $d_{\mathfrak{L}}$.
\end{theorem}

For $1 \le m \le n$, we write $\Lambda_n=[-n,n]^2 \cap \mathbb{Z}^2$ and $\Lambda_{m,n}=\Lambda_n \setminus \Lambda_m$. For a fuzzy Potts configuration sampled from $\mu$, let $A_B(m,n)$ be the event that there is a nearest-neighbor path in $\Lambda_{m,n}$ from $\partial \Lambda_m$ to $\partial \Lambda_n$ such that all vertices on this path are colored blue. The blue \emph{bulk one-arm exponent} $\alpha_B(r)$ of the fuzzy Potts model is the number such that 
\[ \mu[A_B(m,n)]=(m/n)^{\alpha_B(r)+o(1)} \quad \text{as } n/m \to \infty \,. \]
The red one-arm event $A_R(m,n)$ and exponent $\alpha_R(r)$ are defined similarly.

Observe that the complement of the blue one-arm event $A_B(m,n)$ is exactly the event that there exists a red \emph{star-neighbor} circuit surrounding $\Lambda_m$ inside $\Lambda_{m,n}$. Therefore, if we consider the outermost red/blue interface $\mathcal{L}_n$ inside $\Lambda_n$ which is an outer boundary of a red cluster that disconnects 0 from $\partial \Lambda_n$, then $A_B(m,n)=\{\mathcal{L}_n \cap \Lambda_m \neq \emptyset\}$. We set $D=\bbD$ in Theorem~\ref{thm:loop-conv} and recall the outermost interface $\mathcal{L}^o$ in $\Sigma$ constructed along the proof of Theorem~\ref{thm:one-arm}. Observe that for a loop $\eta$ surrounding the origin, if $\mathrm{diam}(\eta) \le \e/2$ then it lies in $\e \mathbb{D}$ necessarily, while a.s.\ there are only finitely many loops in $\Sigma$ with diameter larger than $\e/2$. Hence, if we denote by $f_n$ the conformal map that maps $[-n,n]^2$ to $\mathbb{D}$  {and fixes the origin}, then Theorem~\ref{thm:loop-conv} implies that $f_n(\mathcal{L}_n) \to \mathcal{L}^o$ in probability with respect to the metric $d_{\mathfrak{C}}$. One of the main inputs --- the quasi-multiplicativity, which allows us to derive discrete arm exponents from the continuum ones, is already established.

\begin{theorem}[Theorem 3.8 of~\cite{KSL22}]\label{thm:quasi-multi}
There exist universal constants $0<c<C$, such that for all $1 \le \ell \le m \le n$ we have
\[ c \cdot \mu(A_B(\ell,n)) \le \mu(A_B(\ell,m)) \cdot \mu(A_B(m,n)) \le C \cdot \mu(A_B(\ell,n)) \,. \]
\end{theorem}

The rest of the proof is now standard with discrete tools in hand.

\begin{proof}[Proof of Theorem~\ref{thm:fuzzy-potts-one-arm}]
It suffices to show that $\alpha_B(r)=\alpha_1(r)$; by symmetry we will also have $\alpha_R(r)=\alpha_1(1-r)$. Recall that $\mathcal{A}_\e=\{ \mathcal{L}^o \cap \e \mathbb{D} \neq \emptyset\}$ for $\e \in (0,1)$. Fix $\delta_0\in(0,1)$. By Theorem~\ref{thm:loop-conv} and the discussion above, for all  {sufficient large $N$}, with overwhelming probability $f_N(\mathcal{L}_N)$ lies within distance $\delta:=(\e-\e^{1+\delta_0})/2 \le (\e^{1-\delta_0}-\e)/2$ from $\mathcal{L}^o$. On this event, we have the inclusion
\[ \{ \mathcal{L}^o \cap \e^{1+\delta_0} \mathbb{D} \neq \emptyset \} \subseteq \{ \mathcal{L}_N \cap \Lambda_{\lfloor \e N \rfloor} \neq \emptyset \} \subseteq\{ \mathcal{L}^o \cap \e^{1-\delta_0} \mathbb{D} \neq \emptyset \}, \]
that is,  {$\mathcal{A}_{\e^{1+\delta_0}} \subseteq A_B( {\lfloor\e N\rfloor},N) \subseteq \mathcal{A}_{\e^{1-\delta_0}}$}. Together with Theorem~\ref{thm:one-arm}, this implies $\limsup_{N \to \infty} \mu(A_B(\lfloor\e N\rfloor,N)) \le \PP(\mathcal{A}_{\e^{1-\delta_0}})=\e^{(1-\delta_0)(\alpha_1(r)+o(1))}$ and $\liminf_{N \to \infty} \mu(A_B(\lfloor \e N\rfloor,N)) \ge \PP(\mathcal{A}_{\e^{1+\delta_0}})=\e^{(1+\delta_0)(\alpha_1(r)+o(1))}$, i.e., % Since $\delta_0>0$ is arbitrary, we conclude that for some constant $0<c_1<c_2$ uniform in $\e$,
\begin{equation}\label{eq:asymp}
 {\e^{(1+\delta_0)(\alpha_1(r)+o(1))}} \le \liminf_{N \to \infty} \mu(A_B(\lfloor\e N\rfloor,N)) \le \limsup_{N \to \infty} \mu(A_B(\lfloor \e N\rfloor ,N)) \le  {\e^{(1-\delta_0)(\alpha_1(r)+o(1))}}. % c_2 \e^{\alpha_1(r)} \,.
\end{equation}
Let $\e>0$ and $K$ be a positive integer. By Theorem~\ref{thm:quasi-multi}, for some universal constant $c>0$ we have
\[ \mu(A_B(m,\lfloor m\e^{-K}\rfloor)) \le c^K \prod_{j=1}^K \mu(A_B(\lfloor m\e^{-j+1}\rfloor,\lfloor m\e^{-j}\rfloor)) \,, \]
and thus
\[ \frac{\log \mu(A_B(m,\lfloor m\e^{-K}\rfloor ))}{\log(\e^{-K})} \le \frac{\log c}{\log (1/\e)}+\frac{1}{K \log(1/\e)} \sum_{j=1}^K \log \mu(A_B(\lfloor m\e^{-j+1}\rfloor,\lfloor m\e^{-j}\rfloor)) \,. \]
By~\eqref{eq:asymp} we have ${\limsup_{j \to \infty}} \mu(A_B(\lfloor m\e^{-j+1}\rfloor,\lfloor m\e^{-j}\rfloor)) \le  {\e^{(1-\delta_0)(\alpha_1(r)+o(1))}}$,  {and since $\delta_0>0$ is arbitrary, we have}
\[ \limsup_{K \to \infty} \frac{\log \mu(A_B(m,\lfloor m\e^{-K}\rfloor))}{\log(\e^{-K})} \le \frac{\log c}{\log (1/\e)}-\alpha_1(r) \,. \]
Since $\mu(A_B(m,n))$ is indeed decreasing in $n$ for fixed $m$, this readily implies 
\[ \limsup_{n/m \to \infty} \frac{\log \mu(A_B(m,n))}{\log(n/m)} \le \frac{\log c}{\log (1/\e)}-\alpha_1(r) \,. \]
Further let $\e \to 0$, we get $\limsup_{n/m \to \infty} \frac{\log \mu(A_B(m,n))}{\log(n/m)} \le -\alpha_1(r)$. The lower bound can be proved similarly. Therefore $\mu(A_B(m,n))=(m/n)^{\alpha_1(r)+o(1)}$ as $n/m \to \infty$, concluding the proof.
\end{proof}

\subsection{Bichromatic one-arm exponent for  {3-state Potts} model}
\label{subsec:bichromatic}
We now provide further background of the (ordinary) Potts models and explain Corollary~\ref{cor:3-Potts-bichromatic} in detail.
Let $q \in \mathbb{N}$, we focus on the $q$-state Potts model on a subgraph $G=(V,E)$ of $\mathbb{Z}^2$. A spin configuration on $G$ is an element $\upvarsigma \in \{1,\ldots,q\}^V$, to which we associate the Hamiltonian with free boundary conditions
\[ H_{G,q}(\upvarsigma)=-\sum_{u \sim v, u,v \in V} \1_{\upvarsigma(u)=\upvarsigma(v)} \,. \]
For $\upbeta \ge 0$, the $q$-state Potts model with free boundary conditions at inverse temperature $\upbeta$ is the Gibbs measure on $\{1,\ldots,q\}^V$ given by
\[ \upmu_{G,\upbeta,q}(\upvarsigma):=\frac{1}{\mathcal{Z}_{G,\upbeta,q}} \exp(-\upbeta H_{G,q}(\upvarsigma)) \,, \]
where $\mathcal{Z}_{G,\upbeta,q}$ is the partition function so that $\upmu_{G,\upbeta,q}$ is a probability measure. The Ising model corresponds to $q=2$. The critical temperature for $q$-state Potts model is $\upbeta_c(q)=-\log(1-p_c(q))=\log(1+\sqrt{q})$, where $p_c(q)$ is the critical probability of the corresponding FK$_q$ percolations. For $q \in [1,4]$, the phase transition at the critical point is continuous~\cite{DCST17}, and it is conjectured that the scaling limit of critical $q$-state Potts model can be described by simple $\CLE_\kappa$, where $\kappa=4 \arccos(-\sqrt{q}/2)/\pi \in [8/3,4]$. The convergence of critical Ising interfaces to $\CLE_3$ was proved in~\cite{BH19}.

Let $k\in\{1,\ldots,q-1\}$ and $r=k/q$. As explained in Section~\ref{subsec:discrete}, if we first sample $\upvarsigma$ according to critical $q$-state Potts measure, and then set $\sigma(v)=R$ if $\upvarsigma(v) \in \{1,\ldots,k\}$ and $\sigma(v)=B$ otherwise for each $v \in V$, then the law of $(\sigma(v))_{v\in V}$ is exactly the fuzzy Potts measure $\mu_{G,q,r}$.
Therefore, Theorem~\ref{thm:fuzzy-potts-one-arm} applies to the Ising model and  {3-state Potts} model. Figure~\ref{fig:3-Potts} illustrates a critical Potts configuration, where we treat vertices with  {spin 1 (resp.\ spins 2 and 3)} as red (resp.\ blue), respectively. The existence of a mixed $2,3$ one-arm in the critical 3-state Potts model is equivalent to the existence of a blue arm in the fuzzy Potts model with $q=3$ and $r=1/3$. Corollary~\ref{cor:3-Potts-bichromatic} follows thereby once assuming that Conjecture~\ref{conj:fk-to-cle} holds for $q=3$.

\section{Liouville quantum gravity surfaces}
\label{sec:quantum-surface}

In this section, we review the definition of Liouville quantum gravity surfaces and present some LQG surfaces that  will be used throughout this paper. In Section~\ref{subsec:lqg-q-s}, we recap the notion of quantum surfaces and recall the quantum disks and quantum triangles introduced in~\cite{DMS21,AHS21,ASY22}. In Section~\ref{subsec:pinched-thin-qa}, we introduce the definition of pinched thin quantum annulus and derive its boundary length law, which shall be used later in the proof of Theorems~\ref{thm:CR-BCLE-simple} and~\ref{thm:CR-BCLE-nonsimple} in Section~\ref{sec:proof}.

In this paper we work with non-probability measures and extend the terminology of ordinary probability to this setting. For a finite or $\sigma$-finite  measure space $(\Omega, \mathcal{F}, M)$, we say $X$ is a random variable if $X$ is an $\mathcal{F}$-measurable function with its \textit{law} defined via the push-forward measure $M_X=X_*M$. In this case, we say $X$ is \textit{sampled} from $M_X$ and write $M_X[f]$ for $\int f(x)M_X(dx)$. \textit{Weighting} the law of $X$ by $f(X)$ corresponds to working with the measure $d\widetilde{M}_X$ with Radon-Nikodym derivative $\frac{d\widetilde{M}_X}{dM_X} = f$. \textit{Conditioning} on some event $E\in\mathcal{F}$ (with $0<M[E]<\infty$) refers to the probability measure $\frac{M[E\cap \cdot]}{M[E]} $  on the  measurable space $(E, \mathcal{F}_E)$ with $\mathcal{F}_E = \{A\cap E: A\in\mathcal{F}\}$,  {while \emph{restricting} to $E$ refers to the measure $M[E\cap\cdot]$. }

\subsection{Liouville fields and quantum surfaces}\label{subsec:lqg-q-s}

We first review some background on the Gaussian free field. Let $\mathbb{H}$ be the upper half plane, and let $m$ be the uniform measure on $\partial\bbD\cap\mathbb{H}$. Define the Dirichlet inner product $\langle f,g\rangle_\nabla = (2\pi)^{-1}\int_\bbH \nabla f\cdot\nabla g $ on the space $\{f\in C^\infty( {\mathbb{H}}):\int_{ {\mathbb{H}}}|\nabla f|^2<\infty; \  \int f(z)m(dz)=0\}$, and let $H(\bbH)$ be the closure of this space w.r.t.\ the inner product $\langle f,g\rangle_\nabla$.
Let $(f_n)_{n=1}^\infty$ be an orthonormal basis of $H(\mathbb{H})$, and $(\alpha_n)_{n=1}^\infty$ be a sequence of i.i.d.\ standard Gaussian variables. Then the summation
\eqb\label{eq:def-gff}
h_{\mathbb{H}}=\sum_{n=1}^\infty \alpha_n f_n 
\eqe
converges a.s.\ in the space of distributions. We call $h_\bbH$ a \emph{Gaussian free field} on $\mathbb{H}$ with normalization $\int_{\mathbb{H}} h(z)  m(dz)=0$, and write $P_\bbH$ for its law. See~\cite[Section 4.1.4]{DMS21} for more details.

Write $|z|_+=\max(|z|,1)$. For $z,w \in \ol{\mathbb{H}}=\mathbb{H} \cup \mathbb{R}$, define
\[ G_{\mathbb{H}}(z,w)=-\log|z-w|-\log|z-\ol w|+2\log |z|_+ +2\log |w|_+ \,, ~ G_{\mathbb{H}}(z,\infty)=G_{\mathbb{H}}(\infty,z)=2\log |z|_+ \,. \]
Then $h_{\mathbb{H}}$ is a centered Gaussian free field with covariance $\mathbb{E}[h_{\mathbb{H}}(z) h_{\mathbb{H}}(w)]=G_{\mathbb{H}}(z,w)$.

Fix a parameter $\gamma \in (0,2)$ and let $Q=\frac{\gamma}{2}+\frac{2}{\gamma}$, we now introduce $\gamma$-\emph{Liouville quantum gravity} (LQG) surfaces, or \emph{quantum surfaces} for simplicity. Consider the space of pairs $(D,h)$ where $D \subseteq \mathbb{C}$ is a domain and $h$ is a distribution on $D$, we can define an equivalence relation $\sim_\gamma$ on it as follows: we say $(D,h) \sim_\gamma (\wt D,\wt h)$ if and only if there is a conformal map $g:D \to \wt D$ such that
\begin{equation}\label{eq:equiv-rel}
    \wt h=g \bullet_\gamma h := h \circ g^{-1}+Q \log |(g^{-1})'| \,.
\end{equation}
A quantum surface $S$ is an equivalence class of pairs $(D,h)$ under the equivalence relation $\sim_\gamma$, and a particular choice of such $(D,h)$ is called an \emph{embedding} of $S$. With a slight abuse of notation, we sometimes call $(D,h)$ as a quantum surface, referring to the equivalence class $(D,h)/\!\!\sim_\gamma$ it defines. We can also extend this notion to quantum surfaces with  marked points or curves. In this case, the equivalence relation $\sim_\gamma$ also requires that marked points (and their ordering) or curves are preserved by the conformal map $g$ in~\eqref{eq:equiv-rel}. For $k \ge 1$, a \emph{quantum surface with $k$ marked points} is an equivalence class of tuples $(D,h,x_1,\ldots,x_k)$, where $(D,h)$ is a quantum surface and $x_i \in \ol D$. A \emph{curve-decorated quantum surface} is an equivalence of tuples $(D,h,\eta_1,\ldots,\eta_k)$, where $(D,h)$ is a quantum surface and $\eta_i$ are curves in $\ol D$.

For a $\gamma$-quantum surface $(D,h)/\!\!\sim_{\gamma}$ embedded as $(\mathbb{H},\phi)$, its \emph{quantum area measure} $\mu_\phi$ is defined by the weak limit of $\mu_\phi^\e:=\e^{\gamma^2/2} e^{\gamma \phi_\e(z)} \dd^2 z$ as $\e \to 0$, where $\dd^2 z$ is the Lebesgue measure on $\mathbb{H}$ and $\phi_\e(z)$ is the average of $\phi(z)$ over the circle $\partial \mathcal{B}(z,\e)$. Similarly, the \emph{quantum boundary length measure} $\nu_\phi$ is defined by the weak limit of $\nu_\phi^\e:=\e^{\gamma^2/4} e^{\frac{\gamma}{2} \phi_\e(x)} \dd x$ as $\e \to 0$, where for $x \in \partial \mathbb{H}$, $\phi_\e(x)$ is the average of $\phi(x)$ over the semi-circle $\partial \mathcal{B}(x,\e) \cap \mathbb{H}$. If $\phi$ is the sum of $h_{\mathbb{H}}$ and a (possibly random) continuous function on $\ol{\mathbb{H}}$ except at finitely many points, then the weak limits $\mu_\phi$ and $\nu_\phi$ are well-defined~\cite{DS11,SW16}. If $f$ is a conformal automorphism of $\bbH$, then $f_* \mu_\phi = \mu_{f\bullet_\gamma \phi}$ and $f_* \nu_\phi = \nu_{f\bullet_\gamma \phi}$, which allows us to extend the definition of $\mu_\phi$ and $\nu_\phi$ to other domains by conformally  {mapping} to $\bbH$.

We also consider quantum surfaces with beaded domains. Let $D$ be a closed set such that each component of its interior together with its prime-end boundary is homeomorphic to the closed disk, and suppose $h$ is a generalized function on $D$.
We extend the equivalence relation $\sim_\gamma$ so that $g$ is allowed to be any homeomorphism from $D$ to $\wt D$ that is conformal on each component of the interior of $D$. A \emph{beaded quantum surface} $S$ is an equivalence class of pairs $(D,h)$ under the extended equivalence relation $\sim_\gamma$, and a particular choice of such $(D,h)$ is called an embedding of $S$. Beaded quantum surfaces with  marked points or curves can be defined analogously.

Next we define Liouville fields on $\mathbb{H}$ using $P_\bbH$.

\begin{definition}[Liouville field]
\label{def:lf}
  Fix $\gamma\in(0,2)$.  Let $(h,\mathbf{c})$ be sampled from $P_{\mathbb{H}} \times [e^{-Qc} \dd c]$, and let $\phi(z)=h(z)-2Q \log|z|_++\mathbf{c}$. We say $\phi$ is a Liouville field on $\mathbb{H}$ and write $\LF_{\mathbb{H}}$ for its law.
\end{definition}

We will also need Liouville fields with bulk or boundary insertions.

\begin{definition}
\label{def:lf-bulk-boundary}
    For $(\alpha,w) \in \mathbb{R} \times \mathbb{H}$ and $(\beta,s) \in \mathbb{R} \times \bbR$, let $(h,\mathbf{c})$ be sampled from $C_{\mathbb{H}}^{(\alpha,w),(\beta,s)} P_{\mathbb{H}} \times [e^{(\alpha+\frac{\beta}{2}-Q)c} \dd c]$, where
    \[ C_{\mathbb{H}}^{(\alpha,w),(\beta,s)}=(2 \,\mathrm{Im} w)^{-\frac{\alpha^2}{2}} |w|_+^{-2\alpha(Q-\alpha)} |s|_+^{-\beta(Q-\frac{\beta}{2})} \exp \left( \tfrac{\alpha \beta}{2} G_{\mathbb{H}}(w,s) \right) \,, \]
    and set $\phi(z)=h(z)-2Q \log |z|_+ + \alpha G_{\mathbb{H}}(z,w)+\frac{\beta}{2} G_{\mathbb{H}}(z,s)+\mathbf{c}$. We say $\phi$ is a Liouville field on $\mathbb{H}$ with insertions $(\alpha,w)$ and $(\beta,s)$, and write $\LF_{\mathbb{H}}^{(\alpha,w),(\beta,s)}$ for its law.
\end{definition}

\begin{definition}
\label{def:lf-boundary}
    Let $(\beta_i,s_i) \in \mathbb{R} \times (\mathbb{R} \cup \{\infty\})$ for $i=1,\ldots,m$, where $m \ge 1$ and  {the $s_i$ are distinct}. Assume that $s_i \neq \infty$ for $i \ge 2$. Let $(h,\mathbf{c})$ be sampled from $C_{\mathbb{H}}^{(\beta_i,s_i)_i} P_{\mathbb{H}} \times [e^{(\frac{1}{2}\sum_{i=1}^m \beta_i - Q)c} \dd c]$, where
    \begin{equation*}
        C_{\mathbb{H}}^{(\beta_i,s_i)_i}=
        \begin{cases}
            \prod_{i=1}^m |s_i|_+^{-\beta_i(Q-\frac{\beta_i}{2})} \exp \left( \sum_{j=i+1}^m \frac{\beta_i \beta_j}{4} G_{\mathbb{H}}(s_i,s_j) \right), & \mbox{ if } s_1 \neq \infty \,, \\
            \prod_{i=2}^m |s_i|_+^{-\beta_i(Q-\frac{\beta_i+\beta_1}{2})} \exp \left( \sum_{j=i+1}^m \frac{\beta_i \beta_j}{4} G_{\mathbb{H}}(s_i,s_j) \right), & \mbox{ if } s_1=\infty \,.
        \end{cases}
    \end{equation*}
    and set $\phi(z)=h(z)-2Q \log |z|_+ +\sum_{i=1}^m \frac{\beta_i}{2} G_{\mathbb{H}}(z,s_i)+\mathbf{c}$. We say $\phi$ is a Liouville field on $\mathbb{H}$ with boundary insertions $(\beta_i,s_i)_{1 \le i \le m}$, and write $\LF_{\mathbb{H}}^{(\beta_i,s_i)_i}$ for its law.
\end{definition}

Recall the space $H(\bbH)$ and its radial-lateral decomposition $H(\mathbb{H})=H_1(\mathbb{H}) \oplus H_2(\mathbb{H})$, where $H_1(\mathbb{H})$ (resp.\ $H_2(\mathbb{H})$) is the subspace of functions in $H(\bbH)$ which are constant (resp.\ have mean zero) on $\partial \mathcal{B}(0,r) \cap \mathbb{H}:=\{z \in \mathbb{H}:|z|=r\}$ for each $r>0$. Then we have the decomposition
$h_{\mathbb{H}}=h_{\mathbb{H}}^1+h_{\mathbb{H}}^2$ where $h_{\mathbb{H}}^1$ and $h_{\mathbb{H}}^2$ are projection of $h_\bbH$ on $H_1(\bbH)$ and $H_2(\bbH)$, and $h_{\mathbb{H}}^1$, $h_{\mathbb{H}}^2$ are independent.
Moreover, the constant value $\{ h_{\mathbb{H}}^1(\partial \mathcal{B}(0,e^{-t}) \cap \mathbb{H}) \}_{t \ge 0}$ is distributed as $\{B_{2t}\}_{t \ge 0}$, where $(B_{2t})_{t \ge 0}$ is a standard Brownian motion with $B_0=0$. See~\cite[Section 4.1.6]{DMS21} for details.

Now we present some typical quantum surfaces constructed from Liouville fields. We start with the (two-pointed) thick quantum disk introduced in~\cite[Section 4.5]{DMS21}. 

\begin{definition}[Thick quantum disks]
\label{def:thick-qd}
    Let $W \ge \frac{\gamma^2}{2}$ and write $\beta=\gamma+\frac{2-W}{\gamma} \le Q$. Let $(B_t)_{t \ge 0}$ be a standard Brownian motion conditioned on $B_{2t}-(Q-\beta)t<0$ for all $t>0$\footnote{This conditioning can be made sense via Bessel processes; see e.g.~\cite[Section 4.2]{DMS21}.}, and let $(\wt B_t)_{t \ge 0}$ be an independent copy of $(B_t)_{t \ge 0}$. Write
    \begin{equation*}
        Y_t=
        \begin{cases}
            B_{2t}+\beta t \,, & \mbox{ for }\,t \ge 0 \,, \\
             {\wt B_{-2t}}+(2Q-\beta)t \,,& \mbox{ for }\,t<0 \,,
        \end{cases}
    \end{equation*}
    and set $h_1(z)=Y_{-\log |z|}$ for each $z \in \mathbb{H}$. Let $h_2$ be a random generalized function  {independent from $h_1$} with the same law as $h_\bbH^2$ defined above.
    Independently sample $\mathbf{c}$ from the measure $\frac{\gamma}{2} e^{(\beta-Q)c} \dd c$, and let $\psi(z)=h_1(z)+h_2(z)+\mathbf{c}$. The infinite measure $\Md_2(W)$ is defined as the law of $(\mathbb{H},\psi,0,\infty)/\!\!\sim_\gamma$, and we call a sample from $\Md_2(W)$ a quantum disk with two marked points.

\end{definition}

Thick quantum disks of weight 2 are of special interest. In this case we have $\beta=\gamma$, so by~\cite[Proposition A.8]{DMS21}, the marked points of $\Md_2(2)$ are \emph{quantum typical}, meaning that the two marked points on the quantum disk can be resampled according to the quantum length measure. We can also sample interior marked points according to the quantum area measure. This allows us to define general quantum disks $\QD_{m,n}$ with $m$ interior and $n$ boundary quantum typical points. We present a few that will be used in this paper. For a finite measure $M$, let $M^{\#}=|M|^{-1} M$ be the probability measure proportional to $M$.

\begin{definition}\label{def:QD-mn}
    Let $\QD_{1,2}$ be the law of $(\mathbb{H},\phi,0,\infty,z)/\!\!\sim_\gamma$ where $(\mathbb{H},\phi,0,\infty)/\!\!\sim_\gamma$ is sampled from the weighted measure $ {\mu_\phi( \mathbb{H})} \Md_2(2)$ and $z$ is independently sampled from $\mu_\phi^{\#}$.  {For $(\mathbb{H},\phi,0,\infty,z)/\!\!\sim_\gamma$ sampled from $\QD_{1,2}$, let $\QD_{1,1}$ be the law of $(\mathbb{H},\phi,0,z)/\!\!\sim_\gamma$ weighted by $\nu_\phi(\partial \mathbb{H})^{-1}$, and let $\QD_{1,0}$ be the law of $(\mathbb{H},\phi,z)/\!\!\sim_\gamma$   weighted by $\nu_\phi(\partial \mathbb{H})^{-2}$.}
\end{definition}

For $W\in(0,\frac{\gamma^2}{2})$, the weight $W$ (thin) quantum disk is defined to be the concatenation of weight $\gamma^2-W$ thick quantum disks.  {See Figure~\ref{fig:det-qt} (left) for an illustration.}

\begin{definition}[Thin quantum disks]
\label{def:thin-qd}
    Let $0<W<\frac{\gamma^2}{2}$. First sample $T \sim \1_{t>0} (1-\frac{2W}{\gamma^2})^{-2} \dd t$, then sample a Poisson point process $\{(u,\mathcal{D}_u)\}$ from the measure $ {\1_{u \in [0,T]} {du}} \times \Md_2(\gamma^2-W)$. Concatenate the collection of two-pointed thick quantum disks $\{ \mathcal{D}_u \}$ according to the order induced by label $u$.  {The} obtained doubly-marked surface is called a thin quantum disk of weight $W$, and  {we} write $\Md_2(W)$ for its law.
    
    For a thin quantum disk sampled from $\Md_2(W)$, its left (resp.\ right) boundary length is defined as the sum of the left (resp.\ right) boundary lengths of all quantum disks in $\{\mathcal{D}_u\}$.
\end{definition}

The \emph{quantum triangle} is a quantum surface introduced in~\cite{ASY22} with three marked points and three weight parameters $W_1, W_2, W_3>0$. Thick quantum triangles are constructed via Liouville fields with three boundary insertions. Quantum triangles with thin vertices of weight $W<\frac{\gamma^2}{2}$ are realized as the concatenation of their thick counterparts and weight $W$ thin disks.  {See Figure~\ref{fig:det-qt} (right) for an illustration.}

\begin{definition}[Thick quantum triangles]
\label{def:thick-qt}
    Let $W_1,W_2,W_3>\frac{\gamma^2}{2}$. For $i=1,2,3$, denote $\beta_i=\gamma+\frac{2-W_i}{\gamma}<Q$. Let $\phi$ be sampled from $\frac{1}{(Q-\beta_1)(Q-\beta_2)(Q-\beta_3)} \LF_{\mathbb{H}}^{(\beta_1,\infty),(\beta_2,0),(\beta_3,1)}$. The infinite measure $\QT(W_1,W_2,W_3)$ is defined as the law of $(\mathbb{H},\phi,\infty,0,1)/\!\!\sim_{\gamma}$.

    For a quantum triangle embedded as $(D,\psi,a_1,a_2,a_3)$, let $L_{12}$ be the quantum length of the boundary arc between $a_1$ and $a_2$ (not containing $a_3$). Define $L_{23}$ and $L_{31}$ similarly.
\end{definition}

\begin{definition}[Quantum triangles with thin vertices]
\label{def:thin-qt}
    Let $W_1,W_2,W_3 \in (0,\frac{\gamma^2}{2}) \cup (\frac{\gamma^2}{2},\infty)$. Let $I$ be the set of indices $i$ with $W_i<\frac{\gamma^2}{2}$. Denote $\wt W_i=W_i$ if $i \notin I$, and $\wt W_i=\gamma^2-W_i$ otherwise. Sample $(S_0,(S_i)_{i \in I})$ from
    \[ \QT(\wt W_1,\wt W_2,\wt W_3) \times \prod_{i \in I} (1-\tfrac{2W_i}{ {\gamma^2}}) \Md_2(W_i) \,. \]
    We concatenate these quantum surfaces as follows: embed $S_0$ as $(\wt D_0, {\phi_0},\wt a_1, \wt a_2, \wt a_3)$, and for each $i \in I$, embed $S_i$ as $(\wt D_i, {\phi_i},\wt a_i,a_i)$ such that $\wt D_i$ are disjoint and $\wt D_i \cap \wt D=\{\wt a_i\}$. For each $i \notin I$, let $a_i=\wt a_i$. Let $D=\wt D_0 \cup (\cup_{i \in I} \wt D_i)$  {and $\phi|_{\wt D_j} = \phi_j$ for $j\in \{0\}\cup I$.} Then the measure $\QT(W_1,W_2,W_3)$ is defined as the law of $(D,\phi,a_1,a_2,a_3)/\!\!\sim_\gamma$, and we write
    \[ \QT(W_1,W_2,W_3)=\QT(\wt W_1,\wt W_2,\wt W_3) \times \left(  {\prod_{i \in I} (1-\tfrac{2W_i}{\gamma^2}}) \Md_2(W_i) \right) \,. \]

    The definition of boundary lengths $L_{12}$, $L_{23}$ and $L_{31}$ is the same as in the thick case. Note that boundary arcs with thin endpoints consist of the boundary arcs of thick triangles and those of thin disks.
\end{definition}

\begin{figure}
    \centering
    \begin{tabular}{cc}
         \includegraphics[scale=0.45]{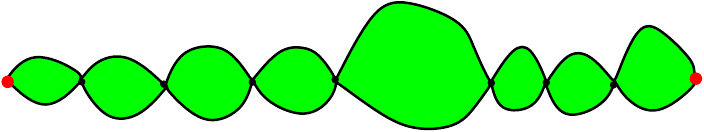}  &    \includegraphics[scale=0.41]{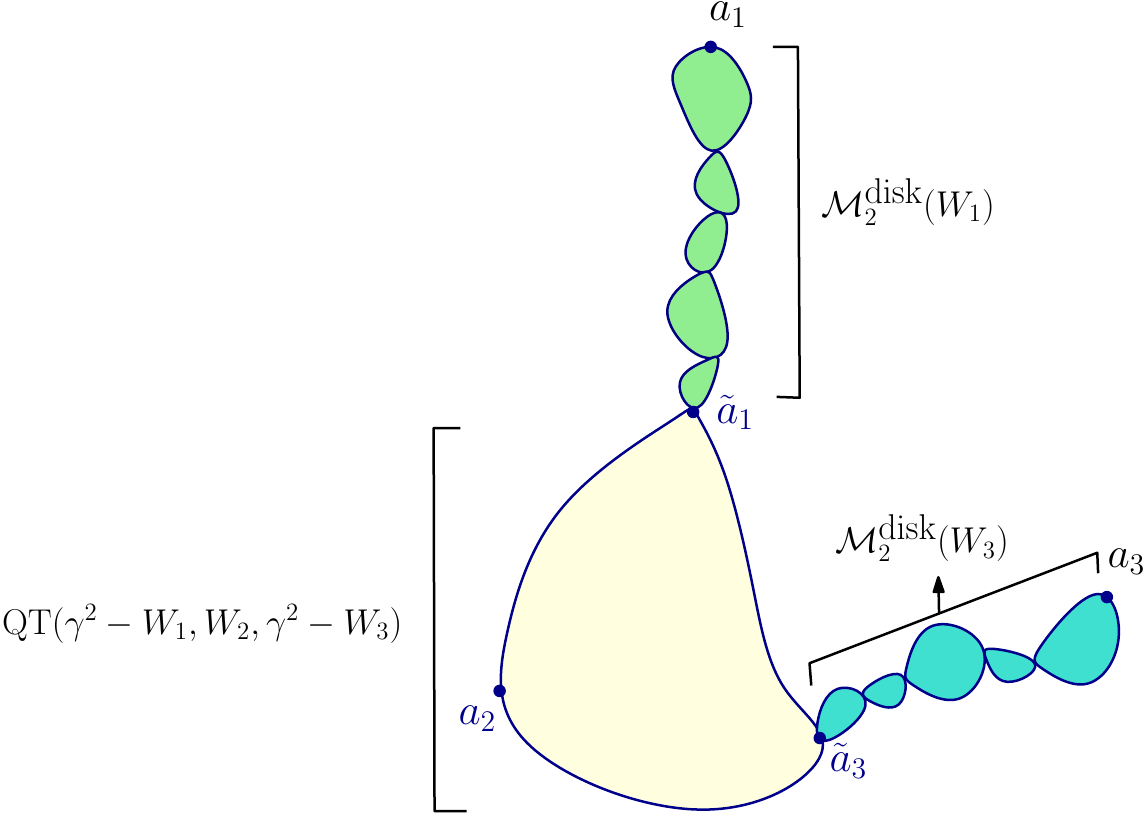}
    \end{tabular}
    \caption{\textbf{Left:} A thin quantum disk of weight $W<\frac{\gamma^2}{2}$. Note that the figure is not accurate since there are infinite number of beads near the two red marked points at the end and between each two beads. \textbf{Right:} {A quantum triangle with $W_2\geq\frac{\gamma^2}{2}$ and $W_1,W_3<\frac{\gamma^2}{2}$ embedded as $(D,\phi,a_1,a_2,a_3).$  The two thin disks (colored green) are concatenated with the thick triangle (colored yellow) at points $\tilde{a}_1$ and $\tilde{a}_3$.}}
    \label{fig:det-qt}
\end{figure}

For $W>0$, let $\Md_{2,\bullet}(W)$ be the law of the three-pointed quantum surfaces obtained by (i) sampling a quantum disk from $\Md_2(W)$ and weighting its law by the quantum length of its {left} boundary arc and (ii) sampling a marked point on the left boundary arc from the probability measure proportional to the quantum boundary length measure. By~\cite[Proposition 4.4]{AHS23}, for $W<\frac{\gamma^2}{2}$, a sample from $\Md_{2,\bullet}(W)$ can be realized as concatenation of samples from
\begin{equation}\label{eq:thin-bdy-decomp}
    (1-\tfrac{2W}{\gamma^2})^2 \Md_2(W) \times \Md_{2,\bullet}(\gamma^2-W) \times \Md_2(W) \,.
\end{equation}
By comparing~\cite[Proposition 2.18]{AHS21} and Definition~\ref{def:thick-qt}, we infer that for $W>\frac{\gamma^2}{2}$, the measure $\Md_{2,\bullet}(W)$ is a constant multiple of $\QT(W,2,W)$. Inspecting Definitions~\ref{def:thin-qt} and~\eqref{eq:thin-bdy-decomp}, this extends to $W \in (0,\frac{\gamma^2}{2})$ as well. We formalize this fact by the following lemma.

\begin{lemma}[Lemma 2.15 of~\cite{ASYZ24}]
\label{lem:QT-W2W}
    For $W \in (0,\frac{\gamma^2}{2}) \cup (\frac{\gamma^2}{2},\infty)$, we have
    \begin{equation}\label{eq:QT-W2W}
        \Md_{2,\bullet}(W)=\tfrac{\gamma(Q-\gamma)}{2} \QT(W,2,W) \,.
    \end{equation}
\end{lemma}

We also consider the case where the third point is an interior quantum typical point. For $W>0$, let $\Md_{2,\circ}(W)$ be the law of the three-pointed quantum surfaces obtained by (i) sampling a quantum disk from $\Md_{2}(W)$ and weighting its law by its quantum area and (ii) sampling a marked point from the probability measure proportional to the quantum area measure. Note from Definition~\ref{def:QD-mn} that $\Md_{2,\circ}(2)=\QD_{1,2}$. Following the same argument from~\cite[Proposition 4.4]{AHS23}, for $W<\frac{\gamma^2}{2}$, a sample from $\Md_{2,\circ}(W)$ can be realized as concatenation of samples from
\begin{equation}\label{eq:thin-interior-decom}
    (1-\tfrac{2W}{\gamma^2})^2 \Md_2(W) \times \Md_{2,\circ}(\gamma^2-W) \times \Md_2(W) \,.
\end{equation}

Let $\mathcal{M}$ be a measure on quantum surfaces, we can disintegrate $\mathcal{M}$ over the quantum lengths of the boundary arcs. For instance, we can disintegrate the two-pointed quantum disk measure $\Md_2(W)$ as
\eqb\label{eq:disint}
 \Md_2(W)=\iint_{\mathbb{R}_+^2} \Md_2(W;\ell_1,\ell_2) \dd \ell_1 \dd \ell_2 \,, 
\eqe
to obtain a measure $\Md_2(W;\ell_1,\ell_2)$ supported on the quantum disks with left boundary length $\ell_1$ and right boundary length $\ell_2$. We can also define $\Md_2(W;\ell)=\int_{\mathbb{R}_+} \Md_2(W;\ell,\ell') \dd \ell'$, the disintegration over the quantum length of the left (or right) boundary arcs.

Similarly, for each $\ell>0$ we can define a measure $\QD_{1,1}(\ell)$ supported on quantum disks with one bulk and one boundary marked point whose quantum boundary length is $\ell$, such that
\[ \QD_{1,1}=\int_{\mathbb{R}_+} \QD_{1,1}(\ell) \dd \ell \,. \]
Likewise, we can disintegrate the measure $\QT(W_1,W_2,W_3)$ as
\[ \QT(W_1,W_2,W_3)=\iiint_{\mathbb{R}_+^3} \QT(W_1,W_2,W_3;\ell_{12},\ell_{23},\ell_{31}) \dd \ell_{12} \dd \ell_{23} \dd \ell_{31} \,, \]
where $\QT(W_1,W_2,W_3;\ell_{12},\ell_{23},\ell_{31})$ is the measure supported on quantum triangles with boundary lengths $\ell_{12}$, $\ell_{23}$ and $\ell_{31}$. We can also disintegrate over one or two boundary lengths. For example, we can define
\[ \QT(W_1,W_2,W_3;\ell_{12},\ell_{23})=\int_{\mathbb{R}_+} \QT(W_1,W_2,W_3;\ell_{12},\ell_{23},\ell_{31}) \dd \ell_{31} \,, \]
and
\[ \QT(W_1,W_2,W_3;\ell_{12})=\iint_{\mathbb{R}_+^2} \QT(W_1,W_2,W_3;\ell_{12},\ell_{23},\ell_{31}) \dd \ell_{23} \dd \ell_{31} \,. \]

For a quantum surface measure $\mathcal{M}$, one can deduce the boundary length distribution via Liouville conformal field theory, where the exact coefficients come from the solvability results in~\cite{RZ22}. We first recall the double gamma function $\Gamma_b(z)$ and the double sine function $S_b(z)$ which are important in Liouville conformal field theory. Fix $b>0$. For $\mathrm{Re}(z)>0$, we define $\Gamma_b(z)$ by
\begin{equation}\label{eq:double-gamma}
    \log \Gamma_b(z)=\int_{\mathbb{R}_+} \frac{1}{t} \left( \frac{e^{-zt}-e^{(b+\frac{1}{b})t/2}}{(1-e^{-bt})(1-e^{-t/b})}-\frac{(\frac{1}{2}(b+\frac{1}{b})-z)^2}{2} e^{-t} - \frac{\frac{1}{2}(b+\frac{1}{b})-z}{t} \right) \dd t \,,
\end{equation}
and it satisfies the shift relations
\begin{equation}\label{eq:double-gamma-shift}
    \Gamma_b(z+b)=\sqrt{2\pi} \frac{b^{bz-\frac{1}{2}}}{\Gamma(bz)} \Gamma_b(z) \,, \quad \Gamma_b(z+\tfrac{1}{b})=\sqrt{2\pi} \frac{(\frac{1}{b})^{\frac{1}{b} z-\frac{1}{2}}}{\Gamma(\frac{1}{b} z)} \Gamma_b(z) \,,
\end{equation}
which allow us to extend $\Gamma_b(z)$ to a meromorphic function on $\mathbb{C}$. The double sine function is defined by
\begin{equation}\label{eq:double-sine}
    S_b(z)=\frac{\Gamma_b(z)}{\Gamma_b(b+\frac{1}{b}-z)} \,,
\end{equation}
which is also meromorphic on $\mathbb{C}$ and satisfies the shift relations $S_b(z+b)=2\sin(\pi bz) S_b(z)$ and $S_b(z+\frac{1}{b})=2\sin(\pi \frac{1}{b} z) S_b(z)$.
Moreover, for $0<\mathrm{Re}(z)<b+\frac{1}{b}$, we can deduce from~\eqref{eq:double-gamma} that
\begin{equation}\label{eq:double-sine-int}
    \log S_b(z)=\int_0^\infty \frac{1}{t} \left( \frac{\sinh((\frac{1}{2}(b+\frac{1}{b})-z)t)}{2\sinh(\frac{b}{2}t) \sinh(\frac{1}{2b}t)}-\frac{b+\frac{1}{b}-2z}{t} \right) \dd t \,.
\end{equation}
We now introduce the boundary Liouville coefficients $\ol R$ and $\ol H$. For this sake, we will only use the notion of $\Gamma_{\gamma/2}$ and $S_{\gamma/2}$. Following the notation of~\cite{RZ22}, for $\mu_1, \mu_2>0$, let $\sigma_j \in \mathbb{C}$ such that $\mu_j=e^{i\pi \gamma (\sigma_j-\frac{Q}{2})}$ and $\mathrm{Re}(\sigma_j)=\frac{Q}{2}$ for $j=1,2$. Let
\begin{equation}\label{eq:refl-coef-norm-1}
    \ol R(\beta,\mu_1,\mu_2)=\frac{(2\pi)^{\frac{2}{\gamma}(Q-\beta)-\frac{1}{2}} (\frac{2}{\gamma})^{\frac{\gamma}{2}(Q-\beta)-\frac{1}{2}}}{(Q-\beta) \Gamma(1-\frac{\gamma^2}{4})^{\frac{2}{\gamma}(Q-\beta)}} \cdot \frac{\Gamma_{\gamma/2}(\beta-\frac{\gamma}{2}) e^{i \pi (\sigma_1+\sigma_2-Q)(Q-\beta)}}{\Gamma_{\gamma/2}(Q-\beta) S_{\gamma/2}(\frac{\beta}{2}+\sigma_2-\sigma_1) S_{\gamma/2}(\frac{\beta}{2}+\sigma_1-\sigma_2)} \,,
\end{equation}
and for $\mu>0$, let
\begin{equation}\label{eq:refl-coef-norm-2}
    \ol R(\beta,\mu,0)=\ol R(\beta,0,\mu)=\mu^{\frac{2}{\gamma}(Q-\beta)} \frac{(2\pi)^{\frac{2}{\gamma}(Q-\beta)-\frac{1}{2}} (\frac{2}{\gamma})^{\frac{\gamma}{2}(Q-\beta)-\frac{1}{2}}}{(Q-\beta) \Gamma(1-\frac{\gamma^2}{4})^{\frac{2}{\gamma}(Q-\beta)}} \cdot \frac{\Gamma_{\gamma/2}(\beta-\frac{\gamma}{2})}{\Gamma_{\gamma/2}(Q-\beta)} \,.
\end{equation}
The function $\ol R(\beta,\mu_1,\mu_2)$ is called the normalized reflection coefficient. Its unnormalized version is
\begin{equation}\label{eq:refl-coef}
    R(\beta,\mu_1,\mu_2)=-\Gamma(1-\tfrac{2}{\gamma}(Q-\beta)) \ol R(\beta,\mu_1,\mu_2) \,,
\end{equation}
and it is shown in~\cite[Lemma 3.4]{RZ22} that the reflection identity
\begin{equation}\label{eq:refl-iden}
    R(\beta,\mu_1,\mu_2) R(2Q-\beta,\mu_1,\mu_2)=1
\end{equation}
holds. Let $\ol \beta=\beta_1+\beta_2+\beta_3$, we define
\begin{equation}\label{eq:H-func}
    \ol H_{(0,1,0)}^{(\beta_1,\beta_2,\beta_3)}=\frac{(2\pi)^{\frac{2Q-\ol \beta}{\gamma}+1} (\frac{2}{\gamma})^{(\frac{\gamma}{2}-\frac{2}{\gamma})(Q-\frac{\ol \beta}{2})-1}}{\Gamma(1-\frac{\gamma^2}{4})^{\frac{2Q-\ol \beta}{\gamma}} \Gamma(\frac{\ol \beta-2Q}{\gamma})} \cdot \frac{\Gamma_{\gamma/2}(\frac{\ol \beta}{2}-Q) \Gamma_{\gamma/2}(\frac{\ol \beta}{2}-\beta_2) \Gamma_{\gamma/2}(\frac{\ol \beta}{2}-\beta_1) \Gamma_{\gamma/2}(Q-\frac{\ol \beta}{2}+\beta_3)}{\Gamma_{\gamma/2}(Q) \Gamma_{\gamma/2}(Q-\beta_1) \Gamma_{\gamma/2}(Q-\beta_2) \Gamma_{\gamma/2}(\beta_3)} \,.
\end{equation}
We can now describe the boundary length distribution of two-pointed quantum disks and quantum triangles in terms of boundary Liouville functions $\ol R$ and $\ol H$.

\begin{proposition}[Lemma 3.3 and Proposition 3.6 of~\cite{AHS21}]\label{prop:qd-length}
    For $W \in (0,\gamma Q)$ and $\beta=\gamma+\frac{2-W}{\gamma} \in (\frac{\gamma}{2},Q+\frac{\gamma}{2})$. Let $L_1$ and $L_2$ be the left and right boundary lengths of a quantum disk from $\Md_2(W)$. If $\mu_1, \mu_2 \ge 0$ and $\mu_1+\mu_2>0$, then the law of $\mu_1 L_1+\mu_2 L_2$ is given by
    \begin{equation}\label{eq:qd-length}
        \1_{\ell>0} \ol R(\beta,\mu_1,\mu_2) \ell^{-\frac{2W}{\gamma^2}} \dd \ell \,.
    \end{equation}
\end{proposition}

\begin{proposition}[Proposition 2.24 of~\cite{ASY22}]
\label{prop:qt-length}
    Let $W_1,W_2 \in (0,\frac{\gamma^2}{2}) \cup (\frac{\gamma^2}{2},\infty)$ and $W_3>\frac{\gamma^2}{2}$. For $i=1,2,3$, denote $\beta_i=\gamma+\frac{2-W_i}{\gamma}$, and set $\ol \beta=\beta_1+\beta_2+\beta_3$. Further let $\wt \beta_i=\beta_i$ if $W_i>\frac{\gamma^2}{2}$, and $\wt \beta_i=2Q-\beta_i$ otherwise. Suppose that $(\wt \beta_1, \wt \beta_2, \wt \beta_3)$ satisfies the constraint
    \begin{equation}\label{eq:H-bound}
        \wt \beta_1, \wt \beta_2 <Q, ~ |\wt \beta_1 - \wt \beta_2|<\wt \beta_3, \mbox{ and }~ \wt \beta_1 + \wt \beta_2 + \wt \beta_3>\gamma \,,
    \end{equation}
    then the boundary length $L_{12}$ of a quantum triangle from $\QT(W_1,W_2,W_3)$ has law
    \begin{equation}\label{eq:qt-length}
        \1_{\ell>0} \frac{2}{\gamma(Q-\beta_1)(Q-\beta_2)(Q-\beta_3)} \ol H_{(0,1,0)}^{(\beta_1,\beta_2,\beta_3)} \ell^{\frac{\ol \beta-2Q}{\gamma}-1} \dd \ell \,.
    \end{equation}
\end{proposition}

\subsection{Pinched thin quantum annulus}\label{subsec:pinched-thin-qa}
In this section we define the pinched thin quantum annulus via the thin quantum disks and prove some of its basic properties. We shall also describe its boundary length law in Proposition~\ref{prop:formula-qa-weight}.
\begin{definition}\label{def:QA}
    For $W\in(0,\frac{\gamma^2}{2})$, define the measure $\wt\QA(W)$ on beaded surfaces as follows. Take $T\sim (1-\frac{2W}{\gamma^2})^{-2}t^{-1}\mathds{1}_{t>0}dt$, and let $S_T:=\{z\in\bbC:|z| = \frac{T}{2\pi}\}$ be the circle with perimeter $T$. Then sample a Poisson point process $\{(u,\cD_u)\}$ from the measure $\mathrm{Leb}_{S_T}\times\Md_2(\gamma^2-W)$, and concatenate the $\cD_u$ to get a cyclic chain according to the ordering induced by $u$. We call a sample from $\wt\QA(W)$ a \emph{pinched quantum annulus of weight $W$}, and call $T$ its quantum cut point measure.
\end{definition}
 {The reason we added a tilde on $\QA$ in the definition above is that, we believe that it is possible to construct a weight $W$ quantum annulus without pinched points, and the versions of quantum annulus in e.g.~\cite{AS21,ARS22} are doubly connected and have no pinched points.}
Comparing with Definition~\ref{def:thin-qd}, the pinched quantum annulus of weight $W$ can be defined alternatively by (i) sampling a two-pointed thin quantum disk $\cD$ from $\Md_2(W)$ and weighting by $T^{-1}$ where $T$ is the total cut point measure of $\cD$ and (ii) gluing the two endpoints of $\cD$ together.

We have the following analog of~\cite[Lemma 4.1]{AHS23}.
\begin{lemma}\label{lem:QApinchedd}
    Fix $T>0$ and $W\in(0,\frac{\gamma^2}{2})$. The following three procedures yield the same measure on $\gamma$-LQG quantum surfaces.
    \begin{enumerate}[(i)]
        \item Sample $\cD'$ from $\wt\QA(W)$ conditioned on having quantum cut point measure $T$ (i.e., cyclically concatenate the surfaces of a  Poisson point process on  $\mathrm{Leb}_{S_T}\times\Md_2(\gamma^2-W)$). Then take a point from the quantum length measure on the outer boundary of $\cD'$ (this induces a weighting by the outer boundary length).
        \item Sample $\cD'$ from $\wt\QA(W)$ conditioned on having quantum cut point measure $T$, then independently take $(u,\cD^\bullet)\sim T^{-1}\mathrm{Leb}_{S_T}\times \Md_{2,\bullet}(\gamma^2-W)$. Insert $\cD^\bullet$ to $\cD'$ at cut point location $u$.
        \item Take $\cD^\bullet\sim \Md_{2,\bullet}(\gamma^2-W)$. Independently sample a thin quantum disk $\cD''$ from $\Md_2(W)$ conditioned on having cut point measure $T$. Concatenate the two endpoints $\cD''$ with the two endpoints of $\cD^\bullet$.
    \end{enumerate}
\end{lemma}

\begin{proof}
    The equivalence of the first two procedures above follows from~\cite[Lemma 4.1]{PPY92} applied to the Poisson point process on  $\mathrm{Leb}_{S_T}\times\Md_2(\gamma^2-W)$. The second and third procedures follows from the fact that for $u\in S_T$, a Poisson point process on $\mathrm{Leb}_{S_T\backslash\{u\}}\times\Md_2(\gamma^2-W)$ can be naturally identified with $\mathrm{Leb}_{(0,T)}\times\Md_2(\gamma^2-W)$ while the Poisson point process on $\mathrm{Leb}_{ \{u\}}\times\Md_2(\gamma^2-W)$ is a.s.\ empty.
\end{proof}

By comparing (i) with (iii) we immediately have the following.
\begin{corollary}\label{cor:QApinched}
    The quantum surfaces constructed from the following two procedures have the same law.
    \begin{enumerate}[(i)]
        \item Take a pinched quantum annulus of weight $W$ from $\wt\QA(W)$ and weight its law by its outer boundary length. Then take a marked point on the outer boundary according to the quantum length measure. Denote its law  by $\wt\QA_{1}(W)$.
        \item Take a pair of quantum surfaces from $\Md_{2,\bullet}(\gamma^2-W)\times\Md_2(W)$ and cyclically concatenate them.
    \end{enumerate}
\end{corollary}

We now calculate the boundary length distributions of the quantum annulus. We begin with an integral formula on the reflection coefficient.

\begin{figure}
    \centering
    \includegraphics[scale=0.6]{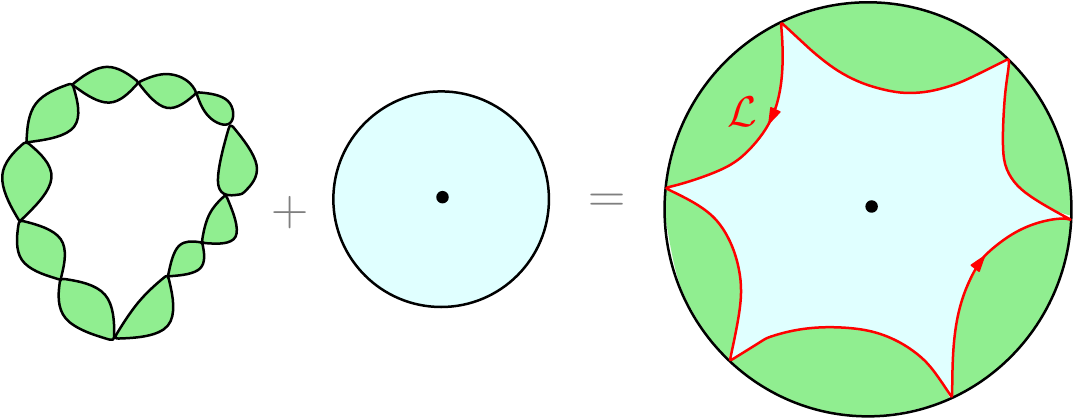}
    \caption{An illustration of pinched quantum annulus of  weight $W$. In Theorem~\ref{thm:weld-BCLE} we prove that when glued to the blue disk from $\QD_{1,0}$, we get another sample from $\QD_{1,0}$ decorated with an independent (counterclockwise) BCLE loop  $\cL$.}
    \label{fig:BCLEzipper-simple}
\end{figure}

\begin{lemma}\label{lem:refl-diff}
    Let $\beta \in (\frac{\gamma}{2},Q)$ so that the reflection coefficient $R(\beta,\mu_1,\mu_2)$ is well-defined. Then
    \[ \frac{\partial \log R(\beta;\mu_1,\mu_2)}{\partial \mu_1}=\frac{Q-\beta}{\gamma \mu_1}+\frac{1}{\pi \gamma \mu_1}\int_{\mathbb{R}_+} \frac{\sinh(\frac{Q-\beta}{2}t)\sin(\frac{\log \mu_1-\log \mu_2}{\pi \gamma}t)}{\sinh(\frac{\gamma}{4}t) \sinh(\frac{1}{\gamma}t)} \dd t \,. \]
\end{lemma}

\begin{proof}
    Recall from Equations~\eqref{eq:refl-coef-norm-1} and~\eqref{eq:refl-coef} that
    \begin{align}\label{eq:log-reflection}
    \log R(\beta;\mu_1,\mu_2)&= C(\gamma,\beta)+\frac{Q-\beta}{\gamma}(\log \mu_1+\log \mu_2) \nonumber \\
    &\quad -\log S_{\gamma/2} \left(\frac{\beta}{2}+\frac{\log \mu_1-\log \mu_2}{\pi \gamma} i \right)-\log S_{\gamma/2} \left(\frac{\beta}{2}-\frac{\log \mu_1-\log \mu_2}{\pi \gamma} i \right) \,,   
    \end{align}
    where $C(\gamma,\beta)$ is a constant depending only on $\gamma$ and $\beta$. The integral formula~\eqref{eq:double-sine-int} implies that
    \begin{align*}
    &\quad \log S_{\gamma/2} \left(\frac{\beta}{2}+\frac{\log \mu_1-\log \mu_2}{\pi \gamma} i \right)+\log S_{\gamma/2} \left(\frac{\beta}{2}-\frac{\log \mu_1-\log \mu_2}{\pi \gamma} i \right) \\
    &=\int_{\mathbb{R}_+} \frac{1}{t} \left( \frac{2\sinh(\frac{Q-\beta}{2}t)\cosh(\frac{\log \mu_1-\log \mu_2}{\pi \gamma} i t)}{2\sinh(\frac{\gamma}{4}t) \sinh(\frac{1}{\gamma}t)}-\frac{2(Q-\beta)}{t} \right) \dd t \,.
    \end{align*}
    Substituting it into~\eqref{eq:log-reflection} and differentiating with respect to $\mu_1$ yields the lemma,   {where the differentiation under the integral is valid by dominated convergence theorem}.
\end{proof}

\begin{proposition}\label{prop:formula-qa-weight}
    For $\gamma \in (\sqrt{2},2)$ and $W \in (0,\frac{\gamma^2}{2})$, let $L_1$ and $L_2$ be the inner and outer boundary lengths of a thin quantum annulus from $\wt\QA(W)$. Then for any $t \in \mathbb{R}_+$ and $y \in (-1,0)$,
     we have
    \begin{equation}\label{eq:qa-length}
        \wt\QA(W)[L_1 e^{-tL_1} L_2^y] = t^{-y-1}\Gamma(y+1) (1-\tfrac{2}{\gamma^2}W)^{-2} \frac{\sin(\frac{\gamma^2 - 2W}{4}\pi y)}{\sin(\frac{\gamma^2}{4} \pi y)} \,.
    \end{equation}
\end{proposition}

\begin{proof}
     For fixed $T>0$, let $\Pi_T$ be the Poisson point process $\{(u,\mathcal{D}_u)\}$ from the measure $\Leb_{S_T} \times \Md_2(\gamma^2-W)$. Then following the argument from~\cite[Proposition 3.6]{AHS21},
     for all $\mu_1, \mu_2 \ge 0$ and $\mu_1+\mu_2>0$,
     \[ \mathbb{E}[\exp(-\sum_{(u,\mathcal{D}_u) \in \Pi_T} (\mu_1 \ell_1(\mathcal{D}_u)+\mu_2 \ell_2(\mathcal{D}_u)))]=\exp \left(-T \cdot \frac{\gamma}{2(Q-\beta)} R(2Q-\beta;\mu_1,\mu_2) \right) \,, \]
     where $\beta=\gamma+\frac{2-W}{\gamma} \in (Q,Q+\frac{\gamma}{2})$, and $\ell_1(\mathcal{D}_u),\ell_2(\mathcal{D}_u)$ are the quantum lengths of the boundary arcs of $\mathcal{D}_u$. Differentiating with respect to $\mu_1$, and then integrate against $\1_{T>0} (1-\frac{2W}{\gamma^2})^{-2} T^{-1} \dd T$, we get
    \begin{align}\label{eq:QA-lap}
    &\quad \wt\QA(W)[L_1 e^{-\mu_1 L_1-\mu_2 L_2}] \nonumber \\
    &=(1-\tfrac{2}{\gamma^2}W)^{-2} \cdot \frac{\gamma}{2(Q-\beta)} \cdot \frac{\partial R(2Q-\beta;\mu_1,\mu_2)}{\partial \mu_1} \int_{\mathbb{R}_+} \exp \left(-T \cdot \frac{\gamma}{2(Q-\beta)} R(2Q-\beta;\mu_1,\mu_2) \right) \dd T \nonumber \\
    &=(1-\tfrac{2}{\gamma^2}W)^{-2} \frac{\partial R(2Q-\beta;\mu_1,\mu_2)}{\partial \mu_1} \cdot \frac{1}{R(2Q-\beta;\mu_1,\mu_2)}=(1-\tfrac{2}{\gamma^2}W)^{-2} \frac{\partial \log R(2Q-\beta;\mu_1,\mu_2)}{\partial \mu_1} \,.
    \end{align}
    For any $y<0$, we have $L_2^y=\frac{1}{\Gamma(-y)} \int_{\mathbb{R}_+} \mu^{-y-1} e^{-\mu L_2} \dd \mu$. Together with~\eqref{eq:QA-lap} and Lemma~\ref{lem:refl-diff} (where $\beta$ is replaced by $2Q-\beta$), we have
    \begin{align}\label{eq:QA-Laplace}
    &\quad \wt\QA(W)[L_1 e^{-t L_1} L_2^y] = \frac{1}{\Gamma(-y)} \int_{\mathbb{R}_+} \mu^{-y-1} \wt\QA(W)[L_1 e^{-t L_1-\mu L_2}] \dd \mu \nonumber \\
    &= \frac{1}{\Gamma(-y)} \int_{\mathbb{R}_+} \mu^{-y-1} (1-\tfrac{2}{\gamma^2}W)^{-2} \frac{\partial \log R(2Q-\beta;t,\mu)}{\partial t} \dd \mu \nonumber \\
    &= (1-\tfrac{2}{\gamma^2}W)^{-2} \frac{1}{\Gamma(-y)} \int_{\mathbb{R}_+} \mu^{-y-1} \left( \frac{\beta-Q}{\gamma t}+\frac{1}{\pi \gamma t} \int_{\mathbb{R}_+} \frac{\sinh(\frac{\beta-Q}{2} v) \sin(\frac{\log t-\log \mu}{\pi \gamma} v)}{\sinh(\frac{\gamma}{4} v) \sinh(\frac{1}{\gamma} v)} \dd v \right) \dd \mu \nonumber \\
    &= (1-\tfrac{2}{\gamma^2}W)^{-2} \frac{t^{-y-1}}{\Gamma(-y)} \int_{\mathbb{R}} e^{-ry} \left( \frac{\beta-Q}{\gamma}+\int_{\mathbb{R}_+} \frac{\sinh(\frac{\gamma(\beta-Q)}{2} \pi x) \sin(-rx)}{\sinh(\frac{\gamma^2}{4} \pi x) \sinh(\pi x)} \dd x \right) \dd r
    \end{align}
    where the fourth line follows by a change of variables $\mu=t e^r$ and $v=\gamma \pi x$. We denote
    \[ F(z)=\frac{\sin(\frac{\gamma(\beta-Q)}{2} \pi z)}{\sin(\frac{\gamma^2}{4} \pi z) \sin(\pi z)} \,, \quad z \in \mathbb{C} \,, \]
    then $F$ is a meromorphic function on $\mathbb{C}$ with all simple poles lying on the real line. We claim that the double integral in \eqref{eq:QA-Laplace} equals $-\pi F(y)$ for $-1<y<0$.
    The proof is then completed by noting that $\frac{\gamma(\beta-Q)}{2}=\frac{\gamma^2-2W}{4}$ and recalling the identity $\Gamma(-y)\Gamma(y+1)=-\frac{\pi}{\sin(\pi y)}$.
    
    For a fixed small $\e \in (0,1)$, let
    \[ f(r)=\frac{1}{2\pi i}\lim_{T \to \infty} \int_{-\e-iT}^{-\e+iT} F(z) e^{rz} \dd z \,,\]
    then the two-sided Laplace transform of $f(r)$ is $F(z)$ inside the strip $-1<\mathrm{Re}(z)<0$. We will show that
    \begin{equation}\label{eq:contour-integral}
        \frac{\beta-Q}{\gamma}+\int_{\mathbb{R}_+} \frac{\sinh(\frac{\gamma(\beta-Q)}{2} \pi x) \sin(-rx)}{\sinh(\frac{\gamma^2}{4} \pi x) \sinh(\pi x)} \dd x=-\pi f(r) \,,
    \end{equation}
    and this proves the previous claim since $ \int_{\mathbb{R}} e^{-ry} f(r) \dd r=  F(y)$ for $-1<y<0$.
    We start by noting that
    \begin{align*}
    &\quad \int_{\mathbb{R}_+} \frac{\sinh(\frac{\gamma(\beta-Q)}{2} \pi x) \sin(-rx)}{\sinh(\frac{\gamma^2}{4} \pi x) \sinh(\pi x)} \dd x = \int_{\mathbb{R}_+} i F(ix) \sin(-rx) \dd x \\
    &= \lim_{T \to \infty} \lim_{\delta \to 0} \int_\delta^T i F(ix) \frac{e^{-irx}-e^{irx}}{2i} \dd x = -\frac{1}{2i} \lim_{T \to \infty}  \lim_{\delta \to 0} \int_{[-T,-\delta] \cup [\delta,T]} i F(ix) e^{irx} \dd x \\
    &= -\frac{1}{2i} \lim_{T \to \infty}  \lim_{\delta \to 0} \int_{[-iT,-i\delta] \cup [i\delta,iT]} F(z) e^{rz} \dd z \,.
    \end{align*}
    Let $\Gamma$ be the boundary of $[-\e,0] \times [-iT,iT] \setminus \delta \mathbb{D}$ with counterclockwise orientation. Since $z \mapsto F(z)e^{rz}$ is holomorphic in the domain enclosed by $\Gamma$, we have $\int_\Gamma F(z)e^{rz} \dd z=0$. Observe that the integrals along $[-\e,0]+iT$ and $[-\e,0]-iT$ are negligible as $T \to \infty$, and the integral along the left half-circle $\Gamma \cap \partial(\delta \mathbb{D})$ converges to $ -\pi i \cdot {\rm Res}(F(z)e^{rz},0)=-\pi i \cdot \frac{2(\beta-Q)}{\gamma \pi}$ as $\delta \to 0$. Therefore, we have
    \[ \lim_{T \to \infty} \int_{-\e-iT}^{-\e+iT} F(z) e^{rz} \dd z = \lim_{T \to \infty}  \lim_{\delta \to 0} \int_{[-iT,-i\delta] \cup [i\delta,iT]} F(z) e^{rz} \dd z - \pi i \cdot \frac{2(\beta-Q)}{\gamma \pi} \,, \]
    which is exactly \eqref{eq:contour-integral}, as desired.
\end{proof}

\section{Simple BCLE loop from conformal welding}\label{sec:welding-simple}
For a pair of certain quantum surfaces, as discovered in ~\cite{She16a,DMS21}, there exists a way to \emph{conformally weld} them together according to the length measure provided that the interface lengths agree.  {To be more precise, for two quantum surfaces $S_1$ and $S_2$, suppose for $i=1,2$, $B_i=[\mathbf{p}_i,\mathbf{q}_i]$ is a boundary segment of $S_i$ such that $B_1$ and $B_2$ have the same quantum length. Then the conformal welding of $S_1$ and $S_2$ is defined to be the gluing of $S_1$ and $S_2$ along the arcs $B_1$ and $B_2$, where we identify $\mathbf{p}_1$ (resp.\ $\mathbf{q}_1$)  with $\mathbf{p}_2$ (resp.\ $\mathbf{q}_2$) , and identify the other points on $B_1$ and $B_2$ according to the LQG length measure. For the case where $B_1$ and $B_2$ have the topology of circles and no marked points, we define the \emph{uniform conformal welding} of $S_1$ and $S_2$ by first sampling $\mathbf{p}_i\in B_i$ according to the probability measure proportional to the LQG length measure, and then gluing $S_1$ and $S_2$ together along the arcs $B_1$ and $B_2$ according to the LQG length measure where we identify $\mathbf{p}_1$ with $\mathbf{p}_2$ and keep the orientations of $B_i$.} See e.g.~\cite[Section 4.1]{AHS21}, ~\cite[Section 4.1]{ASY22} and {~\cite[Section 3.1]{AS21}} for more explanation. 
In this and the following section, we realize the BCLE loops as interfaces under the  {uniform} conformal welding of a (pinched thin) quantum annulus and a quantum disk. The goal of this section is to  prove Theorem~\ref{thm:weld-BCLE} below, which concerns the $\kappa\in(2,4)$ regime, while the $\kappa'\in(4,8)$ case is postponed to the next section.

Let $\kappa\in(2,4)$, $\rho\in(-2,\kappa-4)$ and $\gamma = \sqrt\kappa$. Let $\Gamma$ be a $\mathrm{BCLE}^{\circlearrowright}_\kappa(\rho)$ collection on the unit disk and $\cL$ be the loop surrounding the origin. Let $\mu_{\cL}$ be the law of $\cL$, and  $\mu^{\circlearrowleft}$ (resp.\ $\mu^{\circlearrowright}$) be the \emph{restriction} of $\mu_{\cL}$ on the event that $\cL$ is a counterclockwise false (resp.\ clockwise true) loop. Then we have the following:
\begin{theorem}\label{thm:weld-BCLE}
Let $\kappa\in(2,4)$, $\rho\in(-2,\kappa-4)$ and $\gamma = \sqrt\kappa$. Let $C_0$, $C_1$, $C_2$ be the constants defined in~\eqref{eq:constant-simple-A},  $\ol C_1$ and $\ol C_2$ be the constants defined in~\eqref{eq:constant-simple-B}. Let
\begin{align}
    & C^{\circlearrowleft} = (1-\frac{2(\rho+2)}{\gamma^2})^2C_0C_1\ol C_1^{-2},\label{eq:thm-weld-const-A}\\&
        C^{\circlearrowright} = (1-\frac{2(\kappa-4-\rho)}{\gamma^2})^2C_0C_2\ol C_2^{-2}.\label{eq:thm-weld-const-B}
\end{align}
Then we have
\begin{align}
   & \QD_{1,0}\otimes \mu^{\circlearrowleft} =  C^{\circlearrowleft}\int_0^\infty \wt\QA(\kappa-4-\rho;\ell)\times \ell\,\QD_{1,0}(\ell)\,d\ell; \label{eq:thm-weld-A}\\&  \QD_{1,0}\otimes \mu^{\circlearrowright} = C^{\circlearrowright}\int_0^\infty \wt\QA(\rho+2;\ell)\times \ell\,\QD_{1,0}(\ell)\,d\ell. \label{eq:thm-weld-B}
\end{align}
\end{theorem}
 {See the right panel of Figure~\ref{fig:BCLEzipper-simple} for an illustration. Here the left hand sides of ~\eqref{eq:thm-weld-A} and~\eqref{eq:thm-weld-B} mean drawing independent curves sampled from $\mu^{\circlearrowleft}$ and $\mu^{\circlearrowright}$ on the quantum disks, while  the right hand sides of~\eqref{eq:thm-weld-A} and~\eqref{eq:thm-weld-B} are understood as the uniform conformal welding of the quantum disk and the pinched quantum annulus defined as above. Through out the rest of the paper, we will constantly use this kind of language, where the left hand side of a equation is understood as a measure describing drawing a curve on a quantum surface, while the right hand side of the same equation stands for the measure describing the conformal welding of several quantum surfaces.}

 In Section~\ref{subsec:conf-qs}, we review the conformal welding of quantum disks and quantum triangles in~\cite{AHS23,ASY22}, while in Section~\ref{subsec:weld-simple-pf}, we  prove Theorem~\ref{thm:weld-BCLE}. 

\subsection{Conformal welding of quantum surfaces}\label{subsec:conf-qs}
Given a measure $\mathcal{M}$ on the space of quantum surfaces (possibly with marked points) and a conformally invariant measure $\mathcal{P}$ on curves, we write $\mathcal{M}\otimes\mathcal{P}$ for the law of curve decorated quantum surface described by sampling $(S,\eta)$ from  $\mathcal{M}\times\mathcal{P}$ and then drawing $\eta$ on top of $S$. More concretely, for a domain $\cD = (D,z_1,\ldots,z_n)$ with marked points, assume that for $\phi$ sampled from  some measure $\mathcal{M}_{\cD}$, $(D,\phi,z_1,\ldots,z_n)/{\sim_\gamma}$ has the law $\mathcal{M}$. Let $\mathcal{P}_{\cD}$ be the measure $\mathcal{P}$ on the domain $\cD$,  and suppose for any conformal map $f$ one has $\mathcal{P}_{f\circ\cD} = f\circ \mathcal{P}_{\cD}$, i.e., $\mathcal{P}$ is invariant under conformal maps. Then $\mathcal{M}\otimes\mathcal{P}$ is defined by  $(D,\phi,\eta,z_1,\ldots,z_n)/{\sim_\gamma}$ for $\eta\sim\mathcal{P}_{\cD}$. This notion is well-defined for the quantum surfaces and SLE-type curves considered in this paper.

We begin with the conformal welding of two quantum disks.  {See Figure~\ref{fig-qt-weld} (left) for an illustration.}
\begin{theorem}[Theorem 2.2 of \cite{AHS23}]\label{thm:disk-welding}
    Let $\gamma\in(0,2), \kappa=\gamma^2$ and $W_1,W_2>0$. Then there exists a constant  {$C=C(\gamma;W_1;W_2)\in(0,\infty)$} such that
    \eqb\label{eq:disk-welding}
 \Md_{2}(W_1+W_2)\otimes \SLE_{ \kappa}(W_1-2;W_2-2) =  {C}\, \int_0^\infty\Md_{2}(W_1;\ell)\times\Md_{2}(W_2;\ell)\,d\ell.   
\eqe
\end{theorem}

\begin{figure}
	\centering
	\begin{tabular}{ccc} 
			\includegraphics[scale=0.66]{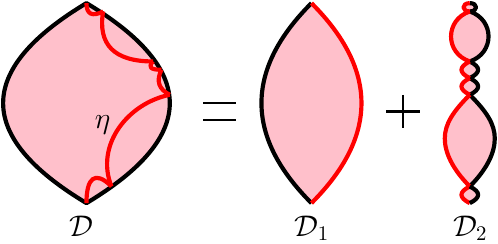}
			& & \ \ \ \
			\includegraphics[scale=0.47]{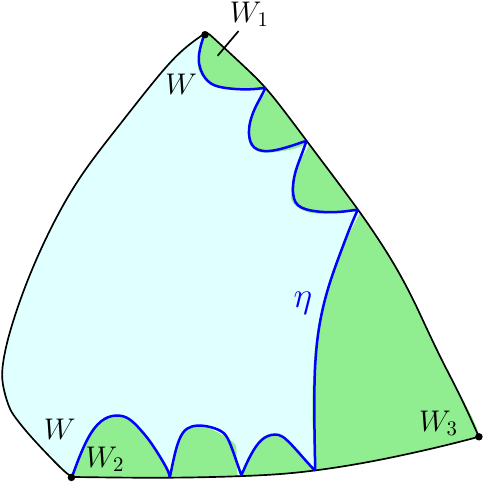}
		\end{tabular}
	\caption{\textbf{Left}: Illustration of Theorem~\ref{thm:disk-welding} with $W_1\ge\frac{\gamma^2}{2}$ and $W_2<\frac{\gamma^2}{2}$. \textbf{Right}: Illustration of Theorem~\ref{thm:disk+QT} with $W,W_3\ge\frac{\gamma^2}{2}$ and $W_1,W_2<\frac{\gamma^2}{2}$.}\label{fig-qt-weld} 
\end{figure}

Here, if $W_1+W_2<\frac{\gamma^2}{2}$, then the left hand side of~\eqref{eq:disk-welding} is defined by replacing the measure $\Md_{2}(\gamma^2-W_1-W_2)$ with $\Md_{2}(\gamma^2-W_1-W_2)\otimes \SLE_{ \kappa}(W_1-2;W_2-2)$ in the Poisson point process construction of $\Md_{2}(W_1+W_2)$ in Definition~\ref{def:thin-qd}.

We remark that different from~\cite{AHS21,ASY22}, in this paper we do not use the $\rm Weld$ operator and simply write product measures for the gluing of independent quantum surfaces instead as on the right hand side of~\eqref{eq:disk-welding}. This will ease the notation in Section~\ref{subsec:weld-simple-pf}  and be helpful to distinguish the interfaces when we  deal with the conformal welding of multiple quantum surfaces along different edges.

We have  {a} similar result for the conformal welding of a quantum disk with a quantum triangle.  {See Figure~\ref{fig-qt-weld} (right) for an illustration.}
\begin{theorem}\label{thm:disk+QT}
Let $\gamma\in(0,2)$ and $\kappa=\gamma^2$. Fix $W,W_1,W_2,W_3>0$ such that $W_2+W_3=W_1+2$ or $W_1+W_3 = W_2+\gamma^2-2$. There exists some constant  {$\ol C=\ol C(\gamma;W;W_1,W_2)\in (0,
\infty)$} such that
\begin{equation}\label{eqn:disk+QT}
    \QT(W+W_1,W+W_2,W_3)\otimes \SLE_{\kappa}(W-2;W_2-2,W_1-W_2) = \ol C\int_0^\infty \Md_{2}(W;\ell)\times\QT(W_1,W_2,W_3;\ell)\,d\ell.
\end{equation}
\end{theorem}
 {Note that as in the statement of~\cite[Theorem 1.2]{ASY22}, the constant $\ol C(\gamma;W;W_1,W_2)$ above does not depend on our choice of $W_3$.}
The $\SLE_{\kappa}(W-2;W_2-2,W_1-W_2)$ curve above is defined as follows. Assume that the quantum triangle is embedded as $(D,\phi,a_1,a_2,a_3)$ where the vertex $a_j$ has weight $W_j$, and the interface is from $a_2$ to $a_1$.  If the domain $D$ is simply connected (which corresponds to the case where $W+W_1,W+W_2, W_3\ge\frac{\gamma^2}{2}$), $\eta$ is just the ordinary  {$\SLE_{\kappa}(W-2;W_2-2,W_1-W_2)$}  with force points at $a_2^-,a_2^+$ and $a_3$. Otherwise, let $(\tilde{D}, \phi,\tilde{a}_1,\tilde{a}_2, \tilde{a}_3)$ be the thick quantum triangle component as in Definition~\ref{def:thin-qt}, and sample an  {$\SLE_{\kappa}(W-2;W_2-2,W_1-W_2)$} curve $\tilde{\eta}$ in $\tilde{D}$ from $\tilde{a}_2$ to $\tilde{a}_1$. Then our curve $\eta$ is the concatenation of $\tilde{\eta}$ with independent $\SLE_{\kappa}(W-2;W_1-2)$ curves in each bead of the weight $W+W_1$   quantum disk (if $W+W_1<\frac{\gamma^2}{2}$) and $\SLE_{\kappa}(W-2;W_2-2)$ curves in each bead of the weight $W+W_2$  quantum disk (if $W+W_2<\frac{\gamma^2}{2}$).

\begin{proof}
    The case where $W_2+W_3 = W_1+2$ is Theorem 1.1 of \cite{ASY22}. The case where $W_1+W_3 = W_2+\gamma^2-2$ follows from~\cite[Theorem 1.2]{ASY22}.
\end{proof}

\subsection{Proof of Theorem~\ref{thm:weld-BCLE}}\label{subsec:weld-simple-pf}

The proof of Theorem~\ref{thm:weld-BCLE} is outlined as follows. We first provide a description of the loop $\cL$ via chordal $\SLE_\kappa(\rho;\kappa-6-\rho)$ processes as in Lemma~\ref{lem:BCLEloop}. Then we cut a quantum disk from $\Md_{2,\circ}(\gamma^2-2)$ using the curves from Lemma~\ref{lem:BCLEloop}, creating five quantum disks as in the left panel of Figure~\ref{fig:target-free-chord}. This is done in Lemma~\ref{lem:target-free-chord}. Then in Proposition~\ref{prop:target-free-radial}, we apply Theorem~\ref{thm:disk-welding} and Theorem~\ref{thm:disk+QT} to reform the picture as the welding of three quantum disks. Finally we apply Theorem~\ref{thm:disk+QT} once more along with the definition of pinched thin quantum annulus in Section~\ref{subsec:pinched-thin-qa} to complete the proof. 

Let $\cT$ be the branching tree of chordal $\SLE_\kappa(\rho;\kappa-6-\rho)$ processes that generates the $\mathrm{BCLE}_{\kappa}^{\circlearrowright}(\rho)$ $\Gamma$ on the unit disk rooted at $-i$ and targeting all other boundary points. Recall that $\cL$ is the unique loop surrounding 0. Consider the component $D_{-i}$ of $\bbD\backslash\cL$ containing $-i$. Let $w$ be the leftmost (resp.\ rightmost) point on $\partial D_{-i}\cap\partial \bbD$ when $\cL$ is a counterclockwise (resp.\ clockwise) loop. Then from the target invariance property, the branch $\eta:=\eta^w$ of $\cT$ targeted at $w$ agrees in law with a radial $\SLE_\kappa(\rho;\kappa-6-\rho)$ starting from $-i$ targeted at 0 before reaching $w$. Let $S_\eta$ be the  components of $\bbD\backslash\eta$ whose boundary contains a segment of $\cL$, and write $I_\eta:=\partial S_\eta\cap \partial\bbD$.  

Fix a boundary point $z_0\neq -i$, and let $\eta^{z_0}$ be the branch of $\cT$ targeted at $z_0$. Consider the connected component $\wt D_{z_0}$ of $\bbD\backslash\eta^{z_0}$ containing 0, and let $z_1$ (resp.\ $z_2$) be the first (resp.\ last) point of $\wt D_{z_0}$ traced by $\eta^{z_0}$.  Given the domain $\wt D_{z_0}$, if 0 is on the left (resp.\ right) hand side of $\eta^{z_0}$, then  the  arc of $\cT$ from $z_2$ to $z_1$ has the law $\SLE_\kappa(0;\kappa-6-\rho)$ (resp.\ $\SLE_\kappa(\rho;0)$) in $\wt D_{z_0}$. Indeed, this can be deduced from the Domain Markov property of $\SLE_\kappa(\underline\rho)$  and the fact that any force point located at the target of the curve does not play any role. This gives a description of the branch $\eta^{z_1}$ from $-i$ to $z_1$. These definitions can be easily extended from $(\bbD,0,-i,z_0)$ to other simply connected domains $(D,z,a,b)$ with $a,b\in\partial D$ and $z\in D$ via conformal maps. See Figure~\ref{fig:BCLEloop} for an illustration. 

\begin{figure}
\centering
\begin{tabular}{cc}
   \includegraphics[scale=0.45]{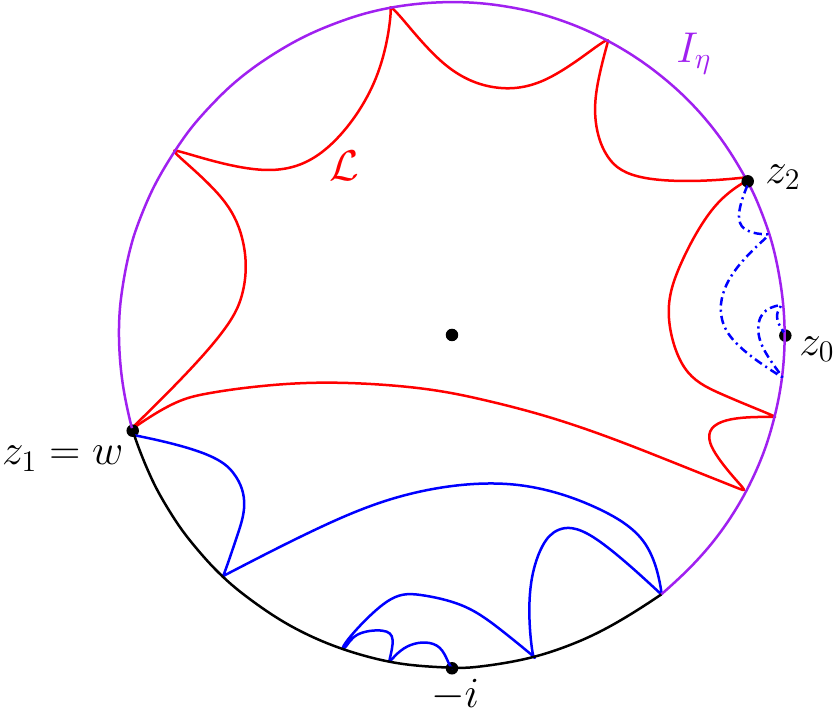}  &   \includegraphics[scale=0.45]{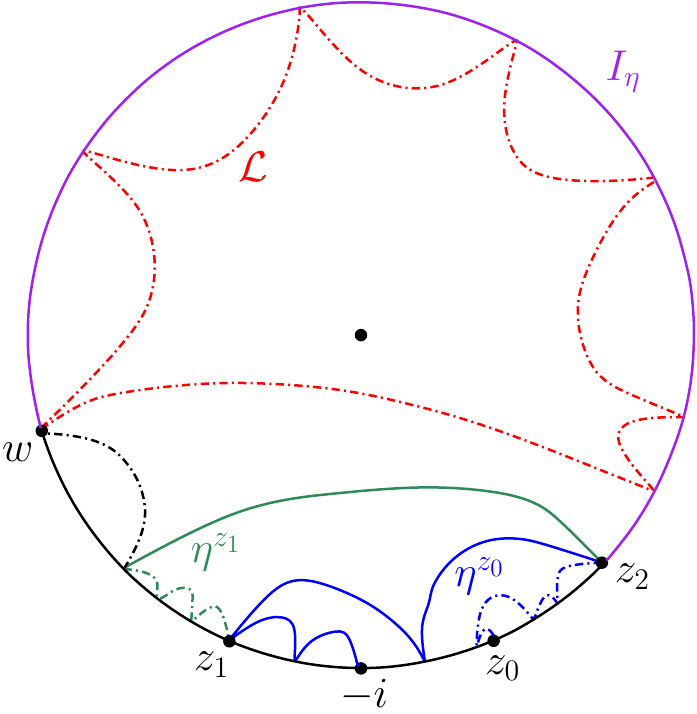}
\end{tabular}
    \caption{An illustration of Lemma~\ref{lem:BCLEloop-0}. \textbf{Left:} The loop $\cL$ is a counterclockwise loop, and the branch $\eta^w$ is the union of the blue and the red curve. In this setting $z_0\in I_\eta$ and $z_1 = w$. \textbf{Right:} An illustration of the branch $\eta^{z_1}$. 0 is not surrounded by $\eta^{z_1}$ and $z_0$ is not on the arc $I_\eta$.}
    \label{fig:BCLEloop}
\end{figure}

\begin{lemma}\label{lem:BCLEloop-0}
    The point $z_0$ is on the arc $I_\eta$ if and only if 0 is surrounded by  $\eta^{z_1}$.
\end{lemma}

\begin{proof}
    First assume that $\eta^{z_1}$ surrounds 0. Then for any $z'\in\partial D$, $\eta^{z'}$ agrees with $\eta^{z_1}$ before separating $z'$ from 0, and it follows that $\eta^{z_1} = \eta = \eta^w$. In particular, $\eta$ must have hit $w$ and traced a part of $\cL$ before separating $z_0$ from 0, and thus $z_0\in I_{\eta}$. On the other hand, if $\eta^{z_1}$ does not contain 0, since the tree $\cT$ is not self-crossing, $\eta^{z_0}$ cannot trace any part of $\cL$ (otherwise it would cross $\eta^{z_1}$). This implies $z_0\notin I_{\eta}$.
\end{proof}

Consider a simply connected domain $(D,z,a,b)$ with $a,b\in\partial D$, and $z\in D$. We define a measure $\mu$ on a pair $(\eta_1,\eta_2)$ of curves as follows. First sample an $\SLE_\kappa(\rho;\kappa-6-\rho)$ curve $\eta_2$ from $a$ to $b$. Let $D_{\eta_2}$ be the connected component of  {$D\backslash\eta_2$} containing $z$, and let $\wt z_1$ (resp.\ $\wt z_2$) be the first (resp.\ last) point on $\partial D_{\eta_2}$ traced by $\eta_2$. On the event where $z$ is on the left hand side of $\eta_2$, sample an $\SLE_\kappa(0;\kappa-6-\rho)$ curve $\eta_1$ in   {$ D_{\eta_2}$} from $\wt z_2$ to $\wt z_1$. Otherwise sample an $\SLE_\kappa(\rho;0)$ curve $\eta_1$ in   {$ D_{\eta_2}$} from $\wt z_2$ to $\wt z_1$.  Let $\mu$ be the law of $(\eta_1,\eta_2)$.

\begin{lemma}\label{lem:BCLEloop}
    In the domain $(\bbD,0,-i,z_0)$, under the setting of Lemma~\ref{lem:BCLEloop-0}, the following two laws $\mu_1$ and $\mu_2$ agree:
    \begin{enumerate}[(i)]
        \item  \emph{Restrict} to the event where $\cL$ is a counterclockwise loop and $z_0\in I_\eta$. Consider the branch   $\eta^{z_1}\backslash\eta^{z_0}$. Let $\mu_1$ be the law of $(\eta^{z_1}\backslash\eta^{z_0},\eta^{z_0})$.

        \item  Let $\mu_2$ be the \emph{restriction} of the measure $\mu$ on the event where 0 is between $\eta_1$ and $\eta_2$ and on the left hand side of $\eta_2$.
    \end{enumerate}
    The analogous statement holds if $\cL$ is a clockwise loop in (i), where in (ii) 0 would be on the right hand side of $\eta_2$ and $\eta_1$ would be an $\SLE_\kappa(\rho;0)$ curve. 
\end{lemma}

\begin{proof}
    This follows directly from Lemma~\ref{lem:BCLEloop-0} by further working on the event where 0 is on the left or right hand side of $\eta^{z_0}$.
\end{proof}

The following is a quick consequence of Theorem~\ref{thm:disk-welding}. Recall the notation  $\Md_{2,\circ}(W)$ defined in Section~\ref{subsec:lqg-q-s} after Lemma~\ref{lem:QT-W2W}.

\begin{lemma}\label{lem:disk-welding-interior}
    Let $\gamma\in(0,2)$, $W_1,W_2>0$, and $C(\gamma;W_1;W_2)$ be the constant in Theorem~\ref{thm:disk-welding}. Then
    \begin{equation}\label{eq:weld-disk-2-interior}
        \begin{split}
        &\mathds{1}_{E_L} \Md_{2,\circ}(W_1+W_2)\otimes \SLE_\kappa(W_1-2;W_2-2) = C(\gamma;W_1;W_2)\int_0^\infty \Md_{2,\circ}(W_1;\ell)\times \Md_{2}(W_2;\ell)\,d\ell;\\
        &\mathds{1}_{E_R} \Md_{2,\circ}(W_1+W_2)\otimes \SLE_\kappa(W_1-2;W_2-2) = C(\gamma;W_1;W_2)\int_0^\infty \Md_{2}(W_1;\ell)\times \Md_{2,\circ}(W_2;\ell)\,d\ell,
    \end{split}
    \end{equation}
    where $E_L$ (resp.\ $E_R$) is the event that the interior marked point in the weight $W_1+W_2$ quantum disk is lying on the left (resp.\ right) hand side of the interface $\eta$.
\end{lemma}

For $W,W_1,W_2>0$, recall the constant $C(\gamma;W_1;W_2)$ from Theorem~\ref{thm:disk-welding} and the constant $\ol C(\gamma;W;W_1,W_2)$ from Theorem~\ref{thm:disk+QT}. Throughout this section, for $\rho\in(-2,\kappa-4)$, we define
\begin{align}
     &C_0 = C(\gamma;\rho+2;\kappa-4-\rho), C_1 = C(\gamma;\kappa-4-\rho;2) \text{  and  } C_2 = C(\gamma;\rho+2;2)  \label{eq:constant-simple-A} \\
     &\ol C_1 = \ol C(\gamma;\rho+2;\kappa-4-\rho,2) \text{  and  } \ol C_2 = \ol C(\gamma; \kappa-4-\rho;\rho+2,2) \label{eq:constant-simple-B}
\end{align}

Consider an embedding of the quantum disk from $\Md_{2,\circ}(\gamma^2-2)$ with the bead containing the interior marked point embedded as $(D,\phi,z,a,b)$. Then we sample a pair of curves $(\eta_1,\eta_2)$ on $(D,z,a,b)$ according to the probability measure $\mu$. On other beads of the weight $\gamma^2-2$ quantum disk, we draw independent $\SLE_\kappa(\rho;\kappa-6-\rho)$ curves. We write $\Md_{2,\circ}(\gamma^2-2)\otimes\mu$ for the law of the corresponding curve-decorated quantum surface. Furthermore, let $E^{\circlearrowleft}$ (resp.\ $E^{\circlearrowright}$) be the event where the interior marked point is between $\eta_1$ and $\eta_2$ and lies on the left (resp.\ right) hand side of $\eta_2$. Then we have the following result. 
\begin{lemma}\label{lem:target-free-chord}
    Let $C_0,C_1,C_2$ be the constants defined in~\eqref{eq:constant-simple-A}.  Then
    \begin{equation}\label{eq:weld-interior-A}
        \begin{split}
        \mathds{1}_{E^{\circlearrowleft}} &\Md_{2,\circ}(\gamma^2-2)\otimes \mu = C_0C_1(1-\frac{2(\rho+2)}{\gamma^2})^2 \int_{\bbR_+^4} \Md_{2}(\kappa-4-\rho;\ell_1)\times \QD_{1,2}(\ell_1;\ell_2)\\&\times \Md_2(\rho+2;\ell_3)\times \Md_2(\rho+2;\ell_4)\times \Md_2(\kappa-4-\rho;\ell_2+\ell_3+\ell_4)\,d\ell_1d\ell_2d\ell_3d\ell_4;
    \end{split}
    \end{equation}
     \begin{equation}\label{eq:weld-interior-B}
        \begin{split}
        \mathds{1}_{E^{\circlearrowright}} &\Md_{2,\circ}(\gamma^2-2)\otimes \mu = C_0C_2(1-\frac{2(\kappa-4-\rho)}{\gamma^2})^2 \int_{\bbR_+^4} \Md_{2}(\rho+2;\ell_1)\times \QD_{1,2}(\ell_1;\ell_2)\\&\times \Md_2(\kappa-4-\rho;\ell_3)\times \Md_2(\kappa-4-\rho;\ell_4)\times \Md_2(\rho+2;\ell_2+\ell_3+\ell_4)\,d\ell_1d\ell_2d\ell_3d\ell_4.
    \end{split}
    \end{equation}
\end{lemma}
 {See  Figure~\ref{fig:target-free-chord} (left) for an illustration for the conformal welding in~\eqref{eq:weld-interior-A}, and Figure~\ref{fig:target-free-chord-A} (left) for an illustration for the conformal welding in~\eqref{eq:weld-interior-B} .}
\begin{proof}
    We start with Lemma~\ref{lem:disk-welding-interior} for $W_1 = \rho+2$ and $W_2 = \kappa-4-\rho$. By~\eqref{eq:thin-interior-decom} and a disintegration, we have
    \begin{equation*}
        \begin{split}
        \mathds{1}_{E_L} \Md_{2,\circ}(\gamma^2-2)&\otimes \SLE_\kappa(\rho;\kappa-6-\rho) = C_0(1-\frac{2(\rho+2)}{\gamma^2})^2\int_{\bbR_+^3} \Md_{2}(\rho+2;\ell_3)\times \Md_{2}(\rho+2;\ell_4)\\& \times \Md_{2,\circ}(\kappa-\rho-2;\ell_2)\times \Md_2(\kappa-4-\rho;\ell_2+\ell_3+\ell_4) \,d\ell_2d\ell_3d\ell_4.
    \end{split}
    \end{equation*}
    Then~\eqref{eq:weld-interior-A} follows by applying Lemma~\ref{lem:disk-welding-interior} once more with $W_1 = \kappa-4-\rho$ and $W_2 = 2$. (Recall $\Md_{2,\circ}(2) = \QD_{1,2}$.) ~\eqref{eq:weld-interior-B} can be proved analogously.
\end{proof}

\begin{figure}[tb]
\centering
\begin{tabular}{cccc}
   \includegraphics[scale=0.6]{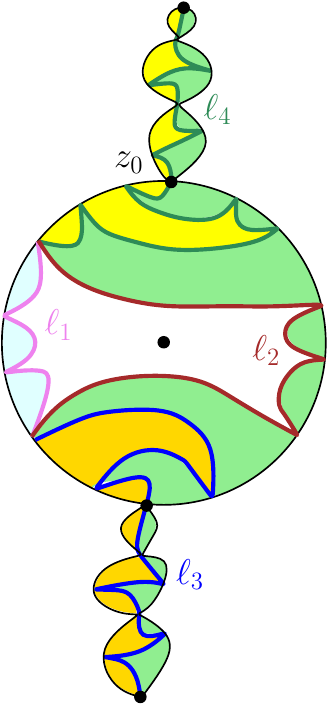} \ \   & \ \   \includegraphics[scale=0.6]{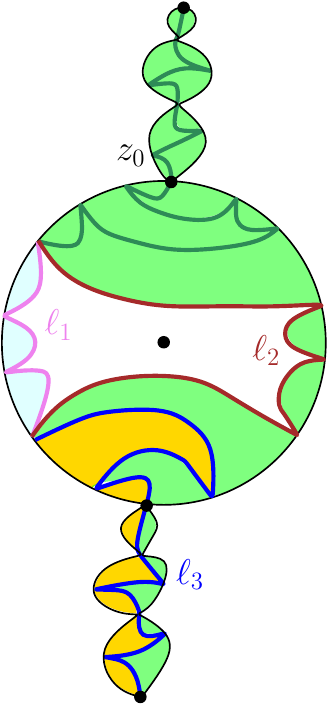} & \ \ \includegraphics[scale=0.6]{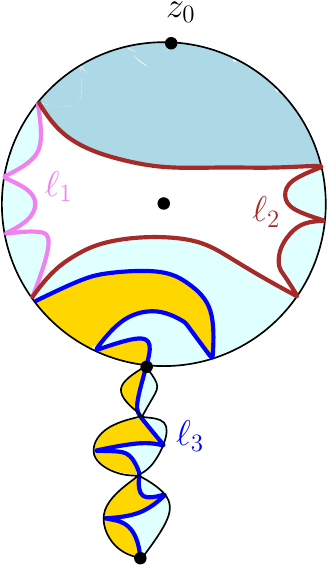} & \ \  \includegraphics[scale=0.6]{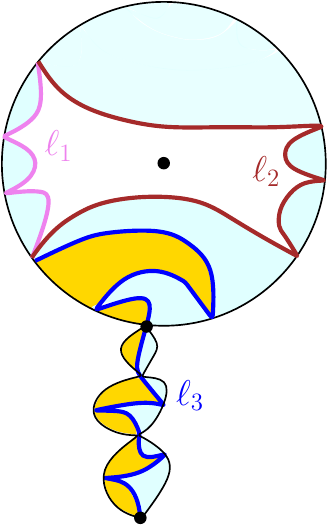}
\end{tabular}
    \caption{\textbf{Left:} The conformal welding picture in Lemma~\ref{lem:target-free-chord} on the event $E^{\circlearrowleft}$. $\ell_1,\ell_2,\ell_3,\ell_4$ are lengths of the  {violet, brown}, blue and green interfaces, respectively. The yellow and orange parts are weight $\rho+2$ quantum disks, the blue and green parts are weight $\kappa-4-\rho$ quantum disks, and the white disk in the center corresponds to the $\QD_{1,2}$ part.  \textbf{Middle left:} The gluing of the weight $\rho+2$ orange disk with the weight $\kappa-4-\rho$ quantum disk in the left panel gives the  weight $(\gamma^2-2,\rho+4,\kappa-4-\rho)$ green quantum triangle with green interface. \textbf{Middle right \& Right:}  The green quantum triangle in the middle panel can be viewed as constant times $\Md_2(\gamma^2-2)\times\Md_{2,\bullet}(\rho+4)\times\Md_2(\kappa-4-\rho)$. By forgetting the top thin quantum disk in the middle panel, one can combine together with the weight $\kappa-4-\rho$ blue quantum disk to get the blue marked quantum disk from $\Md_{2,\bullet}(\kappa-4-\rho)$. The restriction $z_0\in I_\eta$ corresponds to the event where the dark blue disk containing $z_0$ is welded to a part of the white quantum disk in the middle. By further forgetting about the marked point $z_0$ on the top one obtains the conformal welding~\eqref{eq:target-free-radial-A} in Proposition~\ref{prop:target-free-radial} (with $s = \ell_3$ and $\ell=\ell_1+\ell_2$), as in the rightmost panel.}
    \label{fig:target-free-chord}
\end{figure}

Recall the branch $\eta:=\eta^w$ and the loop $\cL$ constructed from the tree $\cT$. We write $\mu_0$ for the law of $\eta$, and $\mu_0^{\circlearrowleft}$ (resp.\ $\mu_0^{\circlearrowright}$) for the restriction of $\mu_0$ on the event where $\cL$ is a counterclockwise (resp.\ clockwise) loop, and extend the definition to other domains $(D,z,a)$ where $z\in D$ and $a\in\partial D$ via conformal maps. Consider the concatenation of samples from $\QD_{1,1}\times \Md_2(\gamma^2-2)$. We define the measure $\big(\QD_{1,1}\times \Md_2(\gamma^2-2)\big)\otimes \mu_0^{\circlearrowleft}$ (resp.\ $\big(\QD_{1,1}\times \Md_2(\gamma^2-2)\big)\otimes \mu_0^{\circlearrowright}$) to be the law describing the curve-decorated quantum surface obtained by independently drawing a sample from $\mu_0^{\circlearrowleft}$ (resp.\ $\mu_0^{\circlearrowright}$) on top of the quantum disk from $\QD_{1,1}$ and a sample from $\SLE_\kappa(\rho;\kappa-6-\rho)$ on each bead of the weight $\gamma^2-2$ quantum disk. The measures $\Md_{2,\circ}(\gamma^2-2)\otimes\mu_0^{\circlearrowleft}$
 and $\Md_{2,\circ}(\gamma^2-2)\otimes\mu_0^{\circlearrowright}$ can be defined analogously.
\begin{proposition}\label{prop:target-free-radial}
    Let $C_0,C_1,C_2$ be the constants in~\eqref{eq:constant-simple-A}, and $\ol C_1$, $\ol C_2$ be the constants in~\eqref{eq:constant-simple-B}. Define
    \begin{align}
        &C_0^{\circlearrowleft} = (\frac{4}{\gamma^2}-1)^{-1}(1-\frac{2(\kappa-\rho-4)}{\gamma^2})^{-1}(1-\frac{2(\rho+2)}{\gamma^2})^2C_0C_1\ol C_1^{-1},\label{eq:target-free-radial-const-A}\\&
        C_0^{\circlearrowright} = (\frac{4}{\gamma^2}-1)^{-1}(1-\frac{2(\rho+2)}{\gamma^2})^{-1}(1-\frac{2(\kappa-4-\rho)}{\gamma^2})^2C_0C_2\ol C_2^{-1}.\label{eq:target-free-radial-const-B}
    \end{align}
    Then we have
    \begin{align}
        & \big(\QD_{1,1}\times \Md_2(\gamma^2-2)\big)\otimes \mu_0^{\circlearrowleft}  = C_0^{\circlearrowleft}\int_{\bbR_+^2} \QD_{1,1}(\ell)\times  {\Md_2(\rho+2;s)}\times \Md_2(\kappa-4-\rho;\ell+s)\,dsd\ell \label{eq:target-free-radial-A} \\
       & \big(\QD_{1,1}\times \Md_2(\gamma^2-2)\big)\otimes \mu_0^{\circlearrowright}  = C_0^{\circlearrowright}\int_{\bbR_+^2} \QD_{1,1}(\ell)\times  {\Md_2(\kappa-4-\rho;s)}\times \Md_2(\rho+2;\ell+s)\,dsd\ell. \label{eq:target-free-radial-B}
    \end{align}
\end{proposition}
 {See the rightmost panel of Figure~\ref{fig:target-free-chord} for an illustration of the conformal welding in~\eqref{eq:target-free-radial-A}, and the right panel of Figure~\ref{fig:target-free-chord-A} for an illustration of the conformal welding in~\eqref{eq:target-free-radial-B}.}

\begin{figure}[tb]
\centering
\begin{tabular}{cc}
   \includegraphics[scale=0.6]{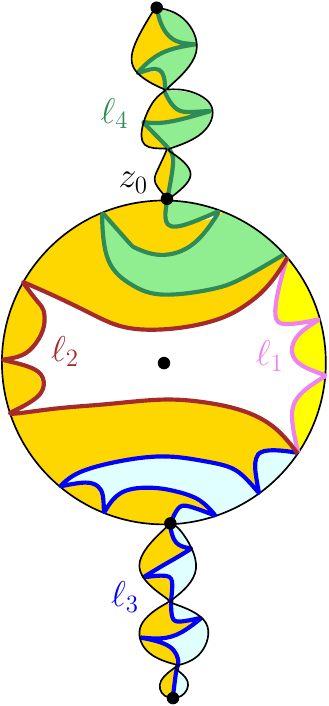} \ \   & \ \   \includegraphics[scale=0.6]{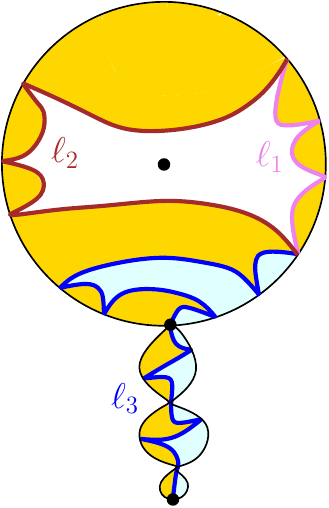}  
\end{tabular}
    \caption{ {\textbf{Left:} An illustration of the conformal welding in~\eqref{eq:weld-interior-B}, where the yellow and orange parts are   weight $\rho+2$ quantum disks, and the green and blue parts are weight $\kappa-4-\rho$ quantum disks, and the white disk in the center corresponds to the $\QD_{1,2}$ part. \textbf{Right:} The conformal welding~\eqref{eq:target-free-radial-B} in Proposition~\ref{prop:target-free-radial} (with $s = \ell_3$ and $\ell=\ell_1+\ell_2$), where the orange part is the weight $\rho+2$ quantum disk, the blue part is the weight $\kappa-4-\rho$ quantum disk, and the white disk in the center corresponds to the $\QD_{1,1}$ part.}}
    \label{fig:target-free-chord-A}
\end{figure}

\begin{proof}
    We only prove~\eqref{eq:target-free-radial-A}; ~\eqref{eq:target-free-radial-B} can be treated similarly.  Without loss of generality assume that the bead containing the interior marked point of the sample from $\Md_{2,\circ}(\gamma^2-2)$ on the left hand side of~\eqref{eq:weld-interior-A} is embedded as $(\bbD,\phi,0,-i,z_0)$. We start with the right hand side of~\eqref{eq:weld-interior-A}. Given $\ell_2,\ell_3,\ell_4$, we mark the point on the left boundary of the weight $\kappa-4-\rho$ quantum disk (which has quantum length $\ell_2+\ell_3+\ell_4$) with distance $\ell_4$ to the top vertex. By definition this gives $\Md_{2,\bullet}(\kappa-4-\rho;\ell_4,\ell_2+\ell_3)$, which further equals $\frac{\gamma(Q-\gamma)}{2}\QT(\kappa-4-\rho,2,\kappa-4-\rho;\ell_4,\ell_2+\ell_3)$ by~\eqref{eq:QT-W2W}, where the two boundary arcs adjacent to the weight 2 vertex have lengths $\ell_4$ and $\ell_2+\ell_3$. Therefore the right hand side of~\eqref{eq:weld-interior-A} equals
    \begin{equation}\label{eq:weld-interior-A-1}
        \begin{split}
          & {C_0C_1(1-\frac{2(\rho+2)}{\gamma^2})^2\frac{\gamma(Q-\gamma)}{2}}\int_{\bbR_+^3}\int_0^\infty \Md_{2}(\kappa-4-\rho;\ell_1)\times \QD_{1,2}(\ell_1;\ell_2)\times \Md_2(\rho+2;\ell_3) \\&\times \big(\int_0^\infty \Md_2(\rho+2;\ell_4)\times \QT(\kappa-4-\rho,2,\kappa-4-\rho;\ell_4,\ell_2+\ell_3)\,d\ell_4\big)\,d\ell_1d\ell_2d\ell_3.
    \end{split}
    \end{equation}
By Theorem~\ref{thm:disk+QT}, ~\eqref{eq:weld-interior-A-1} further equals
     \begin{equation}\label{eq:weld-interior-A-2}
        \begin{split}
          & {C_0C_1\ol C_1^{-1}(1-\frac{2(\rho+2)}{\gamma^2})^2\frac{\gamma(Q-\gamma)}{2}}\int_{\bbR_+^3}  \Md_{2}(\kappa-4-\rho;\ell_1)\times \QD_{1,2}(\ell_1;\ell_2)\times \Md_2(\rho+2;\ell_3) \\&\times \big(\QT(\rho+4,\kappa-4-\rho,\gamma^2-2;\ell_2+\ell_3)\otimes \SLE_\kappa(\rho;0,\kappa-6-\rho)\big)\,d\ell_1d\ell_2d\ell_3,
    \end{split}
    \end{equation}
    where $\ell_2+\ell_3$ in~\eqref{eq:weld-interior-A-2} is the length of the boundary arc between the weight $\rho+4$ and weight $\kappa-4-\rho$ vertex. On the other hand,   by Lemma~\ref{lem:BCLEloop}, once we forget the welding interface in the second line of~\eqref{eq:weld-interior-A-1}, the rest of the interfaces have the same law as $\mathds{1}_{E_0^{\circlearrowleft}}\mu_0^{\circlearrowleft}$. Here $E_0^{\circlearrowleft}$ is the event where $z_0$ is on the arc $I_\eta$ for $\eta$ sampled from $\mu_0^{\circlearrowleft}$, i.e., the boundary of the connected component of $\bbD\backslash\eta$ containing $z_0$ contains a segment of the loop formed by $\eta$. Therefore 
     \begin{equation}\label{eq:weld-interior-A-3}
        \begin{split}
         \mathds{1}_{E_0^{\circlearrowleft}}\Md_{2,\circ}(\gamma^2-2)\otimes\mu_0^{\circlearrowleft} &= C_0C_1\ol C_1^{-1}(1-\frac{2(\rho+2)}{\gamma^2})^2\frac{\gamma(Q-\gamma)}{2} \int_{\bbR_+^3}  \Md_{2}(\kappa-4-\rho;\ell_1)\times \QD_{1,2}(\ell_1;\ell_2)\\&\times \Md_2(\rho+2;\ell_3)\times  \QT( \rho+4,\kappa-4-\rho,\gamma^2-2;\ell_2+\ell_3) \,d\ell_1d\ell_2d\ell_3.
    \end{split}
    \end{equation}
Now we have the decomposition
\begin{equation*}
\begin{split}
   &\QT(\gamma^2-2,\rho+4,\kappa-4-\rho) = (\frac{4}{\gamma^2}-1) \QT(2,\rho+4,\kappa-4-\rho)\times \Md_2(\gamma^2-2).
    \end{split}
\end{equation*}
Then further by a change of variables, the integral on the right hand side of~\eqref{eq:weld-interior-A-3} is equal to
\begin{equation}\label{eq:weld-interior-A-4}
    \begin{split}
       (\frac{4}{\gamma^2}-1)\bigg( \int_{\bbR_+^2}\int_0^\ell &\Md_{2}(\kappa-4-\rho;\ell_1) \times \QD_{1,2}(\ell_1;\ell-\ell_1) \times \Md_2(\rho+2;s) \\& \times \QT(\rho+4,\kappa-4-\rho,2;s+\ell-\ell_1) \,d\ell_1 d\ell ds\bigg) \times \Md_2(\gamma^2-2).
    \end{split}
\end{equation}
Since $\Md_{2,\circ}(\gamma^2-2) = (\frac{4}{\gamma^2}-1)^2\Md_2(\gamma^2-2)\times \QD_{1,2}\times \Md_2(\gamma^2-2)$,  we may discard the components above the bead containing the interior marked points (i.e., transit from the middle panel to the right panel of Figure~\ref{fig:target-free-chord}). This gives
\begin{equation}\label{eq:weld-interior-A-5}
    \begin{split}
        \mathds{1}_{E_0^{\circlearrowleft}}&\big( {\QD_{1,2}}\times \Md_2(\gamma^2-2)\big)\otimes \mu_0^{\circlearrowleft} = {C_0C_1\ol C_1^{-1}(1-\frac{2(\rho+2)}{\gamma^2})^2\frac{\gamma(Q-\gamma)}{2}(\frac{4}{\gamma^2}-1)^{-1}}\int_{\bbR_+^2}\int_0^\ell    \QD_{1,2}(\ell_1;\ell-\ell_1)\\&\times \Md_2(\rho+2,s) \times \big( \Md_{2}(\kappa-4-\rho;\ell_1) \times \QT(\rho+4,\kappa-4-\rho,2;s+\ell-\ell_1)\big) \,d\ell_1 d\ell ds.
    \end{split}
\end{equation}
Note that $\ell$  {indicates} the boundary length of the quantum disk from $\QD_{1,2}$, and recall that a sample from  $\QD_{1,2}(\ell_1;\ell-\ell_1)$ can be produced by starting with $\QD_{1,1}(\ell)$ and marking a second point on the boundary in clockwise direction with distance $\ell_1$ to the first.   
By ~\eqref{eq:thin-bdy-decomp}, we have 
\begin{equation}\label{eq:W2W-A}
    \Md_{2,\bullet}(\kappa-4-\rho) = (1-\frac{2(\kappa-4-\rho)}{\gamma^2})\frac{\gamma(Q-\gamma)}{2}\Md_2(\kappa-4-\rho)\times\QT(\rho+4,\kappa-4-\rho,2).
\end{equation}
Then in~\eqref{eq:weld-interior-A-5}, the integral over $\ell_1$ corresponds to the disintegration over the quantum length $\ell_1$ of the left boundary of the disk from  $\Md_2(\kappa-4-\rho)$ as in the decomposition~\eqref{eq:W2W-A}, and the bound $\ell_1<\ell$  {indicates} the restriction to the event where the bead of the $\Md_{2,\bullet}(\kappa-4-\rho)$ containing the third marked point $z_0$ on the boundary is welded to the quantum disk from $\QD_{1,1}$. Therefore  the conformal welding in~\eqref{eq:weld-interior-A-5} can be viewed as the welding of samples from $\Md_{2,\bullet}(\kappa-4-\rho),\QD_{1,1}$ along with $\Md_2(\rho+2)$, i.e., by dividing the constant in~\eqref{eq:weld-interior-A-5} by the constant in~\eqref{eq:W2W-A},
\begin{equation}\label{eq:weld-interior-A-6}
    \mathds{1}_{E_0^{\circlearrowleft}}\big( {\QD_{1,2}}\times \Md_2(\gamma^2-2)\big)\otimes \mu_0^{\circlearrowleft}  = C_0^{\circlearrowleft}\mathds{1}_{E_0^{\circlearrowleft}}\int_{\bbR_+^2} \QD_{1,1}(\ell)\times \Md_2(\rho+2,s)\times \Md_{2,\bullet}(\kappa-4-\rho,\ell+s)\,dsd\ell
\end{equation}
where the event $E_0^{\circlearrowleft}$ on the right hand side of~\eqref{eq:weld-interior-A-6} indicates that the bead of the  weight $\kappa-4-\rho$ quantum disk containing $z_0$ is welded to the $\QD_{1,1}$  quantum disk. Thus~\eqref{eq:target-free-radial-A} follows from~\eqref{eq:weld-interior-A-6} by further forgetting about the point $z_0$, since ~\eqref{eq:weld-interior-A-6} can be viewed as weighting the law of both sides of~\eqref{eq:target-free-radial-A} by quantum length of the arc $I_\eta$ and then sampling $z_0\in I_\eta$ according to the probability measure proportional to the length measure.
\end{proof}

\begin{proof}[Proof of Theorem~\ref{thm:weld-BCLE}]
    We start with~\eqref{eq:target-free-radial-A}. We mark the point on the left boundary of the weight $\kappa-4-\rho$ quantum disk with distance $s$ to the bottom root. This gives
    \begin{equation}
        \big(\QD_{1,1}\times \Md_2(\gamma^2-2)\big)\otimes \mu_0^{\circlearrowleft}  = C_0^{\circlearrowleft}\int_{\bbR_+^2} \QD_{1,1}(\ell)\times \Md_2(\rho+2,s)\times \Md_{2,\bullet}(\kappa-4-\rho,\ell,s)\,dsd\ell.
    \end{equation}
    Thus by~\eqref{eq:QT-W2W}, we may apply Theorem~\ref{thm:disk+QT} to weld the weight $\rho+2$ quantum disk and the marked weight $\kappa-4-\rho$ quantum disk (viewed as a quantum triangle of weights $(\kappa-4-\rho,2,\kappa-4-\rho)$) to obtain
    \begin{equation}\label{eq:pf-BCLE-zipper-1}
    \begin{split}
        \big(\QD_{1,1}\times \Md_2(\gamma^2-2)\big)\otimes \mu_0^{\circlearrowleft} &= \frac{\gamma(Q-\gamma)}{2}C_0^{\circlearrowleft}\ol C_1^{-1} \int_0^\infty \QD_{1,1}(\ell)\\&\times \big(\QT(\kappa-4-\rho,\rho+4,\gamma^2-2;\ell)\otimes \SLE_\kappa(\rho;0,\kappa-6-\rho) \big) d\ell.
        \end{split}
    \end{equation}
Now we forget about the segment of the interface $\eta$ from $\mu_0^{\circlearrowleft}$ that is not on the loop $\cL$. Since $\QT(\kappa-4-\rho,\rho+4,\gamma^2-2) = (\frac{4}{\gamma^2}-1)\QT(\kappa-4-\rho,\rho+4, 2)\times\Md_2(\gamma^2-2)$, we can also discard the weight $\gamma^2-2$ quantum disk part in~\eqref{eq:pf-BCLE-zipper-1} (which corresponds to removing the beads on the lower part in the right panel of Figure~\ref{fig:target-free-chord}). As a consequence, we get
\begin{equation}\label{eq:pf-BCLE-zipper-2}
    \QD_{1,1}\otimes\mu^{\circlearrowleft} = \frac{\gamma(Q-\gamma)}{2} (\frac{4}{\gamma^2}-1)C_0^{\circlearrowleft}\ol C_1^{-1} \int_0^\infty \QD_{1,1}(\ell)\times \QT(\kappa-4-\rho,\rho+4, 2;\ell)\, d\ell
\end{equation}
where the welding is along the boundary arc of the weight $\kappa-4-\rho$ vertex and the weight $\rho+4$ vertex in the quantum triangle. On the other hand, $$\QT(\kappa-4-\rho,\rho+4,2) = (\frac{\gamma(Q-\gamma)}{2})^{-1}(1-\frac{2(\kappa-4-\rho)}{\gamma^2})\Md_2(\kappa-4-\rho)\times\Md_{2,\bullet}(\rho+4).$$
Thus by Corollary~\ref{cor:QApinched}, the weight $(\kappa-4-\rho,\rho+4,2)$ quantum triangle can be reformed into a pinched quantum annulus of weight  {$\kappa-\rho-4$}.  We may also forget the marked point on the boundary of the $\QD_{1,1}$ quantum disk on the right hand side of~\eqref{eq:pf-BCLE-zipper-2}, which gives the uniform conformal welding of the quantum disk  {from $\QD_{1,0}$} and the pinched quantum annulus  {as introduced at the beginning of Section~\ref{sec:welding-simple}}. Therefore by~\eqref{eq:pf-BCLE-zipper-2}, 
    \begin{equation}\label{eq:pf-BCLE-zipper-3}
    \QD_{1,1}\otimes\mu^{\circlearrowleft} = (1-\frac{2(\kappa-4-\rho)}{\gamma^2}) (\frac{4}{\gamma^2}-1)C_0^{\circlearrowleft}\ol C_1^{-1} \int_0^\infty \ell\, \QD_{1,0}(\ell)\times \wt\QA_1(\kappa-4-\rho;\ell)\, d\ell
\end{equation}
where $\ell$ indicates the inner boundary length of the pinched quantum annulus. We conclude the proof of ~\eqref{eq:thm-weld-A} by forgetting the marked point on the outer boundary, while~\eqref{eq:thm-weld-B} follows analogously.
\end{proof}

\section{Non-simple BCLE loop from conformal welding}\label{sec:welding-non-simple}
The goal of this section is to prove Theorem~\ref{thm:weld-BCLE-non-simple}, which is the analog of Theorem~\ref{thm:weld-BCLE} for the non-simple case. This involves the generalized quantum surfaces studied in~\cite{DMS21,MSW21-nonsimple,AHSY23}. In particular, we show that the $\BCLE_{\kappa'}(\rho')$ loop can be viewed as the interface of the conformal welding of forested pinched thin quantum annulus along with a forested quantum disk. 

In Section~\ref{subsec:def-gen-qs}, we will review the definition of generalized quantum surfaces and introduce the forested version of  pinched thin quantum annulus. In Section~\ref{subsec:welding-non-simple}, we review the conformal welding of generalized quantum surfaces and prove Theorem~\ref{thm:weld-non-simple-radial}, where radial $\SLE_{\kappa'}(\rho';\kappa'-6-\rho')$ is the interface under conformal welding of forested quantum triangle. Finally in Section~\ref{subsec:pf-non-simple-weld} we state and prove Theorem~\ref{thm:weld-BCLE-non-simple}.

\subsection{Definition of generalized quantum surfaces for $\gamma \in (\sqrt2, 2)$}\label{subsec:def-gen-qs}

In this section we recall the  {forested lines} and generalized quantum surfaces considered in~\cite{DMS21,MSW21-nonsimple,AHSY23}, following the treatment of~\cite{AHSY23}. For $\gamma\in(\sqrt{2},2)$, the forested lines are defined in~\cite{DMS21} based on the \emph{$\frac{4}{\gamma^2}$-stable looptrees} studied in~\cite{CK13looptree}.
Consider a stable L\'{e}vy process $(X_t)_{t\ge0}$ starting from 0 of index $\frac{4}{\gamma^2}\in(1,2)$ with only upward jumps, {so} $X_t\overset{d}{=}t^{\frac{\gamma^2}{4}}X_1$ for any $t>0$. By~\cite{CK13looptree},  one can construct a tree of topological disks from $(X_t)_{t\geq0}$ as in Figure~\ref{fig:forestline-def}. The forested line is defined by replacing each disk with an independent sample of the probability measure obtained from $\QD$ by conditioning on the boundary length to be the size of the corresponding jump. 
The quantum disks are glued in a clockwise length-preserving way with the rotation chosen uniformly at random. The unique point corresponding to $(0,0)$ on the graph of $X$ is called the \emph{root}.  
We call the closure of the collection of the points on the boundaries of the quantum disks the \emph{forested boundary arc}, while the set of the points corresponding to the running infimum of $(X_t)_{t\ge0}$ is called the \emph{line boundary arc}. Since $X$ only has positive jumps, the quantum disks are lying on the same side of the line boundary arc.

\begin{figure}
    \centering
    \begin{tabular}{cc} 
		\includegraphics[scale=0.4]{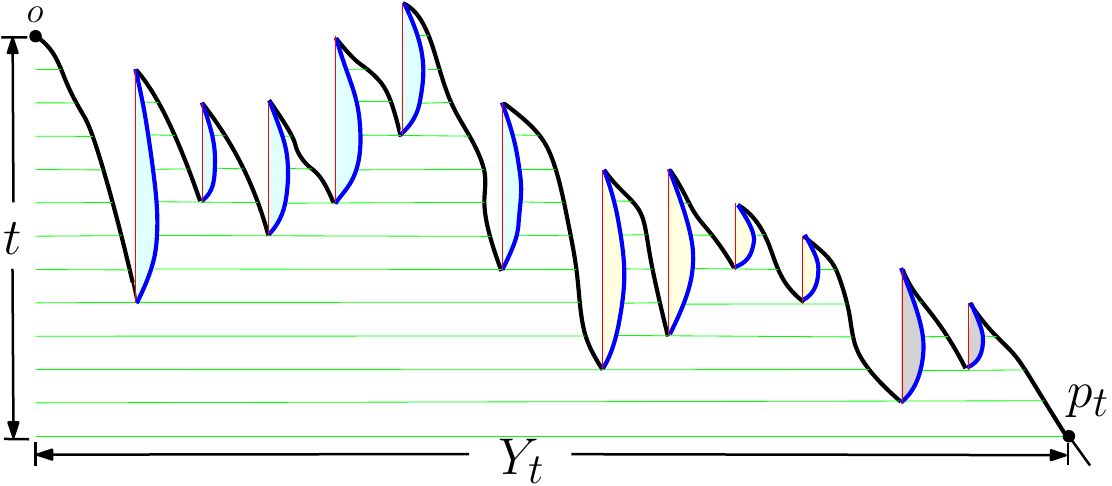}
		&
	   \includegraphics[scale=0.55]{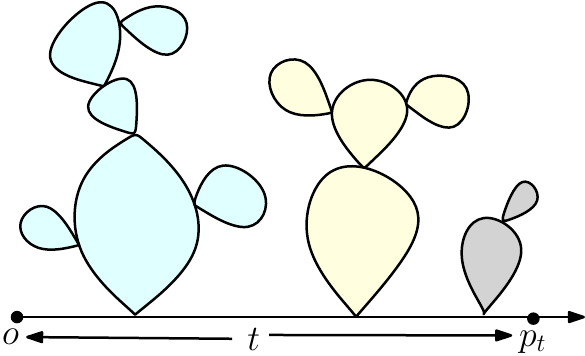}
	\end{tabular}
 \caption{\textbf{Left:} The graph of the L\'{e}vy process $(X_t)_{t>0}$ with only upward jumps. We draw the blue curves for each of the jump, and identify the points that are on the same green horizontal line.   \textbf{Right:} The L\'{e}vy tree of disks obtained from the left panel. For each topological disk we assign a quantum disk $\QD$ conditioned on having the same boundary length as the size of the jump, with the points on each red line in the left panel collapsed to a single point. The quantum length of the line segment between the root $o$ and the point $p_t$ is $t$, while the segment along the forested boundary between $o$ and $p_t$ has generalized quantum length $Y_t = \inf\{s>0:X_s\le -t\}$.  }\label{fig:forestline-def}
 \end{figure}

\begin{definition}[Forested line]\label{def:forested-line}
For $\gamma \in (\sqrt2,2)$, let $(X_s)_{s\geq 0}$ be a stable L\'evy process of index $\frac{4}{\gamma^2}>1$ with only positive jumps {satisfying $X_0=0$ a.s.}. For $t>0$, let $Y_t=\inf\{s>0:X_s\le -t\}$, and fix the multiplicative constant of $X$ such that $\bbE[e^{-Y_1}] = e^{-1}$. Define the forested line as {described above}.

 The line boundary arc is parametrized by quantum length. The forested boundary arc is parametrized by \emph{generalized quantum length}; that is, the length of the corresponding interval of $(X_t)$. For a point $p_t$ on the line boundary arc with LQG distance $t$ to the root, the segment of the forested boundary arc between $p_t$ and the root has generalized quantum length $Y_t$. 
\end{definition}

As in~\cite{AHSY23}, one can define a \emph{truncation} operation on forested lines. For $t>0$ and a forested line $\cL^o$ with root $o$, mark the point $p_t$ on the line boundary arc with quantum length $t$ from $o$. By \emph{truncation of $\cL^o$ at quantum length $t$}, we refer to the surface $\cL_t$ which is the union of the line boundary arc and the quantum disks on the forested boundary arc between $o$ and $p_t$. In other words, $\cL_t$ is the surface generated by $(X_s)_{0\le s\le Y_t}$ in the same way as Definition~\ref{def:forested-line}, and the generalized quantum length of the forested boundary arc of $\cL_t$ is $Y_t$. The beaded quantum surface $\cL_t$ is called a forested line segment.

\begin{definition}[Forested line segment]\label{def:line-segment}
    Fix $\gamma\in(\sqrt{2},2)$. Define $\mathcal{M}_2^\mathrm{f.l.}$ as the law of the surface obtained by first sampling $\mathbf{t}\sim \mathrm{Leb}_{\bbR_+}$ and truncating an independent forested line at quantum length $\mathbf{t}$. 
\end{definition}

Now we recall the  definition of generalized quantum surfaces in~\cite{AHSY23}. Let $n\ge1$, and $(D,\phi,z_1,\ldots,z_n)$ be an embedding of a (possibly beaded) quantum surface $S$ of finite volume, 
with $z_1,\ldots,z_n\in\partial D$  ordered clockwise. 
    We sample independent forested lines $\cL^1,\ldots,\cL^n$, truncate them  {so that} their quantum lengths match the length of boundary segments $[z_1,z_2],\ldots,[z_n,z_1]$ and glue them to $\partial D$ correspondingly.
    Let $S^f$ be the resulting beaded quantum surface. 
    
\begin{definition}\label{def:f.s.}
   We call a beaded quantum surface $S^f$ as above a (finite volume) \emph{generalized quantum surface}.
   We call this procedure \emph{foresting the boundary} of $S$, and say $S$ \emph{is the spine of} $S^f$. 
\end{definition}

We present two types of generalized quantum surfaces needed in Theorem~\ref{thm:weld-BCLE-non-simple} below.
\begin{definition}\label{def:q.t.f.s.}
    Let $W,W_1,W_2,W_3>0$. Recall from Definitions~\ref{def:thick-qt} and~\ref{def:thin-qt} the notion $\QT(W_1,W_2,W_3)$, and the notion $\QD_{1,1}$ from Definition~\ref{def:QD-mn}. We write $\QT^f(W_1,W_2,W_3)$ for the law of the generalized quantum surface obtained by foresting the three boundary arcs of a quantum triangle sampled from $\QT(W_1,W_2,W_3)$. Likewise, we write $\QD_{1,1}^f$ for the law of the generalized quantum surface obtained by foresting the   boundary arc of a quantum disk sampled from $\QD_{1,1}$, and define $\Mfd_{2}(W)$ via $\Md_{2}(W)$ similarly.
\end{definition}

Recall the measure $\Md_{2,\bullet}$ defined in Section~\ref{subsec:lqg-q-s} above Lemma~\ref{lem:QT-W2W}. We   define $\mathcal{M}_{2, \bullet}^{\textup{f.d.}}(W)$ analogously.   First sample  a forested quantum disk from $\Mfd_{2}(W)$ and weight its law by the generalized quantum length of its left boundary arc. Then we sample a marked point on the left boundary according to the probability measure proportional to the generalized quantum length. We denote the law of the triply marked quantum surface  by $\mathcal{M}_{2, \bullet}^{\textup{f.d.}}(W)$.

The following is the analog of Lemma~\ref{lem:QT-W2W}.
\begin{lemma}[Lemma 4.1 of~\cite{ASYZ24}]\label{lem:QT-W2Wf}
   For $W>0$ with $W\neq\frac{\gamma^2}{2}$, we have $ {\rm QT}^f (\gamma^2-2,W,W) = \frac{4}{\gamma^2}  {\Mfd_{2,\bullet}(W)}$.
\end{lemma}

Parallel to Definition~\ref{def:QD-mn}, we introduce the following typical forested quantum disks as below.
\begin{definition}[Definition 3.9 of~\cite{AHSY23}]\label{def:GQD1}
Let $\GQD_{2} := {\mathcal M}_2^{\mathrm{f.d.}}(\gamma^2-2)$ be the infinite measure on generalized quantum surfaces, and 
let $\GQD_{1}$ denote the corresponding measure when we forget one of the marked points and unweight by the generalized quantum length of the forested boundary.
\end{definition}

As shown in~\cite{MSW21-nonsimple,AHSY23}, the forested line can also be viewed as a Poisson point process on generalized quantum disks.
\begin{proposition}[Proposition 3.11 of~\cite{AHSY23}]\label{prop:typical}
Sample a forested line, and consider the collection of pairs $(u, \mathcal D_u^f)$ such that $\mathcal D_u^f$ is a generalized quantum surface attached to the line boundary arc (with its root defined to be the attachment point) and $u$ is the quantum length from the root of the forested line  to the root of $\mathcal D_u^f$. Then the law of this collection is a Poisson point process with intensity measure $c_0\mathrm{Leb}_{\bbR_+} \times \GQD_{1}$ for some constant $c_0>0$.
\end{proposition}

 In Definition~\ref{def:f.s.}, when foresting the boundary of a quantum surface $S$, we required that $S$ has at least one marked point on the boundary. To extend to the case where $S$ has no marked boundary points, we need to introduce the definition of forested circles. 
\begin{definition}\label{def:forested-circle}
    We define a measure  {$\cM^{\mathrm{f.c.}}$} as follows. %\haoyu{** Meaning of 2 in the index? Seems no marked points **} 
    Let $c_0$ be the constant in Proposition~\ref{prop:typical}. Take $\ell\sim c_0\ell^{-1}\mathds{1}_{\ell>0}d\ell$. Consider the unit circle $\cC$, which we assign quantum length $\ell$. Then sample a Poisson point process $\{(u,\cD_u)\}$ from the measure $\mathrm{Leb}_{S_\ell}\times\GQD_1$, where $S_\ell$ is the circle with radius $(2\pi)^{-1}\ell$, and concatenate the $\cD_u$ to $\cC$  {to} get a ring of forested quantum disks according to the quantum length induced by $u$. We call a sample from  {$\cM^{\mathrm{f.c.}}$} a \emph{forested circle}, where $\ell$ is its quantum length.
\end{definition}

 {Note that we do not have the index 2 in the notation $\cM^{\mathrm{f.c.}}$ since there are no marked points on the forested circle as in Definition~\ref{def:forested-circle}.}   {Following Proposition~\ref{prop:typical} and Definition~\ref{def:forested-circle}, for $\ell>0$, the function $\ell'\mapsto |\ell \cM^{\mathrm{f.c.}}(\ell;\ell')|$ agrees with the density function of $Y_\ell$, where $(Y_t)_{t>0}$ is the process in Definition~\ref{def:forested-line}.}  The following result is the analog of Corollary~\ref{cor:QApinched}.
\begin{lemma}\label{lem:forested-ring}
    The following two procedures agree:
    \begin{enumerate}[(i)]
        \item Sample a forested circle, and weight its law by its generalized quantum length. Then mark a point on the forested boundary according to the generalized quantum length measure.
        \item Sample $(\cD^f,\cL^f)$ from $ \Mfd_2(\gamma^2-2)\times\cM_2^{\mathrm{f.l.}}$. Then glue the two endpoints of $\cL^f$ together with one endpoint of $\cD^f$.
    \end{enumerate}
\end{lemma}
\begin{proof}
    The proof is identical to that of Lemma~\ref{lem:QApinchedd} and Corollary~\ref{cor:QApinched} following the same argument where one  {replaces} the measure $\Md_2(\gamma^2-W)$ with $\GQD_1$.
\end{proof}

Recall that by \eqref{eq:disint}, one can disintegrate the quantum disk measure over quantum length.  One can similarly define a disintegration of the measure $\mathcal M_2^\mathrm{f.l.}$ by   disintegrating over the values of $Y_t$: 
\[
\mathcal{M}_2^\mathrm{f.l.} = \int_{\bbR_+^2}\mathcal{M}_2^\mathrm{f.l.}(\ell;\ell')\,d\ell\,d\ell'.
\]
where $\mathcal{M}_2^\mathrm{f.l.}(\ell;\ell')$ is the measure on forested line segments with quantum length $\ell$ for the line boundary arc and generalized quantum length $\ell'$ for the forested boundary arc. 
Similarly, we can define a disintegration over $ {\cM^{\mathrm{f.c.}}}$ over its quantum length and generalized quantum length, namely
 \begin{equation}
      {\cM^{\mathrm{f.c.}}} = \int_{\bbR_+^2}  {\cM^{\mathrm{f.c.}}}(\ell;\ell')\, d\ell' d\ell 
 \end{equation}
 where the forested circle has quantum length $\ell$ and generalized quantum length $\ell'$. Indeed, for fixed quantum length $\ell$, this follows from the same disintegration of the forested line segment $\cM_2^{\mathrm{f.l.}}$ over generalized quantum length. Then we define the measure $\QD_{1,0}^f$ through
 \begin{equation*}
     \QD_{1,0}^f = \int_{\bbR_+^2} \QD_{1,0}(\ell)\times   {\ell\cM^{\mathrm{f.c.}}}(\ell;\ell')\, d\ell' d\ell, 
 \end{equation*}
%\haoyu{** should multiply $\ell$ in the integrand? same applies to Equation (5.2) **}
i.e., we glue a forested circle to the outer boundary of a quantum disk from $\QD_{1,0}$.  {Similar to the uniform conformal welding as at the beginning of Section~\ref{sec:welding-simple}, we first independently sample a point $\mathbf{p}_1$ on the non-forested side of the forested circle, and a point  $\mathbf{p}_2$ on the boundary of the quantum disk, both according to the probability measure proportional to the LQG length measure. Then we glue the quantum disk to the forested circle together according to the LQG length measure where we identify $\mathbf{p}_1$ with $\mathbf{p}_2$ and keep the orientation.} Likewise, for $W\in(0,\frac{\gamma^2}{2})$, we define the measure $\wt\QA^f(W)$ via
 \begin{equation}\label{eq:QAfW}
    \wt\QA^f(W) = \iint_{\bbR_+^4} \wt\QA(W)(\ell_1,\ell_2)\times   {\ell_1\cM^{\mathrm{f.c.}}(\ell_1;\ell_1')\times \ell_2\cM^{\mathrm{f.c.}}(\ell_2;\ell_2')}\, d\ell_1'd\ell_2'd\ell_1 d\ell_2, 
 \end{equation}
 {where the gluing is similar as above. See Figure~\ref{fig:bcle-zipper-nonsimple} for an illustration of the forested pinched thin quantum annulus and forested quantum disk.}

 \begin{figure}
     \centering
     \includegraphics[scale=0.68]{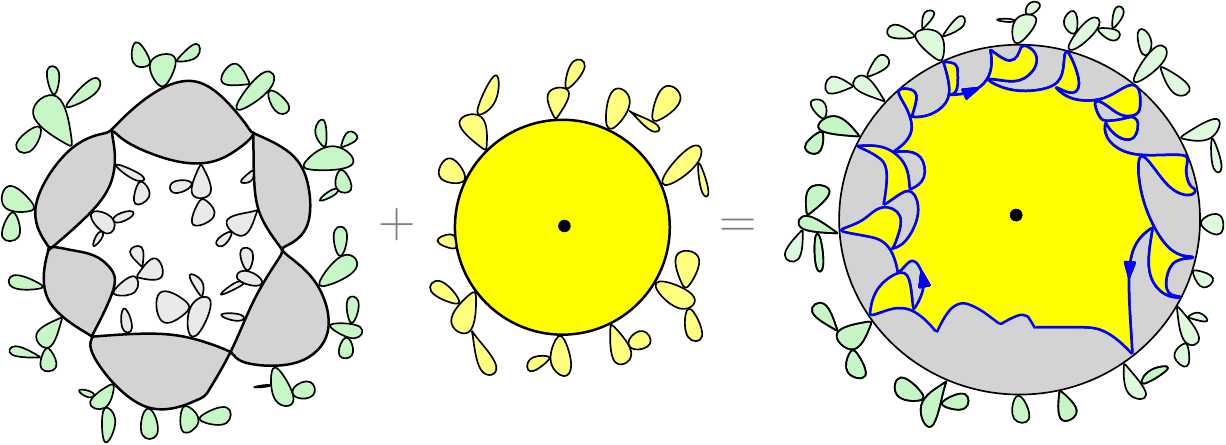}
     \caption{A forested pinched thin quantum annulus from $\wt\QA^f(W)$ (left)   {and a forested quantum disk from $ \QD_{1,0}^f$ (middle)}. In Theorem~\ref{thm:weld-BCLE-non-simple} we prove that its conformal welding with the yellow forested disk from $\QD_{1,0}^f$ gives another sample from $\QD_{1,0}^f$ decorated with independent BCLE loop.}
     \label{fig:bcle-zipper-nonsimple}
 \end{figure}

For a sample from $\QD^f_{1,0}$, we weight its law by its generalized quantum length of its boundary and sample a marked point on the boundary according to the probability measure proportional to the generalized quantum length measure. Denote the law of the output surface by $\wt\QD^f_{1,1}$. We define the measure $\wt\QA_1^f(W)$ analogously.

The following is a direct consequence of Lemma~\ref{lem:QApinchedd}, Corollary~\ref{cor:QApinched} and Lemma~\ref{lem:forested-ring}.
\begin{lemma}\label{lem:QA-f-QD-f}
    We have the following equivalences of measures.
    \begin{enumerate}[(i)]
        \item For some constant $c_1$ depending only on $\gamma$, we have $\wt\QD^f_{1,1} = c_1\QD^f_{1,1} \times \Mfd_2(\gamma^2-2)$, i.e., a sample from $\wt\QD^f_{1,1}$ can be produced by concatenating a weight $\gamma^2-2$ forested quantum disk to the boundary marked point of a forested quantum disk from  $\QD_{1,1}^f$.
        \item $\wt{\QA}_1^f(W) = \Mfd_{2,\bullet}(\gamma^2-W)\times \Mfd_2(W)$, i.e., a sample from $\wt\QA_1^f(W)$ can be produced by cyclically concatenating samples from  $\Mfd_{2,\bullet}(\gamma^2-W)\times \Mfd_2(W)$.
    \end{enumerate}
\end{lemma}

For the rest of this subsection, we derive the laws of generalized quantum lengths of some generalized quantum surfaces. The following are from~\cite[Lemma 3.3 and 3.5]{AHSY23}.
\begin{lemma}[L\'evy process moments]\label{lem:levy-moment}
Let $(Y_t)_{t\geq0}$ be the process in Definition~\ref{def:forested-line}. For $p < \frac{\gamma^2}{4}$,
\[\bbE[Y_1^p] = \frac4{\gamma^2} \frac{\Gamma(-\frac4{\gamma^2}p)}{\Gamma(-p)}. \]
Conversely, for $p \geq \frac{\gamma^2}4$, we have $\bbE[Y_1^p] = \infty$.
\end{lemma}

\begin{lemma}[Law of forested segment length]\label{lem:fl-length}
Fix $q \in \bbR$. 
Suppose we sample $\mathbf t \sim \1_{t > 0} t^{-q} dt$ and independently sample a forested line $\cL$.  
For $q<2$,  the law of  
$Y_{\mathbf t}$ is $\ C_q  \cdot \1_{L>0} L^{-\frac{\gamma^2}4q + \frac{\gamma^2}4 - 1}  {dL}$ where $C_q := \frac{\gamma^2}4 \bbE[Y_1^{\frac{\gamma^2}4 (q-1)}] = \frac{\Gamma(-(q-1))}{\Gamma(-\frac{\gamma^2}4 (q-1))}<\infty$. If $q\geq2$, then for any $0<a<b$, the event $\{Y_{\mathbf t}\in[a,b]\}$ has infinite measure. 
\end{lemma} 

\begin{proposition}\label{prop:formula-gqa-weight}
    For $\gamma \in (\sqrt{2},2)$ and $W \in (0,\frac{\gamma^2}{2})$, let $L_1'$ and $L_2'$ be the inner and outer boundary lengths of a forested quantum annulus from $\wt\QA^f(W)$. Then for any $t \in \mathbb{R}_+$ and $y \in (-\frac{\gamma^2}{4},0)$, we have
    \begin{equation}\label{eq:gqa-length}
        \wt\QA^f(W)[L_1' e^{-tL_1'} (L_2')^y] = t^{-y-1}\Gamma(y+1) (1-\frac{2}{\gamma^2}W)^{-2} \frac{\sin(\frac{\gamma^2 - 2W}{\gamma^2} \pi y)}{\sin(\frac{4}{\gamma^2} \pi y)} \,.
    \end{equation}
\end{proposition}

\begin{proof}
    Let  $(Y_t)_{t \ge 0}$ be as in Definition~\ref{def:forested-line}. By the normalization and the scaling property of $(Y_t)_{t \ge 0}$, we have $\E[e^{-\lambda Y_t}]=e^{-t \lambda^{\gamma^2/4}}$ for $\lambda>0$ and $t>0$, and thus $\E[Y_te^{-\lambda Y_t}]=\frac{\gamma^2}{4}t\lambda^{\frac{\gamma^2}{4}-1}e^{-t \lambda^{\gamma^2/4}}$. Let  $(\ol Y_t)_{t \ge 0}$ be an independent copy of $(Y_t)_{t \ge 0}$. Then by~\eqref{eq:QAfW},
    \begin{align*}
        &\quad \wt\QA^f(W)[L_1' e^{-tL_1'} (L_2')^y]=\wt\QA(W) \left[ \E[Y_{L_1} e^{-tY_{L_1}}] \cdot \E[ (\ol Y_{L_2})^y] \right] \\
        &=\wt\QA(W) \left[ \frac{\gamma^2}{4} t^{\frac{\gamma^2}{4}-1} L_1 e^{-L_1 t^{\gamma^2/4}} \cdot \frac{4}{\gamma^2} \frac{\Gamma(-\frac{4}{\gamma^2} y)}{\Gamma(-y)} L_2^{\frac{4}{\gamma^2} y} \right] \\
        &=\frac{\Gamma(-\frac{4}{\gamma^2} y)}{\Gamma(-y)} t^{\frac{\gamma^2}{4}-1} \wt\QA(W) \left[  {L_1} e^{-L_1 t^{\gamma^2/4}} L_2^{\frac{4}{\gamma^2} y} \right],
    \end{align*}
    where the  {second} equality follows from Lemma~\ref{lem:levy-moment}. The conclusion follows readily by plugging in Proposition~\ref{prop:formula-qa-weight} and using the identities $\Gamma(1+y) \Gamma(-y)=-\pi/\sin(\pi y)$ and $\Gamma(1+\frac{4}{\gamma^2}y) \Gamma(-\frac{4}{\gamma^2} y)=-\pi/\sin(\pi \frac{4}{\gamma^2} y)$.
\end{proof}

\subsection{Conformal welding for generalized quantum surfaces}\label{subsec:welding-non-simple}
As explained in~\cite{DMS21,AHSY23}, for a pair of certain generalized quantum surfaces, there exists a way to conformally weld them together according to the generalized quantum length. The key is the following proposition.

\begin{proposition}[Proposition 3.25 of~\cite{AHSY23}]\label{prop:weld:segment}
     Let  {$\kappa'\in(4,8)$ and $\gamma = \frac{4}{\sqrt{\kappa'}}$}. Consider a quantum disk $\mathcal{D}$ of weight $W=2-\frac{\gamma^2}{2}$, and let $\tilde\eta$ be the concatenation of an independent  {$\SLE_{\kappa'}(\frac{\kappa'}{2}-4;\frac{\kappa'}{2}-4)$} curve on each bead of $\cD$. Then for some constant $c$, $\tilde\eta$
 divides $\cD$ into two forested  {line} segments $\wt\cL_-,\wt\cL_+$, whose law is
 \begin{equation}\label{eq:weld:segment}
      c\int_0^\infty \mathcal{M}_2^\mathrm{f.l.}(\ell)\times\mathcal{M}_2^\mathrm{f.l.}(\ell)d\ell.
 \end{equation}
Moreover, $\wt\cL_\pm$ a.s.\ uniquely determine $(\cD,\tilde\eta)$ in the sense that $(\cD,\tilde\eta)$ is measurable with respect to the $\sigma$-algebra generated by $\wt\cL_\pm$.
\end{proposition}

In light of the last statement of Proposition~\ref{prop:weld:segment}, we call the operation of gluing of the forested line segments above  {the} conformal welding.  {The uniqueness of this version of conformal welding is built on~\cite[Theorem 1.15]{DMS21} and~\cite{McEnteggart-Miller-Qian}. See also the context above and below~\cite[Theorem 1.4]{AHSY23} for more discussions on this. We further comment that the conformal welding above can be seen as the actual conformal welding for $\kappa'\in (4,5.61)$~\cite{KMS23} thanks to the conformal removability of $\SLE_{\kappa'}$, whereas for $\kappa'$ sufficiently close to 8, $\SLE_{\kappa'}$ is not conformal removable~\cite{LZ25}. The above version of conformal welding operation introduced in Proposition~\ref{prop:weld:segment}} can easily be extended to other generalized quantum surfaces. For instance, we have the following extension of Theorem~\ref{thm:disk+QT}.
\begin{theorem}\label{thm:disk+QT-f}
   Let $\kappa'\in(4,8)$ and $\gamma=\frac{4}{\sqrt{\kappa'}}$.  Let $W,W_1,W_2,W_3>0$ with $W_2+W_3 = W_1+\gamma^2-2$. Let $\rho_-'=\frac{4}{\gamma^2}(W+2-\gamma^2)$, $\rho_+' = \frac{4}{\gamma^2}(W_2+2-\gamma^2)$, and $\rho_1' = \frac{4}{\gamma^2}(W_3+2-\gamma^2)$. Then there exists a constant  {$\ol C'=\ol C'(\gamma;W;W_1,W_2)$} such that
    \begin{equation}\label{eq:disk+QT-f}
        \QT^f(W+W_1+2-\frac{\gamma^2}{2}, W+W_2+2-\frac{\gamma^2}{2}, W_3)\otimes \SLE_{\kappa'}(\rho_-';\rho_+',\rho_1') = \ol C' \int_0^\infty \Mfd_2(W;\ell')\times \QT^f(W_1,W_2,W_3;\ell')\, d\ell'.
    \end{equation}
\end{theorem}

\begin{proof}
    Following SLE duality~\cite{zhan2008duality,MS16a}, the right boundary $\eta_R$ of an $\SLE_{\kappa'}(\rho_-';\rho_+',\rho_1')$ curve $\eta'$ in $\bbH$ from 0 to $\infty$ with force points $0^-;0^+,1$ is an $\SLE_\kappa(\kappa-4+\frac{\kappa}{4}(\rho_+'+\rho_1'),-\frac{\kappa}{4}\rho_1';\frac{\kappa}{2}-2+\frac{\kappa}{4}\rho_-')$ curve from $\infty$ to 0 with force points $+\infty,1;-\infty$. Conditioned on $\eta_R$, the left boundary $\eta_L$ of $\eta'$ is an $\SLE_\kappa(-\frac{\kappa}{2};\kappa-4+\frac{\kappa}{4}\rho_-')$ curve in the connected components of $\bbH\backslash\eta_R$ to the left of $\eta_R$, whereas conditioned on $\eta_L$ and $\eta_R$, $\eta'$ are $\SLE_{\kappa'}(\frac{\kappa'}{2}-4;\frac{\kappa'}{2}-4)$ curves in each pocket of $\bbH\backslash(\eta_L\cup\eta_R)$. Thus the theorem follows by applying Theorem~\ref{thm:disk+QT} along with Proposition~\ref{prop:weld:segment} together with the above SLE duality argument.
\end{proof}

In~\cite[Theorem 3.1]{ASYZ24}, it is shown that by welding a forested quantum triangle from $\QT^f(2-\frac{\gamma^2}{2},2-\frac{\gamma^2}{2},\gamma^2-2)$ to itself, one gets a disk from $\QD_{1,1}^f$ decorated with an independent radial $\SLE_{\kappa'}(\kappa'-6)$. The goal of this subsection is to prove the following extension. Throughout this and next subsection, for $\rho'\in(\frac{\kappa'}{2}-4,\frac{\kappa'}{2}-2)$, we set $W_- = \frac{\gamma^2}{4}\rho'+\gamma^2-2$ and $W_+ = 2-\frac{\gamma^2}{2}-\frac{\gamma^2}{4}\rho'$.
\begin{figure}
    \centering
    \includegraphics[scale=0.62]{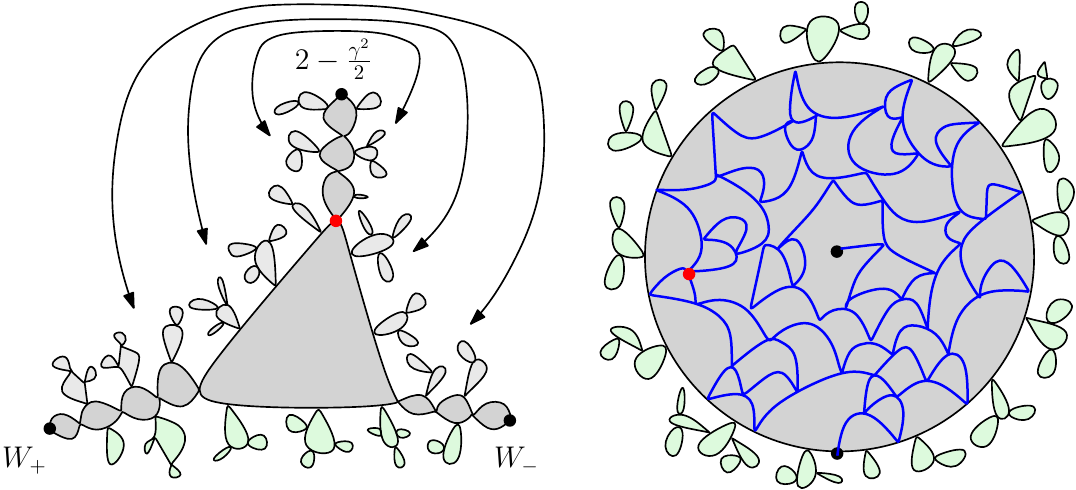}
    \caption{An illustration of the conformal welding in Theorem~\ref{thm:weld-non-simple-radial}.}
    \label{fig:thm-welding}
\end{figure}
\begin{theorem}\label{thm:weld-non-simple-radial}
   Let $\kappa'\in(4,8)$, $\gamma=\frac{4}{\sqrt{\kappa'}}$ and $\rho'\in(\frac{\kappa'}{2}-4,\frac{\kappa'}{2}-2)$. There is a constant $C_0'$ depending on $\rho'$ and $\kappa'$, such that 
    \begin{equation}\label{eq:thm-welding}
        \QD_{1,1}^f\otimes \mathrm{raSLE}_{\kappa'}(\rho';\kappa'-6-\rho') =C_0' \int_0^\infty  \big(\QT^f(2-\frac{\gamma^2}{2},W_-,W_+;\ell',\ell')\big)\,d\ell'.
    \end{equation}
    Here the left hand side of~\eqref{eq:thm-welding} stands for drawing an independent radial $\SLE_{\kappa'}(\rho';\kappa'-6-\rho')$ curve (with two force points lying immediately to the left and right of the root) on top of a forested quantum disk from $\QD^f_{1,1}$; on the right hand side of~\eqref{eq:thm-welding} the two boundary arcs containing the weight $2-\frac{\gamma^2}2$ vertex are conformally welded together.
\end{theorem}

Before proving Theorem~\ref{thm:weld-non-simple-radial}, we briefly recall some preliminaries on imaginary geometry.
 Let $D\subsetneq \mathbb{C}$ be a domain. We recall the construction the GFF on $D$ with \textit{Dirichlet} \textit{boundary conditions}  as follows. Consider the space of compactly supported smooth functions on $D$ with finite Dirichlet energy, and let $H_0(D)$ be its closure with respect to the inner product $(f,g)_\nabla=\int_D(\nabla f\cdot\nabla g)\ dx dy$. Then the (Dirichlet) GFF on $D$ is defined by 
\begin{equation}\label{eqn-def-gff}
h = \sum_{n=1}^\infty \xi_nf_n
\end{equation}
where $(\xi_n)_{n\ge 1}$ is a collection of i.i.d. standard Gaussians and $(f_n)_{n\ge 1}$ is an orthonormal basis of $H_0(D)$. The sum \eqref{eqn-def-gff} a.s.\ converges to a random distribution whose law is independent of the choice of the basis $(f_n)_{n\ge 1}$. For a function $g$ defined on $\partial D$ with harmonic extension $f$ in $D$ and a zero boundary GFF $h$, we say that $h+f$ is a GFF on $D$ with boundary condition specified by $g$. 
 See \cite[Section 4.1.4]{DMS21} for more details. 

 For $\kappa'>4$, let
$$  \kappa = \frac{16}{\kappa'}, \qquad  \lambda = \frac{\pi}{\sqrt\kappa}, \qquad  \lambda' = \frac{\pi\sqrt{\kappa}}{4}, \qquad \chi = \frac{2}{\sqrt\kappa}-\frac{\sqrt\kappa}{2}  . $$
 Given a Dirichlet GFF $h^{\rm IG}$ on $\bbH$ with piecewise boundary conditions and $\theta\in\bbR$, it is possible to construct the \emph{$\theta$-angle flow lines} $\eta_{\theta}^{z}$ of $h^{\rm IG}$ starting from $z\in \overline{\bbH}$ as shown in \cite{MS16a, ig4}. Informally, $\eta_{\theta}^{z}$ is the solution to the ODE $(\eta_{\theta}^{z})'(t) = \exp(ih^{\rm IG}(\eta_{\theta}^{z}(t))/\chi +\theta)$.  When $z\in\bbR$ and the flow line is targeted at $\infty$,  as shown in~\cite[Theorem 1.1]{MS16a}, $\eta_{\theta}^{z}$ is an $\SLE_{\kappa}(\underline\rho)$ process. This construction works for $\kappa'>4$ as well, where the corresponding variants of $\SLE_{\kappa'}$ processes are referred  {to} as the \emph{counterflowlines} of $h^{\rm IG}$. See also~\cite[Section 2]{ang2024radial} for more details on radial SLE and imaginary geometry.

 We will need the following statement. 
 \begin{proposition}\label{prop:radial-IG}
  Fix $a\in(\lambda'(3-\kappa'),\lambda')$.  Let $h^{\rm IG}$ on $\bbH$ be a Dirichlet GFF with boundary value $a$ on $\bbR$. Then the counterflowline $\eta'$ of $h^{\rm IG}$ from $\infty$ targeted at $i$ has the law radial $\SLE_{\kappa'}(\frac{a}{\lambda'}+\kappa'-5;-\frac{a}{\lambda'}-1)$. The left and right boundaries $\eta^L$ and $\eta^R$ are the flow lines of $h^{\rm IG}$ from $i$ with angle $\frac{\pi}{2}$ and $-\frac{\pi}{2}$, respectively. Moreover, conditioned on $\eta^L$ and $\eta^R$, the rest of $\eta'$ has the law chordal $\SLE_{\kappa'}(\frac{\kappa'}{2}-4;\frac{\kappa'}{2}-4)$ in each of the connected component of $\bbH\backslash(\eta^L\cup\eta^R)$ between $\eta^L$ and $\eta^R$.
 \end{proposition}

 \begin{proof}
     Following the same construction as in~\cite[Theorem 3.1]{ig4},
     the counterflowline of $h^{\rm IG}$ is the radial  $\SLE_{\kappa'}(\frac{a}{\lambda'}+\kappa'-5;-\frac{a}{\lambda'}-1)$. The rest of the proposition follows from~\cite[Theorem 4.1]{ig4}.  
 \end{proof}

\begin{figure}
    \centering
    \begin{tabular}{ccc}
      \includegraphics[scale=0.4]{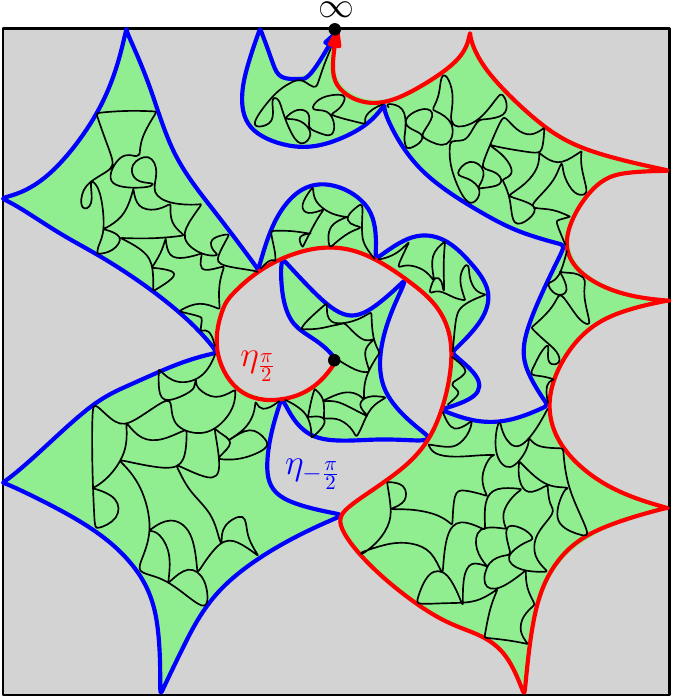}    & \includegraphics[scale=0.4]{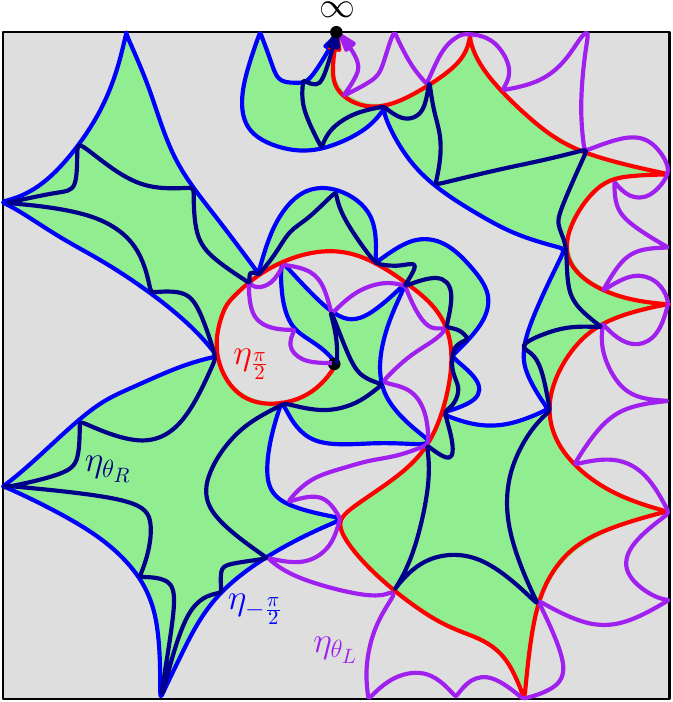}  & \includegraphics[scale=0.4]{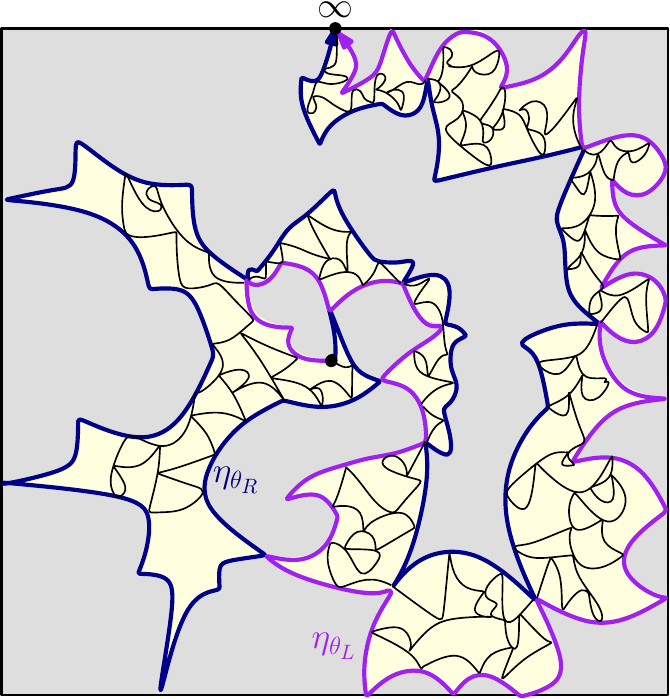} 
    \end{tabular}
    \caption{An illustration for the proof of Theorem~\ref{thm:weld-non-simple-radial}. \textbf{Left:} The red and blue curves are the flow lines of $h^{\rm IG}$ with angles $\frac{\pi}{2}$ and $-\frac{\pi}{2}$, while the black curves are chordal $\SLE_{\kappa'}(\frac{\kappa'}{2}-4;\frac{\kappa'}{2}-4)$ curves in connected components between $\eta_{\frac{\pi}{2}}$ and $\eta_{-\frac{\pi}{2}}$. By Proposition~\ref{prop:radial-IG}, the black curves with the red and blue curve as boundary form the radial $\SLE_{\kappa'}(\kappa'-6;0)$ curve $\eta_0'$. Then by~\cite[Theorem 3.1]{ASYZ24} and Proposition~\ref{prop:weld:segment}, $\eta_{\frac{\pi}{2}}$ and $\eta_{-\frac{\pi}{2}}$ cut $(\bbH,h,i,\infty)$ into a quantum triangle of weight $(2-\frac{\gamma^2}{2},\gamma^2-2,2-\frac{\gamma^2}{2})$ and a  quantum disk of weight $2-\frac{\gamma^2}{2}$ as in~\eqref{eq:thm:weld-non-simple-1}. \textbf{Middle:} We draw the purple and dark blue flow lines of $h^{\rm IG}$ with angle $\theta_L$ and $\theta_R$. By Theorem~\ref{thm:disk-welding} and Theorem~\ref{thm:disk+QT}, this further cuts $(\bbH,h,i,\infty)$ into a quantum triangle and three quantum disks as in~\eqref{eq:thm:weld-non-simple-2}. \textbf{Right:} By Theorem~\ref{thm:disk-welding} and Theorem~\ref{thm:disk+QT}, we may glue the quantum disks and quantum triangles in~\eqref{eq:thm:weld-non-simple-2} with $\eta_{\frac{\pi}{2}}$ and $\eta_{-\frac{\pi}{2}}$ being the interface. This implies that the law of $(\bbH,h, \eta_{\theta_L},\eta_{\theta_R}, i,\infty)/\sim_\gamma$ is described by~\eqref{eq:thm:weld-non-simple-3}. If we further draw  the black $\SLE_{\kappa'}(\frac{\kappa'}{2}-4;\frac{\kappa'}{2}-4)$ curves, then the collection of the black curves with purple and dark blue curves being the outer boundary form the radial $\SLE_{\kappa'}(\rho';\kappa'-6-\rho')$ curve $\eta'$, and the theorem follows by Proposition~\ref{prop:weld:segment}.} 
    \label{fig:weld-non-simple-pf}
\end{figure}

\begin{proof}[Proof of Theorem~\ref{thm:weld-non-simple-radial}]
The $\rho'=\kappa'-6$ case is precisely~\cite[Theorem 3.1]{ASYZ24}. We first work on the case where $\rho'\in(\frac{\kappa'}{2}-4,\kappa'-6)$. Let $(\bbH,h,i,\infty)$ be an embedding of a sample from $\QD_{1,1}$, and $h^{\rm IG}$ on $\bbH$ be an independent Dirichlet GFF with boundary value $-\lambda'$. For $\theta\in\bbR$, let $\eta_\theta$ be the flow line of $h^{\rm IG}$ from $i$ with angle $\theta$. Also let $\eta_0'$ be the counterflow line of $h^{\rm IG}$ targeted at $i$, which has the law radial  $\SLE_{\kappa'}(\kappa'-6;0)$. Then it follows from the $\rho'=\kappa'-6$ case that the law of $(\bbH,h,\eta_0',i,\infty)/\sim_\gamma$ (after foresting the boundary) equals a constant times the right hand side of~\eqref{eq:thm-welding} with $W_-=2-\frac{\gamma^2}{2}$ and $W_+=\gamma^2-2$. Furthermore, combining Proposition~\ref{prop:weld:segment} along with Proposition~\ref{prop:radial-IG},  the law of $(\bbH,h,\eta_{\frac{\pi}{2}},\eta_{-\frac{\pi}{2}},i,\infty)/\sim_\gamma$ equals a constant times
\begin{equation}\label{eq:thm:weld-non-simple-1}
    \iint_{\bbR_+^2} \QT(2-\frac{\gamma^2}{2},\gamma^2-2,2-\frac{\gamma^2}{2};\ell_1,\ell_2)\times \Md_2(2-\frac{\gamma^2}{2};\ell_1,\ell_2)\, d\ell_1d\ell_2
\end{equation}
where the gluing is along the boundary arc connecting the two weight $2-\frac{\gamma^2}{2}$ vertices and the boundary arc immediately to counterclockwise to it in the quantum triangle. 

Now let 
\begin{equation}\label{eq:thm:weld-non-simple-angle}
\theta_L = \frac{\lambda}{\chi}(\frac{5\gamma^2}{4}+\frac{\gamma^2}{4}\rho'-3); \ \ \theta_R = \theta_L-\pi. 
\end{equation}
Then following the imaginary geometry theory in~\cite{MS16a}, conditioned on $\eta_{\frac{\pi}{2}}$ and $\eta_{-\frac{\pi}{2}}$, $\eta_{\theta_L}$ is chordal $\SLE_\kappa(2-\frac{3\gamma^2}{2}-\frac{\gamma^2}{4}\rho';\frac{\gamma^2}{4}\rho'+\gamma^2-4)$ in each connected component of $\bbH\backslash(\eta_{\frac{\pi}{2}}\cup\eta_{-\frac{\pi}{2}})$ to the right of $\eta_{\frac{\pi}{2}}$, while $\eta_{\theta_R}$ is chordal $\SLE_\kappa(2-\frac{3\gamma^2}{2}-\frac{\gamma^2}{4}\rho';\frac{\gamma^2}{4}\rho'+\gamma^2-4)$ in each connected component of $\bbH\backslash(\eta_{\frac{\pi}{2}}\cup\eta_{-\frac{\pi}{2}})$ to the left of $\eta_{\frac{\pi}{2}}$. Thus by Theorem~\ref{thm:disk-welding} and Theorem~\ref{thm:disk+QT}, we observe that the law of $(\bbH,h,\eta_{\frac{\pi}{2}},\eta_{-\frac{\pi}{2}},\eta_{\theta_L},\eta_{\theta_R}, i,\infty)/\sim_\gamma$ equals a constant times
\begin{equation}\label{eq:thm:weld-non-simple-2}
\begin{split}
 \iint_{\bbR_+^4} &\QT(W_-,\gamma^2-2,W_-;\ell_1,\ell_3)\times \Md_2(2-\frac{\gamma^2}{2}-W_-;\ell_3,\ell_2)\times \\&\Md_2(W_-;\ell_2,\ell_4)\times\Md_2(2-\frac{\gamma^2}{2}-W_-;\ell_4,\ell_1) \, d\ell_1d\ell_2d\ell_3d\ell_4.
 \end{split}
\end{equation}
If we start from $\eta_{-\frac{\pi}{2}}$, then the four surfaces in~\eqref{eq:thm:weld-non-simple-2}, which we label by $(\cT,\cD_1,\cD_2,\cD_3),$ are aligned counterclockwise. On the other hand, by Theorem~\ref{thm:disk-welding} and Theorem~\ref{thm:disk+QT}, we may also first glue $\cD_1$ with $\cD_2$, and $\cT$ with $\cD_3$. Then it follows that  the law of $(\bbH,h, \eta_{\theta_L},\eta_{\theta_R}, i,\infty)/\sim_\gamma$ equals a constant times
\begin{equation}\label{eq:thm:weld-non-simple-3}
    \iint_{\bbR_+^2} \QT(2-\frac{\gamma^2}{2},W_+,W_-;\ell_3,\ell_4)\times \Md_2(2-\frac{\gamma^2}{2};\ell_3,\ell_4)\, d\ell_3d\ell_4.
\end{equation}
On the other hand, $ \eta_{\theta_L}$ and $\eta_{\theta_R}$ are flow lines of $h^{\rm IG}+(\theta_L-\frac{\pi}{2})\chi$ from $i$ with angle $\frac{\pi}{2}$ and $-\frac{\pi}{2}$. If we further draw independent $\SLE_{\kappa'}(\frac{\kappa'}{2}-4;\frac{\kappa'}{2}-4)$ curves $\eta_D'$ within each connected component $D$ of $\bbH\backslash (\eta_{\theta_L}\cup\eta_{\theta_R})$ which lies to the left of $\eta_{\theta_L}$ and right of $\eta_{\theta_R}$, then by Proposition~\ref{prop:radial-IG}, the union $\eta'$ of $\eta_D'$'s (with $\eta_{\theta_L}$ and $\eta_{\theta_R}$ as boundary)  {forms} radial $\SLE_{\kappa'}(\rho';\kappa'-6-\rho')$ curve from $\infty$ to $i$. Furthermore,  $\eta'$ is drawn on the weight $2-\frac{\gamma^2}{2}$ quantum disk in~\eqref{eq:thm:weld-non-simple-3}. Therefore, by Proposition~\ref{prop:weld:segment}, the law of  $(\bbH,h,\eta',i,\infty)/\sim_\gamma$ (after foresting the boundary) equals a constant times the right hand side of~\eqref{eq:thm-welding}. This concludes the proof for $\rho'\in(\frac{\kappa'}{2}-4,\kappa'-6]$. 

For $\rho'\in(\kappa'-6,\frac{\kappa'}{2}-2)$, we begin with the same quantum surface $(\bbH,h,i,\infty)/\sim_\gamma$, the imaginary geometry field $h^{\mathrm {IG}}$ and the counterflowline $\eta_0'$. We draw the flow lines $\eta_{\theta_L}$ and $\eta_{\theta_R}$ of $h^{\rm IG}$ with $\theta_L,\theta_R$ as given in~\eqref{eq:thm:weld-non-simple-angle}. Then the claim follows from the same argument as the $\rho'\in(\frac{\kappa'}{2}-4,\kappa'-6)$ case. We omit the details.
\end{proof}

\subsection{Forested quantum annulus and non-simple BCLE loop}\label{subsec:pf-non-simple-weld}
The goal of this  section is to prove Theorem~\ref{thm:weld-BCLE-non-simple}, which is the analog of Theorem~\ref{thm:weld-BCLE} for the non-simple case.
Consider a forested quantum triangle $\cT^f$ of weights $(2-\frac{\gamma^2}{2},W_+,W_-)$ in Theorem~\ref{thm:weld-non-simple-radial}. By Definition~\ref{def:thin-qt}, we have the decomposition $(\cT_1^f,\cD^f)$ of $\cT^f$:
\eqb\label{eq:qt-decom}
(\cT_1^f,\cD^f)\sim  \QT^f(\frac{3\gamma^2}{2}-2,W_+,W_-)\times \Mfd_{2}(2-\frac{\gamma^2}{2}).
\eqe
In other words, $\cT^f$ can be generated by  connecting $(\cT_1^f,\cD^f)$ sampled from~\eqref{eq:qt-decom} as in Definition~\ref{def:thin-qt}.
We write $L_1'$ and $L_2'$ for the generalized boundary lengths for the left and right boundary arcs of $\cT_1$; see Figure~\ref{fig:qt-decomposition} for an illustration.

Consider the conformal welding of $\cT^f$ as in Theorem~\ref{thm:weld-non-simple-radial} and let $\eta'$ be the interface. Since the left and right boundaries of $\cT^f$ are glued together according to the generalized quantum length, as explained in~\cite{ASYZ24},  on the event $\{L_1'>L_2'\}$, a fraction of the right boundary of $\cD^f$ is glued to a fraction of the left boundary of $\cT_1^f$. This forces the first loop around 0 made by the radial $\SLE_{\kappa'}(\rho';\kappa'-6-\rho')$ interface $\eta'$ to be counterclockwise. On the event $\{L_1'<L_2'\}$,  the first loop around 0 made by   interface $\eta'$ will be clockwise.
\begin{figure}
    \centering
      \includegraphics[scale=0.6]{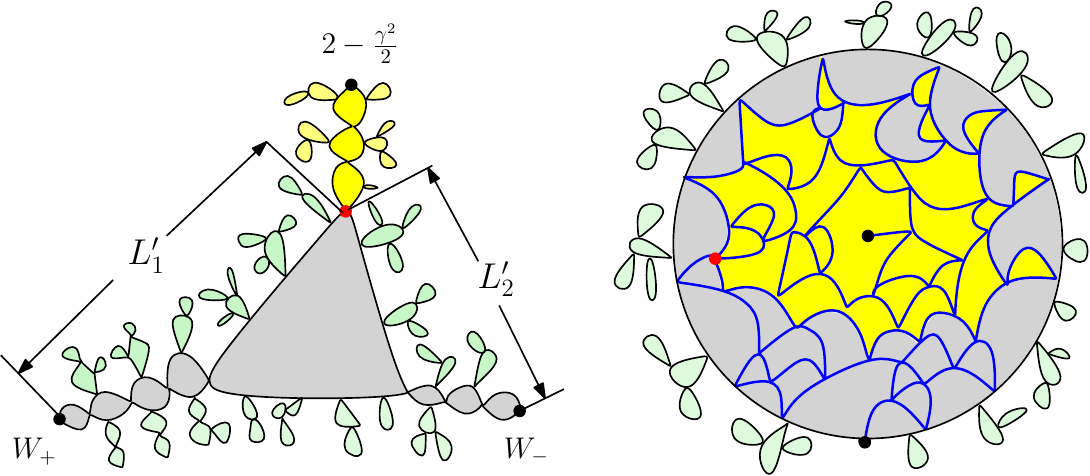}   
    \caption{\textbf{Left:} The decomposition $(\cT_1^f,\cD^f)$ of the weight $(2-\frac{\gamma^2}{2}, {W_+, W_-})$ quantum triangle. \textbf{Right:} The red marked point is the point at which $\eta$ first closes a loop around 0. Under the event $L_1'>L_2'$, the first loop is counterclockwise.}
    \label{fig:qt-decomposition}
\end{figure}

Let $\mathsf{m}$ be the law of a radial $\SLE_{\kappa'}(\rho';\kappa'-6-\rho')$ curve $\wt\eta$ from 1 to 0 with force points $1e^{i0^-}; 1e^{i0^+}$ stopped at the first time $\sigma_1$ when it closes a loop  around 0. Let $\mathsf{m}^\circlearrowleft$ (resp.\ $\mathsf{m}^\circlearrowright$) be the restriction of $\mathsf{m}$ to the event where $\wt\eta$ is counterclockwise (resp.\ clockwise).   {The following is a consequence of Theorem~\ref{thm:weld-non-simple-radial}. See Figure~\ref{fig:weld5} for an illustration.}

\begin{proposition}\label{prop:weld-wind} Let  {$\gamma \in (\sqrt2, 2)$.} For some constant $C_0'$ depending only on $\gamma$ and  {$\rho'$}, we have
\begin{align}
&\QD^f_{1,1} \otimes \mathsf{m}^\circlearrowleft(\widetilde \eta)   = C_0' \iint_{\ell_1'>\ell_2'>0} ( {\rm QT}^f (\frac{3\gamma^2}{2}-2,W_+,W_-;\ell_1',\ell_2')\times \QD^f_{1,1} {(\ell_1'-\ell_2')}) d\ell_1'd\ell_2'\,; \label{eq:weld2-1}\\
&\QD^f_{1,1} \otimes \mathsf{m}^\circlearrowright(\widetilde \eta)  = C_0' \iint_{\ell_2'>\ell_1'>0}  ( {\rm QT}^f (\frac{3\gamma^2}{2}-2,W_+,W_-;\ell_1',\ell_2')\times \QD^f_{1,1} {(\ell_2'-\ell_1')}) d\ell_1'd\ell_2'. \label{eq:weld2-2}
\end{align}
Here, for the quantum triangle we conformally weld the two forested boundary arcs adjacent to the weight $\frac{3\gamma^2}{2}-2$ vertex, starting by identifying the weight $W_-$ vertex with the weight $W_+$ vertex, and conformally welding until the shorter boundary arc has been completely welded to the longer boundary arc. Then, the quantum disk is conformally welded to the remaining segment of the longer boundary arc, identifying its boundary marked point with the weight $\frac{3\gamma^2}2-2$ vertex of the quantum triangle.
\end{proposition}

\begin{figure}[htb]

\centering
\includegraphics[scale=0.65]{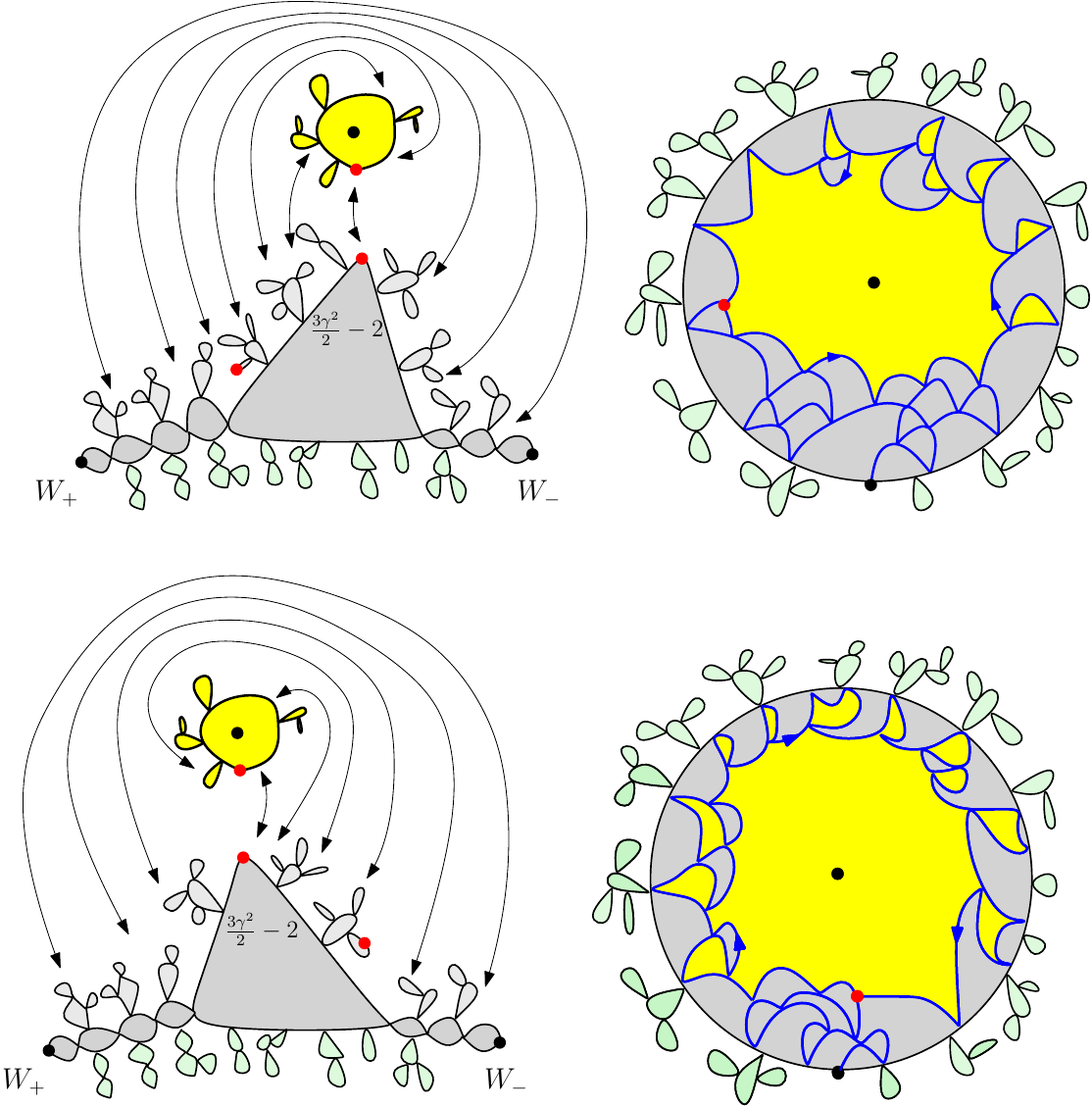}
\caption{An illustration of Proposition~\ref{prop:weld-wind}. The first panel corresponds to the case of $\ell_2'<\ell_1'$, and the second panel corresponds  {to the} case of $\ell_2'>\ell_1'$.}\label{fig:weld5}
\end{figure}

\begin{proof}
    {The proof is identical to that of \cite[Proposition 4.3]{ASYZ24} based on \cite[Theorem 3.1]{ASYZ24}, where we replace \cite[Theorem 3.1]{ASYZ24} in the proof there with Theorem~\ref{thm:weld-non-simple-radial}}.
\end{proof}

Let $(\mu')^{\circlearrowleft}$ (resp.\ $(\mu')^{\circlearrowright}$) be the law of the loop $(\cL')^o$ in the $\mathrm{BCLE}_{\kappa'}(\rho')$ restricted to the event that $(\cL')^o$ is counterclockwise (resp.\ clockwise).  {We are now ready to prove the following. See Figure~\ref{fig:bcle-zipper-nonsimple} for an illustration.}

\begin{theorem}\label{thm:weld-BCLE-non-simple}
    Let $\kappa'\in(4,8)$, $\gamma=\frac{4}{\sqrt{\kappa'}}$ and $ {\rho'}\in(\frac{\kappa'}{2}-4,\frac{\kappa'}{2}-2)$. Let $C_0'$ be the constant from Proposition~\ref{prop:weld-wind}. Let $\ol C_{11}' = \ol C'(\gamma;\gamma^2-2;W_+,\gamma^2-2)$, $\ol C_{12}' = \ol C'(\gamma;W_-;W_+,\gamma^2-2)$, $\ol C_{21}' = \ol C'(\gamma;\gamma^2-2;W_-,\gamma^2-2)$ and $\ol C_{22}' = \ol C'(\gamma;W_+;W_-,\gamma^2-2)$ be the constants from Theorem~\ref{thm:disk+QT-f}.
    Set
    \begin{align}
      & (C')^{\circlearrowleft} =   C_0'(1-\frac{2W_+}{\gamma^2})(1-\frac{2W_-}{\gamma^2})(1-\frac{\gamma^2}{4})^{-1}\ol C_{11}' (\ol C_{12}')^{-1}    \label{eq:constant-non-simple-A}
      \\ & (C')^{\circlearrowright} = C_0'(1-\frac{2W_+}{\gamma^2})(1-\frac{2W_-}{\gamma^2})(1-\frac{\gamma^2}{4})^{-1}\ol C_{21}' (\ol C_{22}')^{-1}     \label{eq:constant-non-simple-B}
    \end{align}
    Then
    \begin{align}
   & \QD_{1,0}^f\otimes (\mu')^{\circlearrowleft} =  (C')^{\circlearrowleft}\int_0^\infty \wt\QA^f(W_+;\ell')\times \ell'\,\QD_{1,0}^f(\ell')\,d\ell'; \label{eq:thm-weld-non-simple-A}\\&  \QD_{1,0}^f\otimes (\mu')^{\circlearrowright} = (C')^{\circlearrowright}\int_0^\infty \wt\QA^f(W_-;\ell')\times \ell'\,\QD_{1,0}^f(\ell')\,d\ell'. \label{eq:thm-weld-non-simple-B}
\end{align}
\end{theorem}

\begin{figure}[t]
    \centering
    \begin{tabular}{cc}
       \includegraphics[scale=0.5]{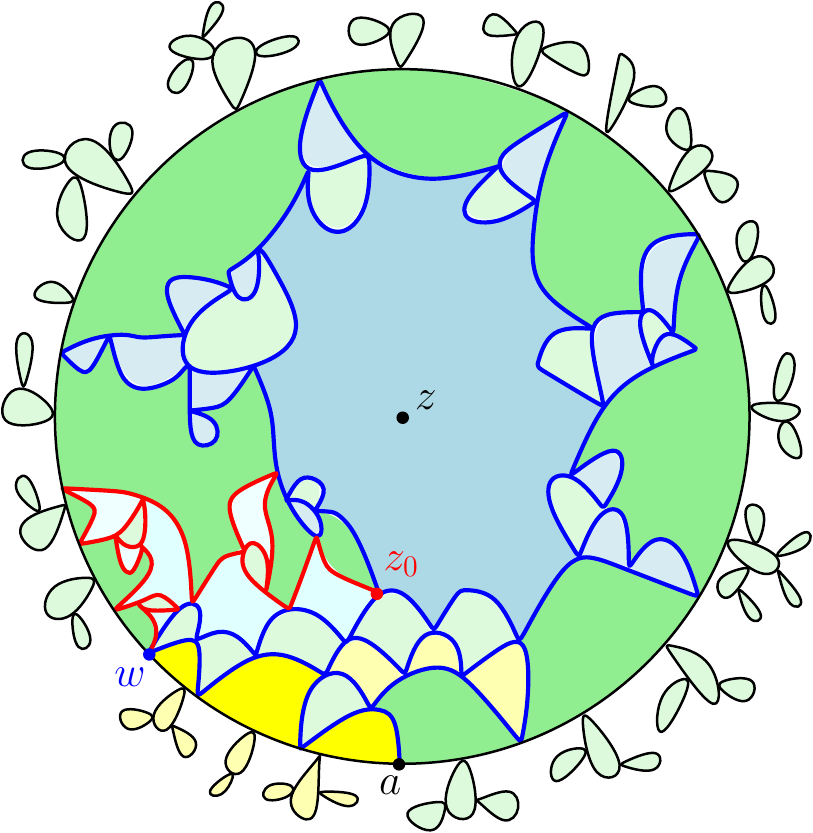}  & \includegraphics[scale=0.57]{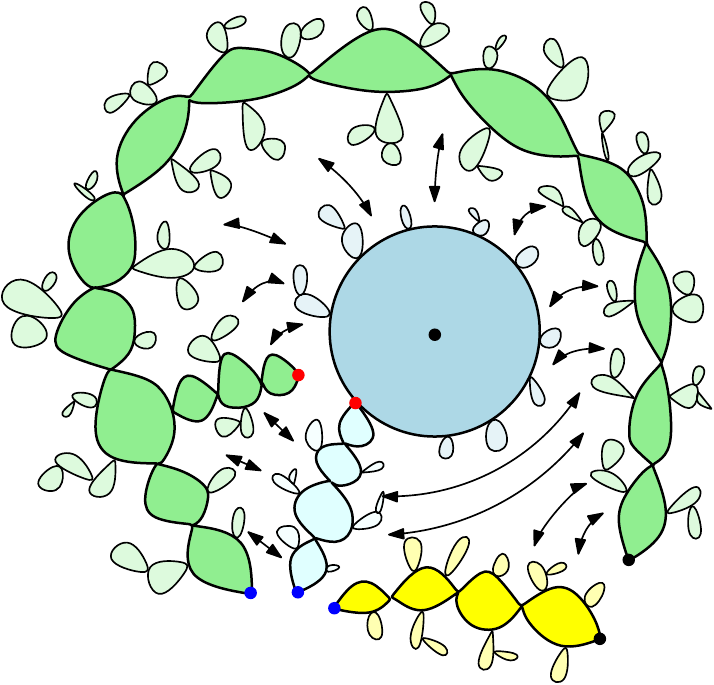}
    \end{tabular}
    \caption{An illustration of the conformal welding picture in the proof of Theorem~\ref{thm:weld-BCLE-non-simple}. We start from the picture in Proposition~\ref{prop:weld-wind} and draw the red chordal $\SLE_{\kappa'}(0;\kappa'-6-\rho')$ curve as in the left panel.   By Theorem~\ref{thm:disk+QT-f}, the whole picture is the same as the conformal welding of the 4 surfaces as on the right panel. The yellow forested disk has weight $W_-$.  By Lemma~\ref{lem:QA-f-QD-f}, the dark blue and light blue forested quantum disks form a sample from $\wt\QD^f_{1,1}$. The green forested quantum triangle has weight $(W_+,W_+,\gamma^2-2)$, which can be viewed as a sample from $\Mfd_{2,\bullet}(W_+)$  {thanks to Lemma~\ref{lem:QT-W2Wf}}. These additional boundary typical points can then be forgotten as the integral over $\ell_4'$ and $\ell_3''$ in~\eqref{eq:weld2-13} can be viewed as disintegration over the location of $z_0$ on the interface, and thus the picture is the welding of two forested quantum disks with weight $W_+$ and $W_-$ along with a forested quantum disk from $\QD^f_{1,0}$ as described in~\eqref{eq:weld2-14}. Then we weld the green surface with the yellow surface following Theorem~\ref{thm:disk+QT-f}. This output forested quantum triangle after attaching an independent sample from $\Mfd_2(\gamma^2-2)$ at the point $a$ can be formed  {into} a forested quantum annulus thanks to Lemma~\ref{lem:QA-f-QD-f}, which completes the proof.}
    \label{fig:BCLEzipper-nonsimple}
\end{figure}

\begin{proof}

 Consider a simply connected domain $(D,z,a)$ with $z\in D$ and $a\in\partial D$, and let $\Gamma'$ be an $\SLE_{\kappa'}(\rho';\kappa'-6-\rho')$ branching tree rooted at $a$. Then one can define an exploration path $(\eta')^z$ from $a$ to $z$, namely the union $\cup_{w'\in\partial D}(\eta_0')^{w'}$, where $(\eta_0')^{w'}$ is the path $(\eta')^{w'}$ targeted at ${w'}$ stopped when separating  {$w'$} from $z$. Then $(\eta')^z$ has the same law as a radial $\SLE_{\kappa'}(\rho';\kappa'-6-\rho')$ curve stopped when making a loop around $z$. Let $z_0$ be the terminal point of $(\eta')^z$. Let $w$ be the leftmost (resp.\ rightmost) point of  $(\eta')^z\cap \partial D$ when $(\eta')^z$ is a counterclockwise (resp.\ clockwise) loop, and consider the branch $(\eta')^{w}$. Let $D'_{(\eta')^z}$ be the connected component of $D\backslash(\eta')^z$ containing both $w$ and $z_0$.  Then on the event where  $(\eta')^z$ is counterclockwise, $(\eta')^{w}$ is the concatenation of $(\eta')^z$ with an $\SLE_{\kappa'}(0;\kappa'-6-\rho')$ $(\eta')^{z_0\to w}$ from $z_0$ to $w$ in  $D'_{(\eta')^z}$.
 Let $(\mu_0')^{\circlearrowleft}$ (resp.\ $(\mu_0')^{\circlearrowright}$) be the law of $(\eta')^{w}$ restricted to the event where $(\eta')^z$ is a counterclockwise (resp.\ clockwise) loop. 

 Now we work on~\eqref{eq:weld2-1}. We rewrite~\eqref{eq:weld2-1} as
 \begin{equation}\label{eq:weld2-11}
     \begin{split}
         &\QD^f_{1,1} \otimes \mathsf{m}^\circlearrowleft(\widetilde \eta)   = C_0'(1-\frac{2W_-}{\gamma^2}) \int_{\bbR_+^2}\int_0^{\ell_1'}  {\rm QT}^f (\frac{3\gamma^2}{2}-2,2-\frac{\gamma^2}{4} {\rho'},W_+;\ell_1'+\ell',\ell_1'-\ell_3')\\&\times \Mfd_2(W_-;\ell_3')\times\QD^f_{1,1}(\ell') \,d\ell_3' d\ell_1'd\ell'
     \end{split}
 \end{equation}
where we apply the change of variables $\ell' = \ell_2'-\ell_1'$ and use the decomposition 
\eqb
\begin{split}
{\rm QT}^f (\frac{3\gamma^2}{2}-2,W_-,W_+) = (1-\frac{2W_-}{\gamma^2}){\rm QT}^f (\frac{3\gamma^2}{2}-2,\gamma^2-W_-,W_+)\times  \Mfd_2(W_-).
\end{split}
\eqe
We further use $(\eta')^{z_0\to w}$ to cut the quantum triangle on the right hand side of~\eqref{eq:weld2-11}. By Theorem~\ref{thm:disk+QT-f} (note that $\gamma^2-W_-+\gamma^2-2 =\frac{3\gamma^2}{2}-2+W_+ $), we obtain
\begin{equation}\label{eq:weld2-12}
     \begin{split}
         &\QD^f_{1,1} \otimes (\mu_0')^{\circlearrowleft}   = C_0'\ol C_{11}'(1-\frac{2W_-}{\gamma^2}) \int_{\bbR_+^3}\int_0^{\ell_1'}  {\rm QT}^f (\gamma^2-2,W_+,W_+;\\&  \ell_1'+\ell',\ell_4')\times \Mfd_2(\gamma^2-2;\ell_4';\ell_1'-\ell_3') \times\Mfd_2(W_-;\ell_3')\times\QD^f_{1,1}(\ell') \,d\ell_3'\, d\ell_4' d\ell_1'd\ell'.
     \end{split}
 \end{equation}
Further performing a change of variables $\ell_3'' = \ell_1'-\ell_3'$ and $\ell'' = \ell_4'+\ell_1'-\ell_3'+\ell'$, the integral on the right hand side of~\eqref{eq:weld2-12} is equal to  
\begin{equation}\label{eq:weld2-13}
     \begin{split}
         &\int_{\bbR_+^2}\int_0^{\ell''}\int_0^{\ell''-\ell_3''}  {\rm QT}^f (\gamma^2-2,W_+,W_+; \ell''+\ell_3'-\ell_4',\ell_4')\\&\times\Mfd_2(W_-;\ell_3')\times\QD^f_{1,1}(\ell''-\ell_4'-\ell_3'') \times \Mfd_2(\gamma^2-2;\ell_4';\ell_3'') \,d\ell_4'\, d\ell_3''\, d\ell_3'd\ell''.
     \end{split}
 \end{equation}
By Lemma~\ref{lem:QA-f-QD-f}, $\QD_{1,1}^f\times \Mfd_2(\gamma^2-2) = c_1^{-1}\wt\QD_{1,1}^f$. Let $\cD_{1,1}^f$ and $\cD_2^f$ be the quantum surfaces corresponding to the last two terms in~\eqref{eq:weld2-13}, which are concatenated at the point $z_0$. Then  $\cD_{1,1}^f$ and $\cD_2^f$ together can be viewed as a single forested quantum disk from $\wt\QD_{1,1}^f$. The integral in~\eqref{eq:weld2-13} over $\ell_3''$ and $\ell_4'$ traces the generalized quantum lengths of the right and left boundaries of the loop tree in this single forested quantum disk  containing $z_0$. By Lemma~\ref{lem:QT-W2Wf}, $ {\rm QT}^f (\gamma^2-2,W_+,W_+) = \frac{4}{\gamma^2} \Mfd_{2,\bullet}(W_+)$, and we may view the forested quantum triangle in the first term of the integral~\eqref{eq:weld2-13} as a weight $W_+$ quantum disk with an additional marked point on the boundary. Moreover, the marked point is identified with the left side of $z_0$. Therefore  we can forget about this additional marked point $z_0$ on the surfaces and integrate over $\ell_4'$ and $\ell_3''$, which gives
\begin{equation}\label{eq:weld2-14}
     \begin{split}
         &\QD^f_{1,1}  \otimes (\mu_0')^{\circlearrowleft}   = C_0'\ol C_{11}'c_1^{-1}\frac{4}{\gamma^2}(1-\frac{2W_-}{\gamma^2}) \int_{\bbR_+^2}  \Mfd_{2}( W_+ ; \ell''+\ell_3')\\&\times  \QD_{1,0}^f(\ell'') \times\Mfd_2(W_-;\ell_3')    d\ell_3'd\ell''.
     \end{split}
 \end{equation}
Then as in the proof of Theorem~\ref{thm:weld-BCLE}, we mark the point on the left boundary of the weight $W_+$ forested quantum disk with distance $\ell''$ to the top vertex, where we view the measure $\Mfd_{2}( W_+ ; \ell''+\ell_3')$ as $\Mfd_{2,\bullet}( W_+ ; \ell'',\ell_3') = \frac{\gamma^2}{4}\QT^f(W_+,W_+,\gamma^2-2;\ell'',\ell_3')$. Thus by Theorem~\ref{thm:disk+QT-f}, 
\begin{equation}\label{eq:weld2-15}
     \begin{split}
         &\QD^f_{1,1}  \otimes (\mu_0')^{\circlearrowleft}   = C_0'\ol C_{11}' (\ol C_{12}')^{-1}  c_1^{-1} (1-\frac{2W_-}{\gamma^2}) \int_{\bbR_+} \QD_{1,0}^f(\ell'')   \\& \times \big(\QT^f(W_+,\gamma^2-W_+,2;\ell'')\otimes \SLE_{\kappa'}( {\rho'};0,\kappa'-6- {\rho'}) \big)    d\ell''.
     \end{split}
 \end{equation}
Then we remove the parts of the interface $(\eta')^w$ which is not on the loop $(\cL')^o$, and attach an additional sample from $\Mfd_2(\gamma^2-2)$ to the boundary marked point $a$. Note the decomposition
\eqb
\begin{split}
&\QT^f(W_+,\gamma^2-W_+,2)\times \Mfd_2(\gamma^2-2) =  (1-\frac{2W_+}{\gamma^2})(\frac{4}{\gamma^2}-1)^{-1}\QT^f(\gamma^2-W_+,\gamma^2-W_+,\gamma^2-2)\times\Mfd_2(W_+) \\&=  (1-\frac{2W_+}{\gamma^2})(1-\frac{\gamma^2}{4})^{-1}\Mfd_{2,\bullet}(\gamma^2-W_+)\times \Mfd_2(W_+) =  (1-\frac{2W_+}{\gamma^2})(1-\frac{\gamma^2}{4})^{-1}\wt\QA_1^f(W_+),
\end{split}
\eqe
where the last equation follows from Lemma~\ref{lem:QA-f-QD-f}. Therefore  
\begin{equation}\label{eq:weld2-16}
     \begin{split}
         &c_1^{-1}\wt\QD^f_{1,1} \otimes (\mu')^{\circlearrowleft}   =   c_1^{-1}(C')^{\circlearrowleft}   \int_{\bbR_+} \QD_{1,0}^f(\ell'')    \times \wt\QA_1^f(W_+;\ell'')    d\ell''.
     \end{split}
 \end{equation}
Further forgetting the marked point on the boundary   yields~\eqref{eq:thm-weld-non-simple-A}. ~\eqref{eq:thm-weld-non-simple-B} can be proved analogously.
\end{proof}

\section{Proof of Theorems \ref{thm:CR-BCLE-simple} and \ref{thm:CR-BCLE-nonsimple}}\label{sec:proof}

In this section, we give proofs to  {Theorems}~\ref{thm:CR-BCLE-simple} and~\ref{thm:CR-BCLE-nonsimple} based on the conformal welding results in Sections~\ref{sec:welding-simple} and~\ref{sec:welding-non-simple}.
Throughout this section, for $\alpha \in \mathbb{R}$, we let $\Delta_\alpha=\frac{\alpha}{2}(Q-\frac{\alpha}{2})$.
We first present the proof of the simple regime  {and start with the following special case of Definition~\ref{def:lf-bulk-boundary} with $\beta=0$.}

\begin{definition}\label{def:lf-bulk}
    For $(\alpha,w) \in \mathbb{R} \times \mathbb{H}$, let $(h,\mathbf{c})$ be sampled from $C_{\mathbb{H}}^{(\alpha,w)} P_{\mathbb{H}} \times [e^{(\alpha-Q)c} \dd c]$, where $C_{\mathbb{H}}^{(\alpha,w)}=(2 \,\mathrm{Im} w)^{-\frac{\alpha^2}{2}} |w|_+^{-2\alpha(Q-\alpha)}$. Let $\phi(z)=h(z)-2Q \log|z|_++\alpha G_{\mathbb{H}}(z,w)+\mathbf{c}$. We write $\LF_{\mathbb{H}}^{(\alpha,w)}$ for the law of $\phi$ and call $\phi$ a Liouville field on $\mathbb{H}$ with bulk insertion $(\alpha,w)$.
\end{definition}

\begin{definition}\label{QD-10-alpha}
    For $\alpha \in \mathbb{R}$, let $\phi$ be sampled from $\LF_{\mathbb{H}}^{(\alpha,i)}$. We define the infinite measure $\QD_{1,0}^\alpha$ to be  {$\frac{\gamma}{2\pi(Q-\gamma)^2}$ times} the law of $(\mathbb{H},\phi,i)/\!\!\sim_\gamma$.
\end{definition}

The following lemma is a consequence of Theorem~\ref{thm:weld-BCLE} and the reweighting argument in~\cite{ARS21, AS21}.

\begin{lemma}\label{lem:weld-B-reweight}
    Recall the setting in Theorem~\ref{thm:weld-BCLE}. For $\alpha \in \mathbb{R}$, let $\mu^{\alpha;\clockwise}$ be the measure  {defined as} $\frac{\dd \mu^{\alpha;\clockwise}}{\dd \mu^{\clockwise}}=\CR(0,D_{\mathcal{L}})^{2 \Delta_\alpha - 2}$. Then we have
    \begin{equation}\label{eq:thm-weld-B-reweight}
    \QD_{1,0}^\alpha \otimes \mu^{\alpha; \circlearrowright} = C^{\circlearrowright}\int_0^\infty \wt\QA(\rho+2;\ell)\times \ell\,\QD_{1,0}^\alpha(\ell) \dd \ell \,.
    \end{equation}
\end{lemma}

\begin{proof}
    By~\cite[Theorem 3.4]{ARS21}, we have ${\rm QD}_{1,0}^\gamma = {\rm QD}_{1,0}$.
    Therefore, by Theorem~\ref{thm:weld-BCLE},~\eqref{eq:thm-weld-B-reweight} holds with $\alpha = \gamma$. For $\alpha \neq \gamma$, let $(\phi,\mathcal{L})$ be a sample from the left hand side of~\eqref{eq:thm-weld-B-reweight} with $\alpha = \gamma$, and let $p$ be a point sampled from the harmonic measure on $\partial D_{\mathcal{L}}$ viewed from 0. Let $\psi_{\mathcal{L}}: D_{\mathcal{L}} \to \mathbb{D}$ be the conformal map fixing 0 and sending $p$ to $1$. Set $X = \phi\circ\psi_{\mathcal{L}}^{-1}+Q\log|(\psi_{\mathcal{L}}^{-1})'|$. Then the claim follows by weighting the law of $(\phi, \mathcal{L})$ by $\e^{\frac{1}{2}(\alpha^2-\gamma^2)}e^{(\alpha-\gamma)X_\e(0)}$ and sending $\e\to0$, where $X_\e(0)$ is the average of the field $X$ around the origin. The proof is identical to that of~\cite[Theorem 8.7]{AS21} and we omit the details. 
\end{proof}

\begin{proof}[Proof of Theorem~\ref{thm:CR-BCLE-simple}]
Recall from~\eqref{eq:thm-weld-const-A} and~\eqref{eq:thm-weld-const-B}  {the constants $C^{\circlearrowright} = (1-\frac{2(\kappa-4-\rho)}{\gamma^2})^2C_0C_2\ol C_2^{-2}$ and $C^{\circlearrowleft} = (1-\frac{2(\rho+2)}{\gamma^2})^2C_0C_1\ol C_1^{-2}$, where $C_0$, $C_1$, $C_2$ are the constants defined in~\eqref{eq:constant-simple-A},  $\ol C_1$ and $\ol C_2$ are the constants defined in~\eqref{eq:constant-simple-B}}. By \cite[Lemma 2.7]{ARS21}, for $\alpha>\frac{\gamma}{2}$, we have $|\QD_{1,0}^{\alpha}(\ell)|=\mathsf{C} \ell^{\frac{2}{\gamma}(\alpha-Q)-1}$ for some constant $\mathsf{C}=\mathsf{C}(\gamma;\alpha)$. In light of this, for $\alpha \in (Q-\frac{\gamma}2,Q)$, we can disintegrate~\eqref{eq:thm-weld-B-reweight} over the quantum length $r$ of the boundary of both sides of~\eqref{eq:thm-weld-B-reweight}, then integrate against $re^{-r} \dd r$, to find that 
\begin{align}\label{eq:CR-simple-clockwise}
    &\quad \E[\CR(0, D_{\mathcal{L}})^{2 \Delta_\alpha - 2} \1_{0 \in \BCLE_{\kappa}^\clockwise(\rho)}] = C^{\clockwise} \frac{\iint_{\mathbb{R}_+^2} 
    |\wt\QA(\rho+2;r,\ell)| \times \ell\,|\QD_{1,0}^{\alpha}(\ell)|\,re^{-r} \dd r \dd \ell}{\int_{\mathbb{R}_+} |\QD_{1,0}^{\alpha}(r)|\,re^{-r} \dd r} \nonumber \\
    &= C^{\clockwise} \frac{\iint_{\mathbb{R}_+^2} |\wt\QA(\rho+2;r,\ell)| \times \ell^{\frac{2}{\gamma}(\alpha-Q)} re^{-r} \dd r \dd \ell}{\int_{\mathbb{R}_+} r^{\frac{2}{\gamma}(\alpha-Q)} e^{-r} \dd r} \nonumber = C^{\clockwise} \frac{\wt\QA(\rho+2)[\ell^{\frac{2}{\gamma}(\alpha-Q)} re^{-r}]}{\Gamma(1+\frac{2}{\gamma}(\alpha-Q))} \nonumber \\
    &= C^{\clockwise} (1-\tfrac{2(\rho+2)}{\gamma^2})^{-2} \frac{\sin(\pi \frac{\gamma(Q-\alpha)}{2} \frac{\kappa-2\rho-4}{\kappa})}{\sin(\pi \frac{\gamma(Q-\alpha)}{2})} \,,
\end{align}
where we use Proposition~\ref{prop:formula-qa-weight} in the last line. In a similar manner, we can use Theorem~\ref{thm:weld-BCLE} to derive that
\begin{equation}\label{eq:CR-simple-counterclockwise}
    \E[\CR(0, D_{\mathcal{L}})^{2 \Delta_\alpha - 2} \1_{0 \notin \BCLE_{\kappa}^\clockwise(\rho)}]=C^{\counterclockwise} (1-\tfrac{2(\kappa-4-\rho)}{\gamma^2})^{-2} \frac{\sin(\pi \frac{\gamma(Q-\alpha)}{2} \frac{2\rho+8-\kappa}{\kappa})}{\sin(\pi \frac{\gamma(Q-\alpha)}{2})} \,.
\end{equation}
The goal of the rest of the proof is to compute the ratio $C^{\clockwise}/C^{\counterclockwise}$, and then determine the exact value of $C^{\clockwise}$ and $C^{\counterclockwise}$.

We first calculate the ratio $C_1/C_2$. By Theorem \ref{thm:disk-welding}, we have
\[ C_1 = C(\gamma;\kappa-4-\rho;2)=\frac{|\Md_2(\kappa-2-\rho;1)|}{\int_{\mathbb{R}_+} |\Md_2(\kappa-4-\rho;\ell)| \times |\Md_2(2;\ell,1)| \dd \ell} \,. \]

Note that Proposition~\ref{prop:qd-length} implies that $|\Md_2(W;\ell)|=\ell^{-\frac{2W}{\gamma^2}} |\Md_2(W;1)|$, and \cite[Proposition 7.8]{AHS23} further gives that $|\Md_2(2;\ell,1)|=c_0(\ell+1)^{-\frac{4}{\gamma^2}-1}$ for some constant  $c_0=c_0(\gamma)$. Hence the above expression simplifies to
\begin{align}\label{eq:C1}
    C_1 &= \frac{|\Md_2(\kappa-2-\rho;1)|}{c_0 |\Md_2(\kappa-4-\rho;1)| \int_{\mathbb{R}_+} \ell^{-\frac{2(\kappa-4-\rho)}{\gamma^2}} (\ell+1)^{-\frac{4}{\gamma^2}-1} \dd \ell} \nonumber \\
    &= \frac{|\Md_2(\kappa-2-\rho;1)|}{c_0 |\Md_2(\kappa-4-\rho;1)|} \cdot \frac{\Gamma(1+\frac{4}{\gamma^2})}{\Gamma(1-\frac{2(\kappa-4-\rho)}{\gamma^2}) \Gamma(\frac{2(\kappa-2-\rho)}{\gamma^2})} \,,
\end{align}
where the last equality follows from the identity that $\int_{\mathbb{R}_+} x^{-a} (x+1)^{-b} \dd x=\frac{\Gamma(1-a) \Gamma(a+b-1)}{\Gamma(b)}$ for $a<1$ and $a+b>1$. Similarly, 
\begin{equation}\label{eq:C2}
    C_2=C(\gamma;\rho+2;2)=\frac{|\Md_2(\rho+4;1)|}{c_0 |\Md_2(\rho+2;1)|} \cdot \frac{\Gamma(1+\frac{4}{\gamma^2})}{\Gamma(1-\frac{2(\rho+2)}{\gamma^2}) \Gamma(\frac{2(\rho+4)}{\gamma^2})} \,.
\end{equation}

Furthermore, it follows from Proposition~\ref{prop:qd-length} and~\eqref{eq:refl-coef} that $|\Md_2(W;1)|=\ol R(\gamma+\frac{2-W}{\gamma};1,0)=-R(\gamma+\frac{2-W}{\gamma};1,0)/\Gamma(2-\frac{2W}{\gamma^2})$. Using the reflection identity~\eqref{eq:refl-iden}, the reflection coefficients in the expression $C_1/C_2$ cancel out, yielding
\begin{align}\label{eq:ratio-C1-C2}
    \frac{C_1}{C_2} &= \frac{\Gamma(2-\frac{2(\kappa-4-\rho)}{\gamma^2}) \Gamma(2-\frac{2(\rho+4)}{\gamma^2})}{\Gamma(2-\frac{2(\rho+2)}{\gamma^2}) \Gamma(2-\frac{2(\kappa-2-\rho)}{\gamma^2})} \cdot \frac{\Gamma(1-\frac{2(\rho+2)}{\gamma^2}) \Gamma(\frac{2(\rho+4)}{\gamma^2})}{\Gamma(1-\frac{2(\kappa-4-\rho)}{\gamma^2}) \Gamma(\frac{2(\kappa-2-\rho)}{\gamma^2})} \nonumber \\
    &= \frac{(1-\frac{2(\rho+4)}{\kappa}) \cdot \pi/\sin(\pi (1-\frac{2(\rho+4)}{\kappa}))}{(1-\frac{2(\kappa-2-\rho)}{\kappa}) \cdot \pi/\sin(\pi (1-\frac{2(\kappa-2-\rho)}{\kappa}))} \cdot \frac{\frac{2(\rho+4)}{\kappa}-1}{\frac{2(\kappa-2-\rho)}{\kappa}-1} \nonumber \\
    &= \left( \frac{2\rho+8-\kappa}{\kappa-2\rho-4} \right)^2 \cdot \frac{\sin(\frac{2\pi}{\kappa} (\rho+2))}{\sin(\frac{2\pi}{\kappa} (\kappa-4-\rho))} \,,
\end{align}
where we use $\Gamma(x+1)=x\Gamma(x)$ and $\Gamma(x) \Gamma(1-x)=\frac{\pi}{\sin(\pi x)}$ repeatedly.

Next we turn to the calculation of $\ol C_1/\ol C_2$. By Theorem \ref{thm:disk+QT}, we have
\[ \ol C_1=\ol C(\rho+2;\kappa-4-\rho;2)=\frac{|\QT(\kappa-2,\rho+4,\rho+4;1)|}{\int_{\mathbb{R}_+} |\Md_2(\rho+2;\ell,1)| \times |\QT(\kappa-4-\rho,2,\rho+4;\ell)| \dd \ell}. \]
For the quantum triangle measure in the denominator, its corresponding $\ol \beta$ equals $2\gamma+\frac{4}{\gamma}$, and from Proposition~\ref{prop:qt-length}, we have $|\QT(\kappa-4-\rho,2,\rho+4;\ell)|=|\QT(\kappa-4-\rho,2,\rho+4;1)|$. Thus
\begin{equation}\label{eq:bar-C1-prelim}
    \ol C_1=\frac{|\QT(\kappa-2,\rho+4,\rho+4;1)|}{|\Md_2(\rho+2;1)| \cdot |\QT(\kappa-4-\rho,2,\rho+4;1)|} \,.
\end{equation}
Recall from Definition~\ref{def:thin-qt} that
\begin{align*}
    \QT(\kappa-2,\rho+4,\rho+4) &= (1-\tfrac{2(\kappa-2)}{\gamma^2})\,\QT(2,\rho+4,\rho+4)  {\times} \Md_2(\kappa-2) \,, \\
    \QT(\kappa-4-\rho,2,\rho+4) &= (1-\tfrac{2(\kappa-4-\rho)}{\gamma^2})\,\QT(\rho+4,2,\rho+4)  {\times} \Md_2(\kappa-4-\rho) \,,   
\end{align*}
Disintegrating over the boundary length between the first and second vertices and using  {Proposition~\ref{prop:qt-length}}, we have
\begin{align}\label{eq:QT-aux-1}
    & \quad |\QT(\kappa-2,\rho+4,\rho+4;1)| \nonumber \\
    &= (1-\tfrac{2(\kappa-2)}{\gamma^2}) \int_{0}^{1} |\QT(2,\rho+4,\rho+4;\ell)| \times |\Md_2(\kappa-2;1-\ell)| \dd \ell \nonumber \\
    &= (1-\tfrac{2(\kappa-2)}{\gamma^2}) \int_{0}^{1} |\QT(2,\rho+4,\rho+4;1)|\,\ell^{1-\frac{2(\rho+4)}{\gamma^2}} |\Md_2(\kappa-2;1)|\,(1-\ell)^{-\frac{2(\kappa-2)}{\gamma^2}} \dd \ell \nonumber \\
    &= (1-\tfrac{2(\kappa-2)}{\gamma^2})\,|\QT(2,\rho+4,\rho+4;1)| \cdot |\Md_2(\kappa-2;1)| \frac{\Gamma(\frac{2(\kappa-4-\rho)}{\gamma^2})\,\Gamma(1-\frac{2(\kappa-2)}{\gamma^2})}{\Gamma(1-\frac{2(\rho+2)}{\gamma^2})} \,,
\end{align}
where the last equality follows from the identity that $\int_{0}^{1} x^{-a} (1-x)^{-b} \dd x=\frac{\Gamma(1-a) \Gamma(1-b)}{\Gamma(2-a-b)}$ for $a,b<1$. Similarly, we have  
\begin{align}\label{eq:QT-aux-2}
    & \quad |\QT(\kappa-4-\rho,2,\rho+4;1)| \nonumber \\
    &= (1-\tfrac{2(\kappa-4-\rho)}{\gamma^2})\,|\QT(\rho+4,2,\rho+4;1)| \cdot |\Md_2(\kappa-4-\rho;1)|\,\Gamma(\tfrac{2(\kappa-4-\rho)}{\gamma^2}) \Gamma(1-\tfrac{2(\kappa-4-\rho)}{\gamma^2}) \,.
\end{align}
We substitute \eqref{eq:QT-aux-1} and \eqref{eq:QT-aux-2} into \eqref{eq:bar-C1-prelim}, which gives
\begin{equation}\label{eq:bar-C1}
    \ol C_1=\frac{1-\frac{2(\kappa-2)}{\gamma^2}}{1-\frac{2(\kappa-4-\rho)}{\gamma^2}} \cdot \frac{|\Md_2(\kappa-2;1)|}{|\Md_2(\rho+2;1)| |\Md_2(\kappa-4-\rho;1)|} \cdot \frac{\Gamma(1-\frac{2(\kappa-2)}{\gamma^2})}{\Gamma(1-\frac{2(\rho+2)}{\gamma^2}) \Gamma(1-\frac{2(\kappa-4-\rho)}{\gamma^2})} \,.
\end{equation}
In the same manner,  $\ol C_2=\ol C(\kappa-4-\rho,\rho+2,2)$ is equal to
\begin{equation}\label{eq:bar-C2}
\ol C_2=\frac{1-\frac{2(\kappa-2)}{\gamma^2}}{1-\frac{2(\rho+2)}{\gamma^2}} \cdot \frac{|\Md_2(\kappa-2;1)|}{|\Md_2(\kappa-4-\rho;1)| |\Md_2(\rho+2;1)|} \cdot \frac{\Gamma(1-\frac{2(\kappa-2)}{\gamma^2})}{\Gamma(1-\frac{2(\kappa-4-\rho)}{\gamma^2}) \Gamma(1-\frac{2(\rho+2)}{\gamma^2})} \,.
\end{equation}
By \eqref{eq:bar-C1} and \eqref{eq:bar-C2}, we have
\begin{equation}\label{eq:ratio-bar-C1-C2}
    \frac{\ol C_1}{\ol C_2}=\frac{1-\frac{2(\rho+2)}{\gamma^2}}{1-\frac{2(\kappa-4-\rho)}{\gamma^2}}=\frac{\kappa-2\rho-4}{2\rho+8-\kappa} \,.
\end{equation}
To sum up, using~\eqref{eq:thm-weld-const-A} and~\eqref{eq:thm-weld-const-B}, together with~\eqref{eq:ratio-C1-C2} and~\eqref{eq:ratio-bar-C1-C2}, the ratio of the coefficients in~\eqref{eq:CR-simple-clockwise} and~\eqref{eq:CR-simple-counterclockwise} equals to
\begin{equation}\label{eq:ratio-simple}
    \frac{C^{\clockwise} (1-\frac{2(\rho+2)}{\gamma^2})^{-2}}{C^{\counterclockwise} (1-\frac{2(\kappa-4-\rho)}{\gamma^2})^{-2}} = \left( \frac{2\rho+8-\kappa}{\kappa-2\rho-4} \right)^4 \cdot \frac{C_2}{C_1} \cdot \left( \frac{\ol C_1}{\ol C_2} \right)^2=\frac{\sin(\frac{2\pi}{\kappa} (\kappa-4-\rho))}{\sin(\frac{2\pi}{\kappa} (\rho+2))} \,.
\end{equation}
Note that the right-hand side of~\eqref{eq:CR-simple-clockwise} and~\eqref{eq:CR-simple-counterclockwise} sums to 1 when $\alpha=\gamma$. The exact coefficients $C^\clockwise$ and $C^\counterclockwise$ can be solved together with~\eqref{eq:ratio-simple},  {from which we conclude that for $\alpha\in(Q-\frac{\gamma}{2},Q)$,
\begin{align*}
     \E[\CR(0, D_{\mathcal{L}})^{2 \Delta_\alpha - 2} \1_{0 \in \BCLE_{\kappa}^\clockwise(\rho)}] &= \frac{\sin(\frac{\pi(4-\kappa)}{4}) \sin(\frac{2\pi}{\kappa}(\kappa-\rho-4))}{\sin(\frac{\pi(4-\kappa)}{\kappa}) \sin(\frac{\pi}{4}(\kappa - 2\rho-4))} \cdot \frac{\sin(\pi \frac{\gamma(Q-\alpha)}{2} \frac{\kappa - 2\rho-4}{\kappa} )}{\sin(\pi \frac{\gamma(Q-\alpha)}{2})} \,, \\      
     \E[\CR(0, D_{\mathcal{L}})^{2 \Delta_\alpha - 2} \1_{0 \notin \BCLE_{\kappa}^\clockwise(\rho)}] &= \frac{\sin(\frac{\pi(4-\kappa)}{4}) \sin(\frac{2\pi}{\kappa}(\rho+2))}{\sin(\frac{\pi(4-\kappa)}{\kappa}) \sin(\frac{\pi}{4}(\kappa - 2\rho-4))} \cdot \frac{\sin(\pi \frac{\gamma(Q-\alpha)}{2} \frac{2 \rho + 8 - \kappa}{\kappa})}{\sin(\pi \frac{\gamma(Q-\alpha)}{2})} \,.
 \end{align*}
~\eqref{eq:CR-BCLE-simple-A} and ~\eqref{eq:CR-BCLE-simple-B} then follow from analytic continuation in terms of $\alpha$ (see e.g.~\cite[Lemma 4.15]{NQSZ23}).  {Finally, observe that $\theta\uparrow\pi$ as $\lambda\downarrow \frac\kappa 8-1$, and $\frac{\kappa-2\rho-4}{\kappa},\frac{2\rho+8-\kappa}{\kappa}\in (1,\frac{4}{\kappa}-1)$. This implies that the right hand sides of~\eqref{eq:CR-BCLE-simple-A} and~\eqref{eq:CR-BCLE-simple-B} tends to $\infty$ as $\lambda \downarrow \frac\kappa 8-1$. Since the left hand sides of~\eqref{eq:CR-BCLE-simple-A} and~\eqref{eq:CR-BCLE-simple-B}  are decreasing in terms of $\lambda$ (the conformal radii are no greater than 1), we conclude that the left hand sides of~\eqref{eq:CR-BCLE-simple-A} and~\eqref{eq:CR-BCLE-simple-B} are  {infinite} for $\lambda\leq \frac\kappa 8-1$.}}
\end{proof}

The proof of the non-simple regime is similar, except that  {we will use the shift equations on boundary Liouville three-point functions}. We start with the following lemma.

\begin{lemma}\label{lem:formula-QT-forested}
    Suppose $W_1,W_2 \in (0,\frac{\gamma^2}{2}) \cup (\frac{\gamma^2}{2},\infty)$ and $W_3>\frac{\gamma^2}{2}$. For $i=1,2,3$, denote $\beta_i=\gamma+\frac{2-W_i}{\gamma}$, and set $\ol \beta=\beta_1+\beta_2+\beta_3$. Let $\wt \beta_i=\beta_i$ if $W_i>\frac{\gamma^2}{2}$, and $\wt \beta_i=2Q-\beta_i$ otherwise. Suppose that $(\wt \beta_1,\wt \beta_2,\wt \beta_3)$ satisfies the bounds~\eqref{eq:H-bound}, then $|\QT^f(W_1,W_2,W_3;\ell')|=c(\gamma;\bar{\beta}) |\QT(W_1,W_2,W_3;1)|\, {(\ell')}^{\frac{\gamma}{4}(\bar{\beta}-2Q)-1}$ for some constant $c(\gamma;\bar{\beta})>0$.
\end{lemma}
\begin{proof}
    Proposition~\ref{prop:qt-length} shows that $|\QT(W_1,W_2,W_3;\ell)|=|\QT(W_1,W_2,W_3;1)|\,\ell^{\frac{\bar{\beta}-2Q}{\gamma}-1}$. The conclusion then follows from Lemma~\ref{lem:fl-length}.
\end{proof}

\begin{corollary}\label{cor:QTf-length}
    For any $W \in (0,\frac{\gamma^2}{2}) \cup (\frac{\gamma^2}{2},\gamma^2)$, there is a constant $c=c(\gamma)>0$ (which does not depend on $W$) such that $|\QT^f(W,\gamma^2-2,\gamma^2-W;\ell')|=c \cdot |\QT(W,\gamma^2-2,\gamma^2-W;1)|$ for any  {$\ell'>0$}.
\end{corollary}

\begin{proof}
    This follows from Lemma ~\ref{lem:formula-QT-forested} and the corresponding  {$\bar{\beta}=\gamma+\frac{8}{\gamma}$}.
\end{proof}

 {Recall that $W_- = \frac{\gamma^2}{4}\rho'+\gamma^2-2$ and $W_+ = 2-\frac{\gamma^2}{2}-\frac{\gamma^2}{4}\rho'$, which satisfies $W_-+W_+=\frac{\gamma^2}{2}$ and $0<W_-,W_+<\frac{\gamma^2}{2}$ for $\rho' \in (\frac{\kappa'}{2}-4,\frac{\kappa'}{2}-2)$.}
\begin{proof}[Proof of Theorem~\ref{thm:CR-BCLE-nonsimple}]

Recall from~\eqref{eq:constant-non-simple-A} and~\eqref{eq:constant-non-simple-B} the constants $(C')^\clockwise$ and $(C')^\counterclockwise$. Let $\QD_{1,0}^{\alpha;f}$ be the generalized quantum surface obtained by foresting the boundary of $\QD_{1,0}^\alpha$, and let $(\mu')^{\alpha;\clockwise}$ be defined by
$\frac{\dd (\mu')^{\alpha;\clockwise}}{\dd (\mu')^\clockwise}=\mathrm{CR}(0,D_{(\cL')^o})^{2\Delta_\alpha-2}$. By Theorem~\ref{thm:weld-BCLE-non-simple} and a similar argument as Lemma~\ref{lem:weld-B-reweight}, we have
\begin{equation}\label{eq:weld-forest-reweight}
    \QD^{\alpha;f}_{1,0} \otimes (\mu')^{\alpha;\clockwise}=(C')^\clockwise \int_0^\infty \wt \QA^f(W_-;\ell') \times \ell'\,\QD^{\alpha;f}_{1,0}(\ell') \dd \ell' \,.
\end{equation}
By~\cite[Lemma 2.7]{ARS21} and Lemma~\ref{lem:fl-length}, when $\alpha>\frac{2}{\gamma}$, we have $|\QD^{\alpha;f}_{1,0}(\ell')|=\mathsf{C}' \ell^{\frac{\gamma}{2}(\alpha-Q)-1}$ for some constant $\mathsf{C}'=\mathsf{C}'(\gamma;\alpha)$. Therefore, a similar disintegration yields that for $\alpha \in (Q-\frac{\gamma}{2},Q)$,
\begin{align}\label{eq:CR-non-simple-clockwise}
    &\quad \E[\CR(0, D_{\mathcal{L}})^{2 \Delta_\alpha - 2} \1_{0 \in \BCLE_{\kappa'}^\clockwise( {\rho'})}] = (C')^{\clockwise} \frac{\iint_{\mathbb{R}_+^2} 
    |\wt\QA^f(W_-;r,\ell)| \times \ell\,|\QD_{1,0}^{\alpha;f}(\ell)|\,re^{-r} \dd r \dd \ell}{\int_{\mathbb{R}_+} |\QD_{1,0}^{\alpha;f}(r)|\,re^{-r} \dd r} \nonumber \\
    &= (C')^{\clockwise} \frac{\iint_{\mathbb{R}_+^2} |\wt\QA^f(W_-;r,\ell)| \times \ell^{\frac{\gamma}{2}(\alpha-Q)}\,re^{-r} \dd r \dd \ell}{\int_{\mathbb{R}_+} r^{\frac{2}{\gamma}(\alpha-Q)} e^{-r} \dd r} \nonumber = (C')^{\clockwise} \frac{\wt\QA^f(W_-)[\ell^{\frac{\gamma}{2}(\alpha-Q)} re^{-r}]}{\Gamma(1+\frac{\gamma}{2}(\alpha-Q))} \nonumber \\
    &= (C')^{\clockwise} (1-\tfrac{2W_-}{\gamma^2})^{-2} \frac{\sin(\pi \frac{2(Q-\alpha)}{\gamma} \frac{\kappa'-2 {\rho'}-4}{ {\kappa'}})}{\sin(\pi \frac{2(Q-\alpha)}{\gamma})} \,,
\end{align}
where we use Proposition ~\ref{prop:formula-gqa-weight} in the last equality. In the same manner we have
\begin{equation}\label{eq:CR-non-simple-counterclockwise}
    \E[\CR(0, D_{\mathcal{L}})^{2 \Delta_\alpha - 2} \1_{0 \notin \BCLE_{\kappa'}^\clockwise( {\rho'})}]=(C')^\counterclockwise (1-\tfrac{2W_+}{\gamma^2})^{-2} \frac{\sin(\pi \frac{2(Q-\alpha)}{\gamma} \frac{2 {\rho'}+8-\kappa'}{ {\kappa'}})}{\sin(\pi \frac{2(Q-\alpha)}{\gamma})} \,.
\end{equation}

It again boils down to computing the ratio $(C')^\clockwise/(C')^\counterclockwise$.
By Theorem~\ref{thm:disk+QT-f}, $\ol C'_{11}=\ol C'(\gamma;\gamma^2-2;W_+,\gamma^2-2)$ is the welding constant in
\[ \QT^f(\frac{\gamma^2}{2}+W_+,\frac{3\gamma^2}{2}-2,W_+)=\ol C'_{11}  {\int_0^\infty \left( \Mfd_2(\gamma^2-2;\ell')\times\QT^f(W_+,\gamma^2-2,W_+;\ell') \right)\,d\ell'.} \]
Removing the forested lines from non-welding parts of the boundary, and further removing the thin disks concatenated to the third vertex of the quantum triangle, we obtain the following conformal welding:
\[ \QT(\frac{\gamma^2}{2}+W_+,\frac{3\gamma^2}{2}-2,\gamma^2-W_+)=\ol C'_{11}  {\int_0^\infty \left( \mathcal{M}_2^{\mathrm{r.f.d.}}(\gamma^2-2;\ell')\times \QT^{12,f}(W_+,\gamma^2-2,\gamma^2-W_+;\ell') \right) d\ell' \,}, \]
where $\mathcal{M}_2^{\mathrm{r.f.d.}}(\gamma^2-2)$ is obtained from only foresting the right boundary arc of $\Md_2(\gamma^2-2)$, and $\QT^{12,f}(W_+,\gamma^2-2,\gamma^2-W_+)$ is obtained from only foresting the boundary arc between the  {weight $W_+$ and $\gamma^2-2$} vertex of $\QT(W_+,\gamma^2-2,\gamma^2-W_+)$. A disintegration yields that
\begin{align*}
    \ol C'_{11} &=\frac{|\QT(\frac{\gamma^2}{2}+W_+,\frac{3\gamma^2}{2}-2,\gamma^2-W_+;1)|}{\int_{\mathbb{R}_+} |\mathcal{M}_2^{\mathrm{r.f.d.}}(\gamma^2-2;1,\ell)| \cdot |\QT^f(W_+,\gamma^2-2,\gamma^2-W_+;\ell)| \dd \ell} \\
    &= \frac{|\QT(\gamma^2-W_-,\frac{3\gamma^2}{2}-2,\gamma^2-W_+;1)|}{c \cdot |\Md_2(\gamma^2-2;1) | \cdot | \QT(W_+,\gamma^2-2,\gamma^2-W_+;1) |} \,,
\end{align*}
where we use Corollary ~\ref{cor:QTf-length} in the second line. One can derive similarly that
\begin{align*}
    \ol C'_{12} &=\ol C'(\gamma;W_-,W_+,\gamma^2-2)=\frac{|\QT(W_-+W_+ +2-\frac{\gamma^2}{2},\frac{\gamma^2}{2}+W_-,\gamma^2-W_+;1)|}{\int_{\mathbb{R}_+} |\mathcal{M}_2^{\mathrm{r.f.d.}}(W_-;1,\ell)| \cdot |\QT^f(W_+,\gamma^2-2,\gamma^2-W_+;\ell)| \dd \ell} \\
    &= \frac{| \QT(2,\gamma^2-W_+,\gamma^2-W_+;1) |}{c \cdot |\Md_2 (W_-;1)| \cdot |\QT(W_+,\gamma^2-2,\gamma^2-W_+;1)|} \,,
\end{align*}
and thus
\begin{equation}\label{eq:non-simple-aux-1}
    \ol C'_{11} (\ol C'_{12})^{-1}=\frac{|\QT(\gamma^2-W_-,\frac{3\gamma^2}{2}-2,\gamma^2-W_+;1)|}{|\QT(2,\gamma^2-W_+,\gamma^2-W_+;1)|} \cdot \frac{| {\Md_2(W_-;1)}|}{|\Md_2(\gamma^2-2;1)|}.
\end{equation}
Analogous computation shows that
\begin{equation}\label{eq:non-simple-aux-2}
    \ol C'_{21} (\ol C'_{22})^{-1}=\frac{|\QT(\gamma^2-W_+,\frac{3\gamma^2}{2}-2,\gamma^2-W_-;1)|}{|\QT(2,\gamma^2-W_-,\gamma^2-W_-;1)|} \cdot \frac{|  {\Md_2(W_+;1)}|}{|\Md_2(\gamma^2-2;1)|} \,.
\end{equation}
Recall from Lemma~\ref{lem:QT-W2W} that $\Md_{2,\bullet}(W)=\frac{\gamma(Q-\gamma)}{2} \QT(W,2,W)$. Let $\Md_{2, \bullet}(W;\ell)$ be the disintegration over the quantum length $\ell$ of the boundary arc between a vertex and the additional marked point; this corresponds to the boundary arc between vertices of weight $W$ and 2 in $\QT(W,2,W)$. Then from the definition of $\Md_{2,\bullet}(W)$,
we find that for $W>\frac{\gamma^2}{2}$,
\begin{align}\label{eq:QT-W2W-length}
    |\QT(W,2,W;1)| &=(\tfrac{\gamma(Q-\gamma)}{2})^{-1} |\Md_{2, \bullet}(W;1)|=(\tfrac{\gamma(Q-\gamma)}{2})^{-1} \int_{\mathbb{R}_+} |\Md_{2}(W;\ell+1)| \dd \ell \nonumber \\
    &=(\tfrac{\gamma(Q-\gamma)}{2})^{-1} |\Md_{2}(W;1)| \int_{\mathbb{R}_+} (\ell+1)^{-\frac{2W}{\gamma^2}} \dd \ell \nonumber \\
    &=(\tfrac{\gamma(Q-\gamma)}{2})^{-1} (\tfrac{2W}{\gamma^2}-1)^{-1} |\Md_{2}(W;1)| \,,
\end{align}
where the third equality follows from Proposition~\ref{prop:qd-length}. Dividing ~\eqref{eq:non-simple-aux-1} by~\eqref{eq:non-simple-aux-2} and using~\eqref{eq:QT-W2W-length},
\begin{equation}\label{eq:non-simple-ratio}
    \frac{\ol C'_{21} (\ol C'_{22})^{-1}}{\ol C'_{11} (\ol C'_{12})^{-1}} =\frac{|\QT(\gamma^2-W_+,\frac{3\gamma^2}{2}-2,\gamma^2-W_-;1)|}{|\QT(\gamma^2-W_-,\frac{3\gamma^2}{2}-2,\gamma^2-W_+;1)|} \cdot \frac{|\Md_2(W_+;1)| |\Md_2(\gamma^2-W_+;1)|}{|\Md_2(W_-;1)| |\Md_2(\gamma^2-W_-;1)|} \cdot \frac{\gamma^2-2W_-}{\gamma^2-2W_+} \,.
\end{equation}
In the sequel, we denote $\beta_+=\frac{2+W_+}{\gamma}$, $\beta_-=\frac{2+W_-}{\gamma}$ and $\beta=\frac{2+(2-\frac{\gamma^2}{2})}{\gamma}=\frac{4}{\gamma}-\frac{\gamma}{2}$ for the parameter related to the weights $\gamma^2-W_+$, $\gamma^2-W_-$ and $\frac{3\gamma^2}{2}-2$. Then $\beta_++\beta_-=\frac{4}{\gamma}+\frac{\gamma}{2}$. For the first ratio in \eqref{eq:non-simple-ratio}, using Proposition~\ref{prop:qt-length} and~\eqref{eq:H-func}, some algebraic manipulation implies that
\begin{align*}
    &\quad \frac{|\QT(\gamma^2-W_+,\frac{3\gamma^2}{2}-2,\gamma^2-W_-;1)|}{|\QT(\gamma^2-W_-,\frac{3\gamma^2}{2}-2,\gamma^2-W_+;1)|} =\frac{\ol{H}_{(0,1,0)}^{(\beta_+,\beta,\beta_-)}}{\ol{H}_{(0,1,0)}^{(\beta_-,\beta,\beta_+)}} \\
    &=\frac{\Gamma_{\gamma/2}(Q-\beta_-)\Gamma_{\gamma/2}(\beta_+)}{\Gamma_{\gamma/2}(\frac{4}{\gamma}-\beta_-)\Gamma_{\gamma/2}(Q-\frac{4}{\gamma}+\beta_+)} \cdot \frac{\Gamma_{\gamma/2}(\frac{4}{\gamma}-\beta_+)\Gamma_{\gamma/2}(Q-\frac{4}{\gamma}+\beta_-)}{\Gamma_{\gamma/2}(Q-\beta_+)\Gamma_{\gamma/2}(\beta_-)} \\
    &=\frac{\Gamma_{\gamma/2}(\beta_+)\Gamma_{\gamma/2}(Q-\beta_-)}{\Gamma_{\gamma/2}(\beta_+-\frac{\gamma}{2})\Gamma_{\gamma/2}(Q+\frac{\gamma}{2}-\beta_-)} \cdot \frac{\Gamma_{\gamma/2}(\beta_--\frac{\gamma}{2})\Gamma_{\gamma/2}(Q+\frac{\gamma}{2}-\beta_+)}{\Gamma_{\gamma/2}(\beta_-)\Gamma_{\gamma/2}(Q-\beta_+)} \\
    &=\frac{\Gamma(\frac{\gamma}{2}(Q-\beta_-))}{\Gamma(\frac{\gamma}{2}(\beta_+-\frac{\gamma}{2}))} \cdot \frac{\Gamma(\frac{\gamma}{2}(\beta_--\frac{\gamma}{2}))}{\Gamma(\frac{\gamma}{2}(Q-\beta_+))}=\frac{\pi/\sin(\pi \frac{\gamma}{2}(\beta_--\frac{\gamma}{2}))}{\pi/\sin(\pi \frac{\gamma}{2}(\beta_+-\frac{\gamma}{2}))}=\frac{\sin(2\pi \frac{\kappa'-4- {\rho'}}{\kappa'})}{\sin(2\pi \frac{ {\rho'}+2}{\kappa'})} \,.
\end{align*}
where we use the shift equation~\eqref{eq:double-gamma-shift} in the fourth equality.

The second ratio concerning quantum disks can be handled in the same way as in the simple regime: recall that $|\Md_2(W;1)|=\ol R(\gamma+\frac{2-W}{\gamma};1,0)=-R(\gamma+\frac{2-W}{\gamma};1,0)/\Gamma(2-\frac{2W}{\gamma^2})$ and the reflection identity~\eqref{eq:refl-iden}, the reflection coefficients cancel out, leaving
\begin{align*}
    &\quad \frac{|\Md_2(W_+;1)| |\Md_2(\gamma^2-W_+;1)|}{|\Md_2(W_-;1)| |\Md_2(\gamma^2-W_-;1)|}= \frac{\Gamma(2-\frac{2W_-}{\gamma^2}) \Gamma(2-\frac{2(\gamma^2-W_-)}{\gamma^2})}{\Gamma(2-\frac{2W_+}{\gamma^2}) \Gamma(2-\frac{2(\gamma^2-W_+)}{\gamma^2})} \\
    &=\frac{(1-\frac{2W_-}{\gamma^2}) \cdot \pi/\sin(\frac{2W_-}{\gamma^2})}{(1-\frac{2W_+}{\gamma^2}) \cdot \pi/\sin(\frac{2W_+}{\gamma^2})}=\frac{\gamma^2-2W_-}{\gamma^2-2W_+} \,.
\end{align*}
Putting it all together, we find that the ratio of coefficients in~\eqref{eq:CR-non-simple-clockwise}  {and~\eqref{eq:CR-non-simple-counterclockwise}} is
\begin{equation}\label{eq:ratio-non-simple}
    \frac{(C')^\clockwise (1-\frac{2W_-}{\gamma^2})^{-2}}{(C')^\counterclockwise (1-\frac{2W_+}{\gamma^2})^{-2}}=\left( \frac{\gamma^2-2W_+}{\gamma^2-2W_-} \right)^2 \cdot \frac{\ol C'_{21} (\ol C'_{22})^{-1}}{\ol C'_{11} (\ol C'_{12})^{-1}}=\frac{\sin(2\pi \frac{\kappa'-4- {\rho'}}{\kappa'})}{\sin(2\pi \frac{ {\rho'}+2}{\kappa'})} \,.
\end{equation}
Note that the right-hand side of~\eqref{eq:CR-non-simple-clockwise} and~\eqref{eq:CR-non-simple-counterclockwise} sums to 1 when $\alpha=\gamma$. Therefore exact coefficients $(C')^\clockwise$ and $(C')^\counterclockwise$ can be solved together with~\eqref{eq:ratio-non-simple}, and for $\alpha\in(Q-\frac{\gamma}{2},Q)$, we have
\begin{align*}
    \E[\CR(0, D_{\mathcal{L}'})^{2 \Delta_\alpha - 2} \1_{0 \in \BCLE_{\kappa'}^\clockwise(\rho')}] = \frac{\sin(\frac{\pi(\kappa'-4)}{4}) \sin(\frac{2\pi}{\kappa'}(\kappa' - \rho' - 4))}{\sin(\frac{\pi(\kappa'-4)}{\kappa'}) \sin(\frac{\pi}{4}(\kappa'-2\rho' - 4))} \cdot \frac{\sin(\pi \frac{2(Q-\alpha)}{\gamma} \frac{\kappa'-2\rho'-4}{\kappa'})}{\sin(\pi \frac{2(Q-\alpha)}{\gamma})} \,, \\
    \E[\CR(0, D_{\mathcal{L}'})^{2 \Delta_\alpha - 2} \1_{0 \notin \BCLE_{\kappa'}^\clockwise(\rho')}] = \frac{\sin(\frac{\pi(\kappa'-4)}{4}) \sin(\frac{2\pi}{\kappa'}(\rho' + 2))}{\sin(\frac{\pi(\kappa'-4)}{\kappa'}) \sin(\frac{\pi}{4}(\kappa'-2\rho' - 4))} \cdot \frac{\sin(\pi \frac{2(Q-\alpha)}{\gamma} \frac{2 \rho'+8-\kappa'}{\kappa'})}{\sin(\pi \frac{2(Q-\alpha)}{\gamma})} \,.
\end{align*}
~\eqref{eq:CR-BCLE-nonsimple-A} and~\eqref{eq:CR-BCLE-nonsimple-B} then follow by analytic continuation in terms of $\alpha$,  {and the claim that the left hand sides of~\eqref{eq:CR-BCLE-nonsimple-A} and~\eqref{eq:CR-BCLE-nonsimple-B} are infinite follows similarly as the final part of the proof of Theorem~\ref{thm:CR-BCLE-simple}.}
\end{proof}

\section{Proof of Theorems~\ref{thm:touch-BCLE-4} and~\ref{thm:CR-BCLE-4} via the level-line coupling}
\label{sec:BCLE-4}

In this short section, we provide proofs of Theorems~\ref{thm:touch-BCLE-4} and~\ref{thm:CR-BCLE-4} via the coupling of Gaussian free fields and $\SLE_4$-type curves.  {These results are implicit in~\cite{ASW19}. For completeness, and to make this implication explicit, we include the proof here.}
 Before heading into the proof, we make some comments on alternative strategies. A natural approach is to show that $\BCLE_\kappa(\tilde\rho)$ converges in distribution to $\BCLE_4(\rho)$ (e.g., following the strategy of~\cite{Leh23}), and take the limit $\kappa \uparrow 4$ in Theorems~\ref{thm:touch-BCLE-simple} and~\ref{thm:CR-BCLE-simple}. However, this would require simultaneously controlling a countable collection of loops (or equivalently, the \emph{branching} $\SLE_\kappa(\tilde\rho;\kappa-6-\tilde\rho)$ process), which is more difficult to handle. Yet another possible approach is to prove the analogous conformal welding statements as Theorems~\ref{thm:weld-BCLE} and~\ref{thm:weld-BCLE-non-simple} for $\gamma=2$ (critical) Liouville quantum gravity surfaces, which we believe is of independent interest.

Recall that the zero-boundary GFF on a domain $D$ can be viewed as a Gaussian process $h$ indexed by the set of continuous  {functions} $f$ with compact support on $D$, with covariance given by
\[ \mathbb{E}[(h,f)(h,g)]=\iint_{D \times D} f(x) G_D(x,y) g(y) \dd x \dd y \,,\]
where $G_D$ is the Green's function in $D$ with Dirichlet boundary conditions. We choose the normalization so that $G_D(x,y) \sim \log(1/|x-y|)$ as $y \to x$ for $x \in D$, thus the natural height-gap~\cite{Dub09,SS09,SS13} of GFF is equal to $2\lambda$, where $\lambda=\pi/2$.

For $\kappa=4$ and $\rho \in (-2,0)$, we define $\BCLE_4^\clockwise(\rho)$ via a branching $\SLE_4(\rho;-2-\rho)$ process as in the $\kappa \in (2,4)$ regime illustrated in Section~\ref{subsec:CLE-prelim}. Let $D=\mathbb{D}$. By~\cite[Section 7.5]{MSW2017}, $\BCLE_4^\clockwise(\rho)$ can be coupled with a zero-boundary GFF $h$ on $\mathbb{D}$ such that the conditional law of $h$ inside a true (resp.\ false) loop is a GFF with boundary value $\lambda(2+\rho)$ (resp.\ $\lambda \rho$), independent from other domains. In other words, $\BCLE_4^\clockwise(\rho)$ is the collection of all $\lambda(1+\rho)$-level lines of $h$ that touch the boundary.

Since $(h,1)=\int_{\mathbb{D}} h = \int_{\mathbb{D}} h \1_{x \in \BCLE_4^\clockwise(\rho)} + \int_{\mathbb{D}} h \1_{x \notin \BCLE_4^\clockwise(\rho)}$ has zero mean, and  $\PP[x \notin \BCLE_4^\clockwise(\rho)]$ does not depend on $x$ thanks to the conformal invariance of BCLE, we get
\[ 0=\lambda(2+\rho) \cdot \PP[0 \in \BCLE_4^\clockwise(\rho)]+\lambda \rho \cdot \PP[0 \notin \BCLE_4^\clockwise(\rho)] \,, \]
which gives
\begin{equation}
    \PP[0 \in \BCLE_4^\clockwise(\rho)] = -\frac{\rho}{2} \,, \quad
    \PP[0 \notin \BCLE_4^\clockwise(\rho)] = \frac{\rho+2}{2} \,. 
\end{equation}
This establishes Theorem~\ref{thm:touch-BCLE-4}. The moment of the conformal radius of $\BCLE_4(\rho)$ may be extracted from~\cite{ASW19}. In fact, $\BCLE_4^\clockwise(\rho)$ coincides with the \emph{two-valued set} $A_{-a,b}$ defined therein with $a=-\lambda \rho$ and $b=\lambda(2+\rho)$.
Let $\mathcal{L}$ be the loop in $\BCLE_4(\rho)$ surrounding the origin (which can be either clockwise or counterclockwise) and $D_{\mathcal{L}}$ be the connected component of $\mathbb{D} \setminus \mathcal{L}$ that contains the origin. By~\cite[Proposition 20]{ASW19},
$-\log \CR(0, D_{\mathcal{L}})$ is distributed as the exit time $\tau$ from $(\frac{\rho}{2} \pi,\frac{\rho+2}{2} \pi)$ of a one-dimensional standard Brownian motion $(B_t)_{t \ge 0}$ started from 0 and, moreover, the event $\{0 \in \BCLE_4^\clockwise(\rho)\}$ corresponds to $\{B_\tau=\frac{\rho+2}{2} \pi \}$, since the boundary value inside a true loop is $b=\lambda(2+\rho)$. By Equation 3.0.5 of~\cite[Section 1.3]{BS02}, for $\mu>-\frac{1}{2}$, 
\begin{align}
    &\E[\CR(0,D_{\mathcal{L}})^\mu \1_{0 \in \BCLE_4^\clockwise(\rho)}]=\E[e^{-\mu \tau} \1_{B_\tau=\frac{\rho+2}{2} \pi}]=\frac{\sinh(\sqrt{2\mu} (-\frac{\rho}{2} \pi))}{\sinh(\sqrt{2\mu} \pi)} \,, \label{eq:CR-BCLE-4-A} \\
    &\E[\CR(0,D_{\mathcal{L}})^\mu \1_{0 \notin \BCLE_4^\clockwise(\rho)}]=\E[e^{-\mu \tau} \1_{B_\tau=\frac{\rho}{2} \pi}]=\frac{\sinh(\sqrt{2\mu} \frac{\rho+2}{2} \pi)}{\sinh(\sqrt{2\mu} \pi)} \,, \label{eq:CR-BCLE-4-B}
\end{align}
and both equal to $+\infty$ for $\mu\leq-\frac{1}{2}$. This completes the proof of Theorem~\ref{thm:CR-BCLE-4}.

\bibliographystyle{alpha}
\bibliography{theta}

@ARTICLE{ASYZ24,
      title={Boundary touching probability and nested-path exponent for nonsimple {CLE}},
  author={Ang, Morris and Sun, Xin and Yu, Pu and Zhuang, Zijie},
  journal={The Annals of Probability},
  volume={53},
  number={3},
  pages={797--847},
  year={2025},
  publisher={Institute of Mathematical Statistics}
}

@article {AHS23,
    AUTHOR = {Ang, Morris and Holden, Nina and Sun, Xin},
     TITLE = {Conformal welding of quantum disks},
   JOURNAL = {Electron. J. Probab.},
  FJOURNAL = {Electronic Journal of Probability},
    VOLUME = {28},
      YEAR = {2023},
     PAGES = {Paper No. 52, 50},
      ISSN = {1083-6489},
   MRCLASS = {60J67 (60G60)},
  MRNUMBER = {4574484},
MRREVIEWER = {G.\ V.\ Riabov},
       DOI = {10.1214/23-ejp943},
       URL = {https://doi.org/10.1214/23-ejp943},
}

@ARTICLE{ASY22,
       title={Quantum triangles and imaginary geometry flow lines},
  author={Ang, Morris and Sun, Xin and Yu, Pu},
  journal={Annales de l'Institut Henri Poincare (B) Probabilites et statistiques},
  volume={62},
  number={2},
  pages={751--801},
  year={2026},
  organization={Institut Henri Poincar{\'e}}
}

@article{smirnov2001critical,
  title={Critical exponents for two-dimensional percolation},
  author={Smirnov, Stanislav and Werner, Wendelin},
  journal={Mathematical Research Letters},
  volume={8},
  number={6},
  pages={729--744},
  year={2001},
  publisher={International Press of Boston, Inc. Somerville, MA 02143, USA}
}

@article {AHS21,
    AUTHOR = {Ang, Morris and Holden, Nina and Sun, Xin},
     TITLE = {Integrability of {SLE} via conformal welding of random
              surfaces},
   JOURNAL = {Comm. Pure Appl. Math.},
  FJOURNAL = {Communications on Pure and Applied Mathematics},
    VOLUME = {77},
      YEAR = {2024},
    NUMBER = {5},
     PAGES = {2651--2707},
      ISSN = {0010-3640,1097-0312},
   MRCLASS = {60J67 (81T40 83C45)},
  MRNUMBER = {4720219},
       DOI = {10.1002/cpa.22180},
       URL = {https://doi.org/10.1002/cpa.22180},
}

@article {SheffieldCLE,
    AUTHOR = {Sheffield, Scott},
     TITLE = {Exploration trees and conformal loop ensembles},
   JOURNAL = {Duke Math. J.},
  FJOURNAL = {Duke Mathematical Journal},
    VOLUME = {147},
      YEAR = {2009},
    NUMBER = {1},
     PAGES = {79--129},
      ISSN = {0012-7094,1547-7398},
   MRCLASS = {60J67 (60D05)},
  MRNUMBER = {2494457},
MRREVIEWER = {Robert\ Otto\ Bauer},
       DOI = {10.1215/00127094-2009-007},
       URL = {https://doi.org/10.1215/00127094-2009-007},
}

@article {Sheffield-Werner-CLE,
    AUTHOR = {Sheffield, Scott and Werner, Wendelin},
     TITLE = {Conformal loop ensembles: the {M}arkovian characterization and
              the loop-soup construction},
   JOURNAL = {Ann. of Math. (2)},
  FJOURNAL = {Annals of Mathematics. Second Series},
    VOLUME = {176},
      YEAR = {2012},
    NUMBER = {3},
     PAGES = {1827--1917},
      ISSN = {0003-486X,1939-8980},
   MRCLASS = {60J67},
  MRNUMBER = {2979861},
MRREVIEWER = {Zhen-Qing\ Chen},
       DOI = {10.4007/annals.2012.176.3.8},
       URL = {https://doi.org/10.4007/annals.2012.176.3.8},
}

@article {She16a,
    AUTHOR = {Sheffield, Scott},
     TITLE = {Conformal weldings of random surfaces: {SLE} and the quantum
              gravity zipper},
   JOURNAL = {Ann. Probab.},
  FJOURNAL = {The Annals of Probability},
    VOLUME = {44},
      YEAR = {2016},
    NUMBER = {5},
     PAGES = {3474--3545},
      ISSN = {0091-1798,2168-894X},
   MRCLASS = {60K35 (60D05 60J67)},
  MRNUMBER = {3551203},
       DOI = {10.1214/15-AOP1055},
       URL = {https://doi.org/10.1214/15-AOP1055},
}

@ARTICLE{AHSY23,
        title={Conformal welding of quantum disks and multiple {SLE}: the non-simple case},
  author={Ang, Morris and Holden, Nina and Sun, Xin and Yu, Pu},
  journal={Probability Theory and Related Fields},
  volume={194},
  number={1},
  pages={717--777},
  year={2026},
  publisher={Springer}
}

@article {ig4,
    AUTHOR = {Miller, Jason and Sheffield, Scott},
     TITLE = {Imaginary geometry {IV}: interior rays, whole-plane
              reversibility, and space-filling trees},
   JOURNAL = {Probab. Theory Related Fields},
  FJOURNAL = {Probability Theory and Related Fields},
    VOLUME = {169},
      YEAR = {2017},
    NUMBER = {3-4},
     PAGES = {729--869},
      ISSN = {0178-8051,1432-2064},
   MRCLASS = {60J67 (60K35)},
  MRNUMBER = {3719057},
MRREVIEWER = {Ben\ Dyhr},
       DOI = {10.1007/s00440-017-0780-2},
       URL = {https://doi.org/10.1007/s00440-017-0780-2},
}

@article {CK13looptree,
    AUTHOR = {Curien, Nicolas and Kortchemski, Igor},
     TITLE = {Random stable looptrees},
   JOURNAL = {Electron. J. Probab.},
  FJOURNAL = {Electronic Journal of Probability},
    VOLUME = {19},
      YEAR = {2014},
     PAGES = {no. 108, 35},
      ISSN = {1083-6489},
   MRCLASS = {60F17 (05C80 60G52)},
  MRNUMBER = {3286462},
MRREVIEWER = {D.\ Yogeshwaran},
       DOI = {10.1214/EJP.v19-2732},
       URL = {https://doi.org/10.1214/EJP.v19-2732},
}

@ARTICLE{ARS21,
      title={{FZZ} formula of boundary Liouville {CFT} via conformal welding},
  author={Ang, Morris and Remy, Guillaume and Sun, Xin},
  journal={Journal of the European Mathematical Society},
  volume={27},
  number={3},
  pages={1209--1266},
  year={2023}
}

@ARTICLE{ARS22,
       author = {{Ang}, Morris and {Remy}, Guillaume and {Sun}, Xin},
        title = "{The moduli of annuli in random conformal geometry}",
      journal = {arXiv e-prints},
         year = 2022,
        month = mar,
          eid = {arXiv:2203.12398},
        pages = {arXiv:2203.12398},
          doi = {10.48550/arXiv.2203.12398},
archivePrefix = {arXiv},
       eprint = {2203.12398},
 primaryClass = {math.PR},
       adsurl = {https://ui.adsabs.harvard.edu/abs/2022arXiv220312398A}
}

@article {DMS21,
    AUTHOR = {Duplantier, Bertrand and Miller, Jason and Sheffield, Scott},
     TITLE = {Liouville quantum gravity as a mating of trees},
   JOURNAL = {Ast\'{e}risque},
  FJOURNAL = {Ast\'{e}risque},
    NUMBER = {427},
      YEAR = {2021},
     PAGES = {viii+257},
      ISSN = {0303-1179,2492-5926},
      ISBN = {978-2-85629-941-8},
   MRCLASS = {60-02 (60D05 60J67 82B20 82B41 83C45)},
  MRNUMBER = {4340069},
       DOI = {10.24033/ast},
       URL = {https://doi.org/10.24033/ast},
}

@ARTICLE{AS21,
       author = {{Ang}, Morris and {Cai}, Gefei and {Sun}, Xin and {Wu}, Baojun},
        title = "{Integrability of Conformal Loop Ensemble: Imaginary DOZZ Formula and Beyond}",
      journal = {arXiv e-prints},
         year = 2024,
        month = sep,
          eid = {arXiv:2107.01788},
        pages = {arXiv:2107.01788},
          doi = {10.48550/arXiv.2107.01788},
archivePrefix = {arXiv},
       eprint = {2107.01788},
 primaryClass = {math-ph},
       adsurl = {https://ui.adsabs.harvard.edu/abs/2021arXiv210701788A}
}

@article {MSW2017,
    AUTHOR = {Miller, Jason and Sheffield, Scott and Werner, Wendelin},
     TITLE = {C{LE} percolations},
   JOURNAL = {Forum Math. Pi},
  FJOURNAL = {Forum of Mathematics. Pi},
    VOLUME = {5},
      YEAR = {2017},
     PAGES = {e4, 102},
      ISSN = {2050-5086},
   MRCLASS = {60J67 (60G60 60K35 81T40 82B43)},
  MRNUMBER = {3708206},
MRREVIEWER = {Ben\ Dyhr},
       DOI = {10.1017/fmp.2017.5},
       URL = {https://doi.org/10.1017/fmp.2017.5},
}

@ARTICLE{KSL22,
       title={The fuzzy Potts model in the plane: scaling limits and arm exponents},
  author={K{\"o}hler-Schindler, Laurin and Lehmkuehler, Matthis},
  journal={Probability Theory and Related Fields},
  volume={191},
  number={1},
  pages={287--359},
  year={2025},
  publisher={Springer}
}

@ARTICLE{KMS23,
       author = {{Kavvadias}, Konstantinos and {Miller}, Jason and {Schoug}, Lukas},
        title = "{Conformal removability of non-simple Schramm-Loewner evolutions}",
      journal = {arXiv e-prints},
         year = 2023,
        month = feb,
          eid = {arXiv:2302.10857},
        pages = {arXiv:2302.10857},
          doi = {10.48550/arXiv.2302.10857},
archivePrefix = {arXiv},
       eprint = {2302.10857},
 primaryClass = {math.PR},
       adsurl = {https://ui.adsabs.harvard.edu/abs/2023arXiv230210857K}
}

@article{ang2024radial,
  title={Radial conformal welding in Liouville quantum gravity},
  author={Ang, Morris and Yu, Pu},
  journal={arXiv preprint arXiv:2411.19810},
  year={2024}
}

@ARTICLE{LZ25,
       author = {{Liu}, Haoyu and {Zhuang}, Zijie},
        title = "{Schramm-Loewner evolution contains a topological Sierpi{\'n}ski carpet when $\kappa$ is close to 8}",
      journal = {arXiv e-prints},
         year = 2025,
        month = jun,
          eid = {arXiv:2506.09609},
        pages = {arXiv:2506.09609},
          doi = {10.48550/arXiv.2506.09609},
archivePrefix = {arXiv},
       eprint = {2506.09609},
 primaryClass = {math.PR},
       adsurl = {https://ui.adsabs.harvard.edu/abs/2025arXiv250609609L}
}

@article {McEnteggart-Miller-Qian,
    AUTHOR = {McEnteggart, Oliver and Miller, Jason and Qian, Wei},
     TITLE = {Uniqueness of the welding problem for {SLE} and {L}iouville
              quantum gravity},
   JOURNAL = {J. Inst. Math. Jussieu},
  FJOURNAL = {Journal of the Institute of Mathematics of Jussieu. JIMJ.
              Journal de l'Institut de Math\'{e}matiques de Jussieu},
    VOLUME = {20},
      YEAR = {2021},
    NUMBER = {3},
     PAGES = {757--783},
      ISSN = {1474-7480},
   MRCLASS = {60D05 (60J67 60K35 81T40 83C45)},
  MRNUMBER = {4260641},
       DOI = {10.1017/S1474748019000331},
       URL = {https://doi.org/10.1017/S1474748019000331},
}

@article {MSW22-simple,
    AUTHOR = {Miller, Jason and Sheffield, Scott and Werner, Wendelin},
     TITLE = {Simple conformal loop ensembles on {L}iouville quantum
              gravity},
   JOURNAL = {Ann. Probab.},
  FJOURNAL = {The Annals of Probability},
    VOLUME = {50},
      YEAR = {2022},
    NUMBER = {3},
     PAGES = {905--949},
      ISSN = {0091-1798,2168-894X},
   MRCLASS = {60J67 (60K35 83C45)},
  MRNUMBER = {4413208},
       DOI = {10.1214/21-aop1550},
       URL = {https://doi.org/10.1214/21-aop1550},
}

@article {MSW21-nonsimple,
    AUTHOR = {Miller, Jason and Sheffield, Scott and Werner, Wendelin},
     TITLE = {Non-simple conformal loop ensembles on {L}iouville quantum
              gravity and the law of {CLE} percolation interfaces},
   JOURNAL = {Probab. Theory Related Fields},
  FJOURNAL = {Probability Theory and Related Fields},
    VOLUME = {181},
      YEAR = {2021},
    NUMBER = {1-3},
     PAGES = {669--710},
      ISSN = {0178-8051,1432-2064},
   MRCLASS = {60J67 (60G52 60G60 60J80 60K35 82B27 82B41)},
  MRNUMBER = {4341084},
       DOI = {10.1007/s00440-021-01070-4},
       URL = {https://doi.org/10.1007/s00440-021-01070-4},
}

@article {SSW09,
    AUTHOR = {Schramm, Oded and Sheffield, Scott and Wilson, David B.},
     TITLE = {Conformal radii for conformal loop ensembles},
   JOURNAL = {Comm. Math. Phys.},
  FJOURNAL = {Communications in Mathematical Physics},
    VOLUME = {288},
      YEAR = {2009},
    NUMBER = {1},
     PAGES = {43--53},
      ISSN = {0010-3616},
   MRCLASS = {60K35 (60J67 82B43)},
  MRNUMBER = {2491617},
MRREVIEWER = {Dmitry Beliaev},
       DOI = {10.1007/s00220-009-0731-6},
       URL = {https://doi.org/10.1007/s00220-009-0731-6},
}

@article {zhan2008duality,
    AUTHOR = {Zhan, Dapeng},
     TITLE = {Duality of chordal {SLE}},
   JOURNAL = {Invent. Math.},
  FJOURNAL = {Inventiones Mathematicae},
    VOLUME = {174},
      YEAR = {2008},
    NUMBER = {2},
     PAGES = {309--353},
      ISSN = {0020-9910,1432-1297},
   MRCLASS = {60J67},
  MRNUMBER = {2439609},
       DOI = {10.1007/s00222-008-0132-z},
       URL = {https://doi.org/10.1007/s00222-008-0132-z},
}

@article {MS16a,
    AUTHOR = {Miller, Jason and Sheffield, Scott},
     TITLE = {Imaginary geometry {I}: interacting {SLE}s},
   JOURNAL = {Probab. Theory Related Fields},
  FJOURNAL = {Probability Theory and Related Fields},
    VOLUME = {164},
      YEAR = {2016},
    NUMBER = {3-4},
     PAGES = {553--705},
      ISSN = {0178-8051,1432-2064},
   MRCLASS = {60J67 (81T40)},
  MRNUMBER = {3477777},
MRREVIEWER = {Ben\ Dyhr},
       DOI = {10.1007/s00440-016-0698-0},
       URL = {https://doi.org/10.1007/s00440-016-0698-0},
}

@article {PPY92,
    AUTHOR = {Perman, Mihael and Pitman, Jim and Yor, Marc},
     TITLE = {Size-biased sampling of {P}oisson point processes and
              excursions},
   JOURNAL = {Probab. Theory Related Fields},
  FJOURNAL = {Probability Theory and Related Fields},
    VOLUME = {92},
      YEAR = {1992},
    NUMBER = {1},
     PAGES = {21--39},
      ISSN = {0178-8051},
   MRCLASS = {60G55 (60E07 60J65)},
  MRNUMBER = {1156448},
MRREVIEWER = {Paavo H. Salminen},
       DOI = {10.1007/BF01205234},
       URL = {https://doi.org/10.1007/BF01205234},
}

@article {Nak07,
    AUTHOR = {Nakagawa, Kenji},
     TITLE = {Application of {T}auberian theorem to the exponential decay of
              the tail probability of a random variable},
   JOURNAL = {IEEE Trans. Inform. Theory},
  FJOURNAL = {Institute of Electrical and Electronics Engineers.
              Transactions on Information Theory},
    VOLUME = {53},
      YEAR = {2007},
    NUMBER = {9},
     PAGES = {3239--3249},
      ISSN = {0018-9448,1557-9654},
   MRCLASS = {60F99 (40E05 60E10)},
  MRNUMBER = {2417689},
MRREVIEWER = {Jaap\ Geluk},
       DOI = {10.1109/TIT.2007.903114},
       URL = {https://doi.org/10.1109/TIT.2007.903114},
}

@article {Smi10,
    AUTHOR = {Smirnov, Stanislav},
     TITLE = {Conformal invariance in random cluster models. {I}.
              {H}olomorphic fermions in the {I}sing model},
   JOURNAL = {Ann. of Math. (2)},
  FJOURNAL = {Annals of Mathematics. Second Series},
    VOLUME = {172},
      YEAR = {2010},
    NUMBER = {2},
     PAGES = {1435--1467},
      ISSN = {0003-486X,1939-8980},
   MRCLASS = {60K35 (30G25 60J67 81T40 82B20)},
  MRNUMBER = {2680496},
MRREVIEWER = {Roland\ M.\ Friedrich},
       DOI = {10.4007/annals.2010.172.1441},
       URL = {https://doi.org/10.4007/annals.2010.172.1441},
}

@article {KS19,
    AUTHOR = {Kemppainen, Antti and Smirnov, Stanislav},
     TITLE = {Conformal invariance of boundary touching loops of {FK}
              {I}sing model},
   JOURNAL = {Comm. Math. Phys.},
  FJOURNAL = {Communications in Mathematical Physics},
    VOLUME = {369},
      YEAR = {2019},
    NUMBER = {1},
     PAGES = {49--98},
      ISSN = {0010-3616,1432-0916},
   MRCLASS = {82B20 (60J67)},
  MRNUMBER = {3959554},
       DOI = {10.1007/s00220-019-03437-0},
       URL = {https://doi.org/10.1007/s00220-019-03437-0},
}

@ARTICLE{KS16,
       author = {{Kemppainen}, Antti and {Smirnov}, Stanislav},
        title = "{Conformal invariance in random cluster models. II. Full scaling limit as a branching SLE}",
      journal = {arXiv e-prints},
         year = 2016,
        month = sep,
          eid = {arXiv:1609.08527},
        pages = {arXiv:1609.08527},
          doi = {10.48550/arXiv.1609.08527},
archivePrefix = {arXiv},
       eprint = {1609.08527},
 primaryClass = {math-ph},
       adsurl = {https://ui.adsabs.harvard.edu/abs/2016arXiv160908527K}
}

@article {DCST17,
    AUTHOR = {Duminil-Copin, Hugo and Sidoravicius, Vladas and Tassion,
              Vincent},
     TITLE = {Continuity of the phase transition for planar random-cluster
              and {P}otts models with {$1 \leq q \leq 4$}},
   JOURNAL = {Comm. Math. Phys.},
  FJOURNAL = {Communications in Mathematical Physics},
    VOLUME = {349},
      YEAR = {2017},
    NUMBER = {1},
     PAGES = {47--107},
      ISSN = {0010-3616,1432-0916},
   MRCLASS = {82B26 (82B27 82B43)},
  MRNUMBER = {3592746},
MRREVIEWER = {Farrukh\ Mukhamedov},
       DOI = {10.1007/s00220-016-2759-8},
       URL = {https://doi.org/10.1007/s00220-016-2759-8},
}

@article {RZ22,
    AUTHOR = {Remy, Guillaume and Zhu, Tunan},
     TITLE = {Integrability of boundary {L}iouville conformal field theory},
   JOURNAL = {Comm. Math. Phys.},
  FJOURNAL = {Communications in Mathematical Physics},
    VOLUME = {395},
      YEAR = {2022},
    NUMBER = {1},
     PAGES = {179--268},
      ISSN = {0010-3616,1432-0916},
   MRCLASS = {60D05 (60G15 60G60 81T40)},
  MRNUMBER = {4483018},
       DOI = {10.1007/s00220-022-04455-1},
       URL = {https://doi.org/10.1007/s00220-022-04455-1},
}

@article {DS11,
    AUTHOR = {Duplantier, Bertrand and Sheffield, Scott},
     TITLE = {Liouville quantum gravity and {KPZ}},
   JOURNAL = {Invent. Math.},
  FJOURNAL = {Inventiones Mathematicae},
    VOLUME = {185},
      YEAR = {2011},
    NUMBER = {2},
     PAGES = {333--393},
      ISSN = {0020-9910,1432-1297},
   MRCLASS = {81T40 (60K35)},
  MRNUMBER = {2819163},
MRREVIEWER = {Lee-Peng\ Teo},
       DOI = {10.1007/s00222-010-0308-1},
       URL = {https://doi.org/10.1007/s00222-010-0308-1},
}

@ARTICLE{SW16,
       author = {{Sheffield}, Scott and {Wang}, Menglu},
        title = "{Field-measure correspondence in Liouville quantum gravity almost surely commutes with all conformal maps simultaneously}",
      journal = {arXiv e-prints},
         year = 2016,
        month = may,
          eid = {arXiv:1605.06171},
        pages = {arXiv:1605.06171},
          doi = {10.48550/arXiv.1605.06171},
archivePrefix = {arXiv},
       eprint = {1605.06171},
 primaryClass = {math.PR},
       adsurl = {https://ui.adsabs.harvard.edu/abs/2016arXiv160506171S}
}

@article {Hag99,
    AUTHOR = {H{\"a}ggstr{\"o}m, Olle},
     TITLE = {Positive correlations in the fuzzy {P}otts model},
   JOURNAL = {Ann. Appl. Probab.},
  FJOURNAL = {The Annals of Applied Probability},
    VOLUME = {9},
      YEAR = {1999},
    NUMBER = {4},
     PAGES = {1149--1159},
      ISSN = {1050-5164,2168-8737},
   MRCLASS = {60K35 (82B20 82B43)},
  MRNUMBER = {1728557},
MRREVIEWER = {Christian\ Maes},
       DOI = {10.1214/aoap/1029962867},
       URL = {https://doi.org/10.1214/aoap/1029962867},
}

@article {KW07,
    AUTHOR = {Kahn, Jeff and Weininger, Nicholas},
     TITLE = {Positive association in the fractional fuzzy {P}otts model},
   JOURNAL = {Ann. Probab.},
  FJOURNAL = {The Annals of Probability},
    VOLUME = {35},
      YEAR = {2007},
    NUMBER = {6},
     PAGES = {2038--2043},
      ISSN = {0091-1798,2168-894X},
   MRCLASS = {60K35 (05D40)},
  MRNUMBER = {2353381},
MRREVIEWER = {Andr\'as\ Kr\'amli},
       DOI = {10.1214/009117907000000042},
       URL = {https://doi.org/10.1214/009117907000000042},
}

@article {MV95,
    AUTHOR = {Maes, Christian and Vande Velde, Koen},
     TITLE = {The fuzzy {P}otts model},
   JOURNAL = {J. Phys. A},
  FJOURNAL = {Journal of Physics. A. Mathematical and General},
    VOLUME = {28},
      YEAR = {1995},
    NUMBER = {15},
     PAGES = {4261--4270},
      ISSN = {0305-4470,1751-8121},
   MRCLASS = {82B20 (82B26)},
  MRNUMBER = {1351929},
MRREVIEWER = {Miroslav\ Koles\'ik},
       URL = {http://stacks.iop.org/0305-4470/28/4261},
}

@ARTICLE{NQSZ23,
       author = {{Nolin}, Pierre and {Qian}, Wei and {Sun}, Xin and {Zhuang}, Zijie},
        title = "{Backbone exponent for two-dimensional percolation}",
      journal = {arXiv e-prints},
         year = 2023,
        month = sep,
          eid = {arXiv:2309.05050},
        pages = {arXiv:2309.05050},
          doi = {10.48550/arXiv.2309.05050},
archivePrefix = {arXiv},
       eprint = {2309.05050},
 primaryClass = {math.PR},
       adsurl = {https://ui.adsabs.harvard.edu/abs/2023arXiv230905050N}
}

@article {DKRV16,
    AUTHOR = {David, Fran\c cois and Kupiainen, Antti and Rhodes, R\'emi and
              Vargas, Vincent},
     TITLE = {Liouville quantum gravity on the {R}iemann sphere},
   JOURNAL = {Comm. Math. Phys.},
  FJOURNAL = {Communications in Mathematical Physics},
    VOLUME = {342},
      YEAR = {2016},
    NUMBER = {3},
     PAGES = {869--907},
      ISSN = {0010-3616,1432-0916},
   MRCLASS = {81T20},
  MRNUMBER = {3465434},
       DOI = {10.1007/s00220-016-2572-4},
       URL = {https://doi.org/10.1007/s00220-016-2572-4},
}

@article {BPZ84,
    AUTHOR = {Belavin, A. A. and Polyakov, A. M. and Zamolodchikov, A. B.},
     TITLE = {Infinite conformal symmetry in two-dimensional quantum field
              theory},
   JOURNAL = {Nuclear Phys. B},
  FJOURNAL = {Nuclear Physics. B. Theoretical, Phenomenological, and
              Experimental High Energy Physics. Quantum Field Theory and
              Statistical Systems},
    VOLUME = {241},
      YEAR = {1984},
    NUMBER = {2},
     PAGES = {333--380},
      ISSN = {0550-3213,1873-1562},
   MRCLASS = {81E13 (17B70 58G37 81D15)},
  MRNUMBER = {757857},
       DOI = {10.1016/0550-3213(84)90052-X},
       URL = {https://doi.org/10.1016/0550-3213(84)90052-X},
}

@incollection {ghs-mating-survey,
    AUTHOR = {Gwynne, Ewain and Holden, Nina and Sun, Xin},
     TITLE = {Mating of trees for random planar maps and {L}iouville quantum
              gravity: a survey},
 BOOKTITLE = {Topics in statistical mechanics},
    SERIES = {Panor. Synth\`eses},
    VOLUME = {59},
     PAGES = {41--120},
 PUBLISHER = {Soc. Math. France, Paris},
      YEAR = {2023},
   MRCLASS = {60J67 (82B44 83C45)},
  MRNUMBER = {4619311},
       eprint = {\arxiv{1910.04713}},
}

@article {Pol81,
    AUTHOR = {Polyakov, A. M.},
     TITLE = {Quantum geometry of bosonic strings},
   JOURNAL = {Phys. Lett. B},
  FJOURNAL = {Physics Letters. B. Particle Physics, Nuclear Physics and
              Cosmology},
    VOLUME = {103},
      YEAR = {1981},
    NUMBER = {3},
     PAGES = {207--210},
      ISSN = {0370-2693,1873-2445},
   MRCLASS = {81E99 (58D30 81G05 82A68)},
  MRNUMBER = {623209},
       DOI = {10.1016/0370-2693(81)90743-7},
       URL = {https://doi.org/10.1016/0370-2693(81)90743-7},
}

@article {LG13,
    AUTHOR = {Le Gall, Jean-Fran\c cois},
     TITLE = {Uniqueness and universality of the {B}rownian map},
   JOURNAL = {Ann. Probab.},
  FJOURNAL = {The Annals of Probability},
    VOLUME = {41},
      YEAR = {2013},
    NUMBER = {4},
     PAGES = {2880--2960},
      ISSN = {0091-1798,2168-894X},
   MRCLASS = {60D05 (05C80 60F17)},
  MRNUMBER = {3112934},
MRREVIEWER = {David\ J.\ Aldous},
       DOI = {10.1214/12-AOP792},
       URL = {https://doi.org/10.1214/12-AOP792},
}

@article {BM17,
    AUTHOR = {Bettinelli, J\'er\'emie and Miermont, Gr\'egory},
     TITLE = {Compact {B}rownian surfaces {I}: {B}rownian disks},
   JOURNAL = {Probab. Theory Related Fields},
  FJOURNAL = {Probability Theory and Related Fields},
    VOLUME = {167},
      YEAR = {2017},
    NUMBER = {3-4},
     PAGES = {555--614},
      ISSN = {0178-8051,1432-2064},
   MRCLASS = {60F17 (60B05 60D05)},
  MRNUMBER = {3627425},
MRREVIEWER = {Wolfgang\ L\"ohr},
       DOI = {10.1007/s00440-016-0752-y},
       URL = {https://doi.org/10.1007/s00440-016-0752-y},
}

@article {HS19,
    AUTHOR = {Holden, Nina and Sun, Xin},
     TITLE = {Convergence of uniform triangulations under the {C}ardy
              embedding},
   JOURNAL = {Acta Math.},
  FJOURNAL = {Acta Mathematica},
    VOLUME = {230},
      YEAR = {2023},
    NUMBER = {1},
     PAGES = {93--203},
      ISSN = {0001-5962,1871-2509},
   MRCLASS = {60D05 (57Q15)},
  MRNUMBER = {4567714},
}

@article {GM21,
    AUTHOR = {Gwynne, Ewain and Miller, Jason},
     TITLE = {Percolation on uniform quadrangulations and {$\rm SLE_6$} on
              {$\sqrt{8/3}$}-{L}iouville quantum gravity},
   JOURNAL = {Ast\'erisque},
  FJOURNAL = {Ast\'erisque},
    NUMBER = {429},
      YEAR = {2021},
     PAGES = {vii+242},
      ISSN = {0303-1179,2492-5926},
      ISBN = {978-2-85629-947-0},
   MRCLASS = {60K35 (60F17 60G57 60J67 83C45)},
  MRNUMBER = {4390048},
       DOI = {10.24033/ast},
       URL = {https://doi.org/10.24033/ast},
}

@article {AHS17,
    AUTHOR = {Aru, Juhan and Huang, Yichao and Sun, Xin},
     TITLE = {Two perspectives of the 2{D} unit area quantum sphere and
              their equivalence},
   JOURNAL = {Comm. Math. Phys.},
  FJOURNAL = {Communications in Mathematical Physics},
    VOLUME = {356},
      YEAR = {2017},
    NUMBER = {1},
     PAGES = {261--283},
      ISSN = {0010-3616,1432-0916},
   MRCLASS = {83C45 (60G57 81T20 81T40 83C80)},
  MRNUMBER = {3694028},
MRREVIEWER = {Yu.\ N.\ Obukhov},
       DOI = {10.1007/s00220-017-2979-6},
       URL = {https://doi.org/10.1007/s00220-017-2979-6},
}

@article {Cer21,
    AUTHOR = {Cercl\'e, Baptiste},
     TITLE = {Unit boundary length quantum disk: a study of two different
              perspectives and their equivalence},
   JOURNAL = {ESAIM Probab. Stat.},
  FJOURNAL = {ESAIM. Probability and Statistics},
    VOLUME = {25},
      YEAR = {2021},
     PAGES = {433--459},
      ISSN = {1292-8100,1262-3318},
   MRCLASS = {60D05 (81T20 81T40 83C45)},
  MRNUMBER = {4338790},
       DOI = {10.1051/ps/2021016},
       URL = {https://doi.org/10.1051/ps/2021016},
}

@ARTICLE{GKRV-sphere,
	author = {{Guillarmou}, Colin and {Kupiainen}, Antti and {Rhodes}, R{\'e}mi and {Vargas}, Vincent},
	title = "{Conformal bootstrap in Liouville Theory}",
	journal = {ArXiv e-prints},
	year = 2020,
	month = may, 
	archivePrefix = {arXiv},
	eprint = {\arxiv{2005.11530}},
	primaryClass = {math.PR},
	adsurl = {https://ui.adsabs.harvard.edu/abs/2020arXiv200511530G}
}

@article {krv-dozz,
    AUTHOR = {Kupiainen, Antti and Rhodes, R\'emi and Vargas, Vincent},
     TITLE = {Integrability of {L}iouville theory: proof of the {DOZZ}
              formula},
   JOURNAL = {Ann. of Math. (2)},
  FJOURNAL = {Annals of Mathematics. Second Series},
    VOLUME = {191},
      YEAR = {2020},
    NUMBER = {1},
     PAGES = {81--166},
      ISSN = {0003-486X,1939-8980},
   MRCLASS = {81T40 (60D99)},
  MRNUMBER = {4060417},
MRREVIEWER = {Nizar\ Demni},
       DOI = {10.4007/annals.2020.191.1.2},
       URL = {https://doi.org/10.4007/annals.2020.191.1.2},
}

@ARTICLE{GKR-review,
       author = {{Guillarmou}, Colin and {Kupiainen}, Antti and {Rhodes}, R{\'e}mi},
        title = "{Review on the probabilistic construction and Conformal bootstrap in Liouville Theory}",
      journal = {arXiv e-prints},
         year = 2024,
        month = mar,
          eid = {arXiv:2403.12780},
        pages = {arXiv:2403.12780},
          doi = {10.48550/arXiv.2403.12780},
archivePrefix = {arXiv},
       eprint = {2403.12780},
 primaryClass = {math-ph},
       adsurl = {https://ui.adsabs.harvard.edu/abs/2024arXiv240312780G}
}

@article {ASW19,
    AUTHOR = {Aru, Juhan and Sep\'ulveda, Avelio and Werner, Wendelin},
     TITLE = {On bounded-type thin local sets of the two-dimensional
              {G}aussian free field},
   JOURNAL = {J. Inst. Math. Jussieu},
  FJOURNAL = {Journal of the Institute of Mathematics of Jussieu. JIMJ.
              Journal de l'Institut de Math\'ematiques de Jussieu},
    VOLUME = {18},
      YEAR = {2019},
    NUMBER = {3},
     PAGES = {591--618},
      ISSN = {1474-7480,1475-3030},
   MRCLASS = {60G60 (60J67)},
  MRNUMBER = {3936643},
MRREVIEWER = {Xin\ Sun},
       DOI = {10.1017/s1474748017000160},
       URL = {https://doi.org/10.1017/s1474748017000160},
}

@book {BS02,
    AUTHOR = {Borodin, Andrei N. and Salminen, Paavo},
     TITLE = {Handbook of {B}rownian motion---facts and formulae},
    SERIES = {Probability and its Applications},
   EDITION = {Second},
 PUBLISHER = {Birkh\"auser Verlag, Basel},
      YEAR = {2002},
     PAGES = {xvi+672},
      ISBN = {3-7643-6705-9},
   MRCLASS = {60-00 (60H05 60J25 60J55 60J60 60J65)},
  MRNUMBER = {1912205},
MRREVIEWER = {S\'andor\ Cs\"org\H o},
       DOI = {10.1007/978-3-0348-8163-0},
       URL = {https://doi.org/10.1007/978-3-0348-8163-0},
}

@article {MSW14,
    AUTHOR = {Miller, Jason and Sun, Nike and Wilson, David B.},
     TITLE = {The {H}ausdorff dimension of the {CLE} gasket},
   JOURNAL = {Ann. Probab.},
  FJOURNAL = {The Annals of Probability},
    VOLUME = {42},
      YEAR = {2014},
    NUMBER = {4},
     PAGES = {1644--1665},
      ISSN = {0091-1798,2168-894X},
   MRCLASS = {60J67 (60D05)},
  MRNUMBER = {3262488},
MRREVIEWER = {Dmitry\ Beliaev},
       DOI = {10.1214/12-AOP820},
       URL = {https://doi.org/10.1214/12-AOP820},
}

@article {NW11,
    AUTHOR = {Nacu, \c Serban and Werner, Wendelin},
     TITLE = {Random soups, carpets and fractal dimensions},
   JOURNAL = {J. Lond. Math. Soc. (2)},
  FJOURNAL = {Journal of the London Mathematical Society. Second Series},
    VOLUME = {83},
      YEAR = {2011},
    NUMBER = {3},
     PAGES = {789--809},
      ISSN = {0024-6107,1469-7750},
   MRCLASS = {28A80 (28A78 60G55 60J67 60K35)},
  MRNUMBER = {2802511},
MRREVIEWER = {Lars\ Olsen},
       DOI = {10.1112/jlms/jdq094},
       URL = {https://doi.org/10.1112/jlms/jdq094},
}

@article {Leh23,
    AUTHOR = {Lehmkuehler, Matthis},
     TITLE = {The trunks of {${\rm CLE}(4)$} explorations},
   JOURNAL = {Ann. Appl. Probab.},
  FJOURNAL = {The Annals of Applied Probability},
    VOLUME = {33},
      YEAR = {2023},
    NUMBER = {5},
     PAGES = {3387--3417},
      ISSN = {1050-5164,2168-8737},
   MRCLASS = {60J67 (60G51 60H10 60J55)},
  MRNUMBER = {4663486},
MRREVIEWER = {Yizheng\ Yuan},
       DOI = {10.1214/22-aap1895},
       URL = {https://doi.org/10.1214/22-aap1895},
}

@article {Dub09,
    AUTHOR = {Dub\'edat, Julien},
     TITLE = {S{LE} and the free field: partition functions and couplings},
   JOURNAL = {J. Amer. Math. Soc.},
  FJOURNAL = {Journal of the American Mathematical Society},
    VOLUME = {22},
      YEAR = {2009},
    NUMBER = {4},
     PAGES = {995--1054},
      ISSN = {0894-0347,1088-6834},
   MRCLASS = {60J67 (60G17 60K35)},
  MRNUMBER = {2525778},
MRREVIEWER = {Dmitry\ Beliaev},
       DOI = {10.1090/S0894-0347-09-00636-5},
       URL = {https://doi.org/10.1090/S0894-0347-09-00636-5},
}

@article {SS09,
    AUTHOR = {Schramm, Oded and Sheffield, Scott},
     TITLE = {Contour lines of the two-dimensional discrete {G}aussian free
              field},
   JOURNAL = {Acta Math.},
  FJOURNAL = {Acta Mathematica},
    VOLUME = {202},
      YEAR = {2009},
    NUMBER = {1},
     PAGES = {21--137},
      ISSN = {0001-5962,1871-2509},
   MRCLASS = {60J67 (60D05 60G17 60K35)},
  MRNUMBER = {2486487},
MRREVIEWER = {Julien\ Dub\'edat},
       DOI = {10.1007/s11511-009-0034-y},
       URL = {https://doi.org/10.1007/s11511-009-0034-y},
}

@article {SS13,
    AUTHOR = {Schramm, Oded and Sheffield, Scott},
     TITLE = {A contour line of the continuum {G}aussian free field},
   JOURNAL = {Probab. Theory Related Fields},
  FJOURNAL = {Probability Theory and Related Fields},
    VOLUME = {157},
      YEAR = {2013},
    NUMBER = {1-2},
     PAGES = {47--80},
      ISSN = {0178-8051,1432-2064},
   MRCLASS = {60J67 (60G15)},
  MRNUMBER = {3101840},
MRREVIEWER = {Fredrik\ Johansson Viklund},
       DOI = {10.1007/s00440-012-0449-9},
       URL = {https://doi.org/10.1007/s00440-012-0449-9},
}

@article {Edwards-Sokal,
    AUTHOR = {Edwards, Robert G. and Sokal, Alan D.},
     TITLE = {Generalization of the {F}ortuin-{K}asteleyn-{S}wendsen-{W}ang
              representation and {M}onte {C}arlo algorithm},
   JOURNAL = {Phys. Rev. D (3)},
  FJOURNAL = {Physical Review. D. Particles and Fields. Third Series},
    VOLUME = {38},
      YEAR = {1988},
    NUMBER = {6},
     PAGES = {2009--2012},
      ISSN = {0556-2821},
   MRCLASS = {82-04 (81-08 81E25 82-08 82A68)},
  MRNUMBER = {965465},
       DOI = {10.1103/PhysRevD.38.2009},
       URL = {https://doi.org/10.1103/PhysRevD.38.2009},
}

@article{DJS10,
  title={Bulk and boundary critical behaviour of thin and thick domain walls in the two-dimensional {P}otts model},
  author={Dubail, J{\'e}r{\^o}me and Jacobsen, Jesper Lykke and Saleur, Hubert},
  journal={Journal of Statistical Mechanics: Theory and Experiment},
  volume={2010},
  number={12},
  pages={P12026},
  year={2010},
  publisher={IOP Publishing}
}

@article{AOS98,
  title={Boundary critical phenomena in the three-state {P}otts model},
  author={Affleck, Ian and Oshikawa, Masaki and Saleur, Hubert},
  journal={Journal of Physics A: Mathematical and General},
  volume={31},
  number={28},
  pages={5827},
  year={1998},
  publisher={IOP Publishing}
}

@ARTICLE{FPS20,
       author = {{Fukusumi}, Yoshiki and {Picco}, Marco and {Santachiara}, Raoul},
        title = "{Spin interfaces and crossing probabilities of spin clusters in parafermionic models}",
      journal = {arXiv e-prints},
         year = 2020,
        month = jun,
          eid = {arXiv:2006.09850},
        pages = {arXiv:2006.09850},
          doi = {10.48550/arXiv.2006.09850},
archivePrefix = {arXiv},
       eprint = {2006.09850},
 primaryClass = {hep-th},
       adsurl = {https://ui.adsabs.harvard.edu/abs/2020arXiv200609850F}
}

@book {CFT-book,
    AUTHOR = {Di Francesco, Philippe and Mathieu, Pierre and S\'en\'echal,
              David},
     TITLE = {Conformal field theory},
    SERIES = {Graduate Texts in Contemporary Physics},
 PUBLISHER = {Springer-Verlag, New York},
      YEAR = {1997},
     PAGES = {xxii+890},
      ISBN = {0-387-94785-X},
   MRCLASS = {81T40 (81-02)},
  MRNUMBER = {1424041},
MRREVIEWER = {Christoph\ Schweigert},
       DOI = {10.1007/978-1-4612-2256-9},
       URL = {https://doi.org/10.1007/978-1-4612-2256-9},
}

@article {DPSV13,
    AUTHOR = {Delfino, G. and Picco, M. and Santachiara, R. and Viti, J.},
     TITLE = {Spin clusters and conformal field theory},
   JOURNAL = {J. Stat. Mech. Theory Exp.},
  FJOURNAL = {Journal of Statistical Mechanics: Theory and Experiment},
      YEAR = {2013},
    NUMBER = {11},
     PAGES = {P11011, 15},
      ISSN = {1742-5468},
   MRCLASS = {82B80},
  MRNUMBER = {3152308},
       DOI = {10.1088/1742-5468/2013/11/p11011},
       URL = {https://doi.org/10.1088/1742-5468/2013/11/p11011},
}

@ARTICLE{SXZ24,
       author = {{Sun}, Xin and {Xu}, Shengjing and {Zhuang}, Zijie},
        title = "{Annulus crossing formulae for critical planar percolation}",
      journal = {arXiv e-prints},
         year = 2024,
        month = oct,
          eid = {arXiv:2410.04767},
        pages = {arXiv:2410.04767},
archivePrefix = {arXiv},
       eprint = {2410.04767},
 primaryClass = {math.PR},
       adsurl = {https://ui.adsabs.harvard.edu/abs/2024arXiv241004767S}
}

@article {BH19,
    AUTHOR = {Benoist, St\'ephane and Hongler, Cl\'ement},
     TITLE = {The scaling limit of critical {I}sing interfaces is
              {$\mathrm{CLE}_3$}},
   JOURNAL = {Ann. Probab.},
  FJOURNAL = {The Annals of Probability},
    VOLUME = {47},
      YEAR = {2019},
    NUMBER = {4},
     PAGES = {2049--2086},
      ISSN = {0091-1798,2168-894X},
   MRCLASS = {60J67 (60K35 82B20 82B27)},
  MRNUMBER = {3980915},
MRREVIEWER = {Jianping\ Jiang},
       DOI = {10.1214/18-AOP1301},
       URL = {https://doi.org/10.1214/18-AOP1301},
}

@article {BDC12,
    AUTHOR = {Beffara, Vincent and Duminil-Copin, Hugo},
     TITLE = {The self-dual point of the two-dimensional random-cluster
              model is critical for {$q\geq 1$}},
   JOURNAL = {Probab. Theory Related Fields},
  FJOURNAL = {Probability Theory and Related Fields},
    VOLUME = {153},
      YEAR = {2012},
    NUMBER = {3-4},
     PAGES = {511--542},
      ISSN = {0178-8051,1432-2064},
   MRCLASS = {60K35 (82B20)},
  MRNUMBER = {2948685},
MRREVIEWER = {Enza\ Orlandi},
       DOI = {10.1007/s00440-011-0353-8},
       URL = {https://doi.org/10.1007/s00440-011-0353-8},
}

\end{document}